\documentclass[10pt,a4paper,reqno]{article}
\usepackage{graphicx} 
\usepackage[english]{babel} 
\usepackage{csquotes} 
\usepackage{filecontents} 
\usepackage[dvipsnames]{xcolor}
\usepackage{color}
\usepackage[colorinlistoftodos]{todonotes}
\usepackage{tensor}
\usepackage{thmtools}
\usepackage{microtype}

\usepackage[
backend=biber,
hyperref=true,
backref=true,
isbn=false,
doi=true,
url=false,
natbib=true,
eprint=true,
useprefix=true,
maxcitenames=99,
maxbibnames=99,  
maxalphanames=99, 
minalphanames=99,
safeinputenc,
style=alphabetic,
citestyle=alphabetic,
block=space,
datamodel=ext-eprint,
]{biblatex}
\usepackage[
hypertexnames = false,
colorlinks    = true,
citecolor     = teal,
linkcolor     = blue,
urlcolor      = blue,
linktocpage = true,
breaklinks
]{hyperref}
\renewbibmacro{in:}{} 

\DeclareSourcemap{
	\maps[datatype=bibtex]{
		\map{
			\step[fieldsource=pmid, fieldtarget=pubmed]
		}
	}
}

\makeatletter
\DeclareFieldFormat{arxiv}{%
	arXiv\addcolon\space
	\ifhyperref
	{\href{http://arxiv.org/\abx@arxivpath/#1}{%
			\nolinkurl{#1}%
			\iffieldundef{arxivclass}
			{}
			{\addspace\texttt{\mkbibbrackets{\thefield{arxivclass}}}}}}
	{\nolinkurl{#1}
		\iffieldundef{arxivclass}
		{}
		{\addspace\texttt{\mkbibbrackets{\thefield{arxivclass}}}}}}
\makeatother
\DeclareFieldFormat{pmcid}{%
	PMCID\addcolon\space
	\ifhyperref
	{\href{http://www.ncbi.nlm.nih.gov/pmc/articles/#1}{\nolinkurl{#1}}}
	{\nolinkurl{#1}}}
\DeclareFieldFormat{mrnumber}{%
	MR\addcolon\space
	\ifhyperref
	{\href{http://www.ams.org/mathscinet-getitem?mr=MR#1}{\nolinkurl{#1}}}
	{\nolinkurl{#1}}}
\DeclareFieldFormat{zbl}{%
	Zbl\addcolon\space
	\ifhyperref
	{\href{http://zbmath.org/?q=an:#1}{\nolinkurl{#1}}}
	{\nolinkurl{#1}}}
\DeclareFieldAlias{jstor}{eprint:jstor}
\DeclareFieldAlias{hdl}{eprint:hdl}
\DeclareFieldAlias{pubmed}{eprint:pubmed}
\DeclareFieldAlias{googlebooks}{eprint:googlebooks}

\renewbibmacro*{eprint}{%
	\printfield{arxiv}%
	\newunit\newblock
	\printfield{jstor}%
	\newunit\newblock
	\printfield{mrnumber}%
	\newunit\newblock
	\printfield{zbl}%
	\newunit\newblock
	\printfield{hdl}%
	\newunit\newblock
	\printfield{pubmed}%
	\newunit\newblock
	\printfield{pmcid}%
	\newunit\newblock
	\printfield{googlebooks}%
	\newunit\newblock
	\iffieldundef{eprinttype}
	{\printfield{eprint}}
	{\printfield[eprint:\strfield{eprinttype}]{eprint}}}

\usepackage[margin=1.9cm,includehead]{geometry}

\usepackage{seqsplit}
\usepackage{xstring}

\usepackage{amsmath,stmaryrd, upgreek}
\usepackage{amsthm,amssymb}
\usepackage{latexsym}
\usepackage{amscd}
\usepackage{mathrsfs}
\usepackage{url}
\usepackage{mathtools}
\usepackage{cleveref}
\usepackage{doi}

\usepackage[T1]{fontenc}
\usepackage{tikz-cd}
\usepackage{titlesec}

\usepackage[titletoc,toc,title]{appendix}

\makeatletter 
\renewcommand\tableofcontents{%
	\@starttoc{toc}%
	
}
\makeatother

\usepackage{enumerate}
\usepackage{slashed}
\graphicspath{}
\usepackage{fancyhdr} 

\usepackage{changepage}

\newcounter{noteCounter}
\newcommand\shorttitle{Ricci-harmonic flow of $\G2$ and Spin(7)-structures} 
\newcommand\authors{S. Dwivedi} 

\newcounter{commentCounter}
\titleformat{\section}    
{\normalfont\large\bfseries\center}{\thesection.}{1em}{}

\makeatletter
\newcommand*{\rom}[1]{\expandafter\@slowromancap\romannumeral #1@}
\makeatother

\newtheorem{theorem}{Theorem}[section]
\newtheorem{corollary}[theorem]{Corollary}
\newtheorem{lemma}[theorem]{Lemma}
\newtheorem{proposition}[theorem]{Proposition}

\newtheorem{question}[theorem]{Question}
\newtheorem{claim}[theorem]{Claim}

\theoremstyle{definition}
\newtheorem{definition}[theorem]{Definition}
\newtheorem{remark}[theorem]{Remark}
\newtheorem*{ack}{Acknowledgments}
\numberwithin{equation}{section}
\def\bR{\mathbb R}

\def\bO{\mathbb O}

\def\pt{\partial}
\def\del{\nabla}
\def\G2{\mathrm{G}_2}
\def\g2{\varphi}
\def\S7{\mathrm{Spin}(7)}
\def\s7{\Phi}
\def\ddt{\frac{d}{dt}}
\def\ptt{\frac{\partial}{\partial t}}
\def\sgn{\mathrm{sgn}}

\def\mhatl{\widehat{m}_l}
\def\ml{m_l}

\def\cF{\mathcal{F}}

\def\cL{\mathcal{L}}

\def\cT{\mathcal{T}}

\newcommand{\Vop}{\mathsf V}

\def\Spin7{\mathrm{Spin(7)}}
\def\Ric{\mathrm{Ric}}
\def\Riem{\mathrm{Rm}}
\def\Rm{\mathrm{Rm}}
\def\SO{\mathrm{SO}}

\def\dots7{\Dot{\Phi}}
\def\inj{\mathrm{inj}}

\DeclareMathOperator\Div{div}

\DeclareMathOperator\curl{curl}

\DeclareMathOperator\vol{vol}
\DeclareMathOperator\Vol{Vol}
\DeclareMathOperator\tr{tr}

\newcommand\xqed[1]{%
	\leavevmode\unskip\penalty9999 \hbox{}\nobreak\hfill
	\quad\hbox{#1}}
\newcommand\demo{\xqed{$\blacktriangle$}}

\usepackage{cancel}
\usepackage[all]{xy}
\usepackage{multicol}

\usepackage{charter}

\begin{document}

	\title{Ricci-harmonic flow of $\mathrm{G_2}$ and Spin(7)-structures}
	\author{Shubham Dwivedi}
	\date{}
	
	\maketitle

	\begin{abstract}
	We introduce and study a new general flow  of $\G2$-structures which we call the Ricci-harmonic flow of $\G2$-structures. The flow is the coupling of the Ricci flow of underlying metrics and the isometric flow of $\G2$-structures, but we also provide explicit lower order in the torsion terms. The lower order terms and the flow are obtained by analyzing the second order term in the Taylor series expansion of $\G2$-structures in normal coordinates. As such, the Ricci-harmonic flow described in the paper can be interpreted as the "heat equation" for $\G2$-structures. The lower order terms allow us to prove that the stationary points of the Ricci-harmonic flow are exactly torsion-free $\G2$-structures on compact manifolds. We study various analytic and geometric properties of the flow. We show that the flow has short-time existence and uniqueness on compact manifolds starting with an arbitrary $\G2$-structure and prove global Shi-type estimates. We also prove a modified local Shi-type estimates for the flow which assume bounds on the initial derivatives of the Riemann curvature tensor and the torsion but give uniform bounds on these quantities for all times. We prove a compactness theorem for the solutions of the flow and use it to prove that the Ricci-harmonic flow exists as long as the velocity of the flow remains bounded. We also study Ricci-harmonic solitons where we prove that there are no compact expanding solitons and the only steady solitons are torsion-free. We derive an analog of Hamilton's identity for gradient Ricci-harmonic solitons and prove some integral identities for the solitons. Finally, we prove a version of the Taylor series expansion for Spin(7)-structures and use it to derive the Ricci-harmonic flow of Spin(7)-structures.  
	\end{abstract}
	
	\begin{adjustwidth}{0.95cm}{0.95cm}
	 \tableofcontents
	\end{adjustwidth}
	
	\let\thefootnote\relax\footnotetext{\emph{MSC (2020): 53E99, 53C29, 53C21, 53C15.}}

	\section{Introduction}\label{sec:intro}

The search for ``best'' geometric structures on a manifold gives rise to interesting problems which are both geometric and analytic in nature. Over the past few decades, use of geometric flows and related theory of parabolic partial differential equations has been very successful in providing solutions to longstanding problems as well as generating new powerful machinery and theories. Perhaps, the most well-known intrinsic flow of geometric structures is the Ricci flow of Riemannian metrics which was introduced by Richard Hamilton in \cite{hamilton-3manifolds} as a means to solve the Poincaré conjecture. Hamilton developed new tools to study the Ricci flow and proved the Poincaré conjecture for $3$-manifolds which admit metrics of positive Ricci curvature. Hamilton also established a roadmap for proving the conjecture in full generality which was finally achieved in the breakthroughs of Perelman \cite{perelman-1}. Hamilton's ideas and philosophies in proposing the Ricci flow and the reults therein are a major motivation for us in this paper.

\medskip

The geometric structures we are interested in this paper are $\G2$-structures on seven dimensional manifolds and Spin(7)-structures on eight dimensional manifolds. Both these structures have origins in the theory of the non-associative real normed division algebra of the octonions $\mathbb{O}$. A $\G2$-structure on a $7$-manifold $M$ is the reduction of the frame bundle $Fr(M)$ from the group $\mathrm{GL}(7, \bR)$ to the Lie group $\G2$ which can be described as the automorphism group of $\mathbb{O}$. A $\G2$-structure on $M$ induces a Riemannian metric and an orientation. From the point of view of differential geometry, a $\G2$-structure is a \emph{positive} (or nondegenerate) $3$-form which we denote by $\g2$. The subclass of torsion-free $\G2$-structures are those $\g2$s which are parallel with respect to the Levi-Civita connection of the induced metric. The \emph{holonomy} of the metric of a torsion-free $\G2$-structure is contained in the exceptional Lie group $\G2$ which is one of the groups that appears on the Berger's list of possible Riemannian holonomy groups. Such metrics are also Ricci-flat which make their study very valuable and interesting. The existence of torsion-free $\G2$-structures on a manifold is a challenging problem as one needs to solve a highly nonlinear PDE. Given the success of the Ricci flow of metrics,  one hopes that a suitable flow of $\G2$-structures might help in proving the existence of torsion-free $\G2$-structures. Ideally, one would like to start with an arbitrary $\G2$-structure and hope that the flow (or some normalization of it) will exist for all time and converge to a torsion-free $\G2$-structure as long as the manifold satisfies required topological conditions. 

\medskip

There has been a lot of work in that direction with various different flow of $\G2$-structures each with its own motivation and applicability. The first such flow was the Laplacian flow of closed $\G2$-structures proposed by Bryant \cite{bryant-remarks}. Various foundation results for the flow have been established by Bryant--Xu \cite{bryant-xu} and Lotay--Wei \cite{lotay-wei-gafa, lotay-wei-cag, lotay-wei-jdg}. Inspired by the Laplacian flow, the Laplacian co-flow of co-closed $\G2$-structures was defined by Karigiannis--McKay--Tsui \cite{kmt} which is not yet known to have a short-time existence result, an issue which was modified by Grigorian who proposed the modified Laplacian co-flow of co-closed $\G2$-structures which is well-posed on compact manifolds \cite{grigorian-modified}. The theory of general flows of $\G2$-structures, i.e., flows which do not impose any \emph{a priori} conditions on the evolving $\G2$-structures was initiated by Karigiannis \cite{flows1}. The study was continued in \cite{dgk-flows} where the authors classified \emph{all} possible second order differential invariants of a $\G2$-structure which can be evolved as a $3$-form. As a result of the general theory, the authors proved a very general short-time existence and uniqueness for solutions of vast family of flows of $\G2$-structures. A general flow of $\G2$-structures which is the negative gradient flow of the natural energy functional $\g2 \mapsto$ constant$\cdot \int_M ||\del \g2||^2 \vol$, was proposed by Weiss--Witt \cite{weiss-witt} where the authors proved short-time existence and uniqueness of solutions as well as stability of torsion-free $\G2$-structures along the flow.
	
	\medskip

In this paper, we propose a new general flow of $\G2$-structures which is inspired by Hamilton's methodology for proposing the Ricci flow. In order to explain the flow, we briefly describe the setup for flows of $\G2$-structures. More details can be found in \textsection~\ref{sec:prelims}. Any flow of a family $\g2(t)$ of $\G2$-structures  is described by the evolution equation
\begin{align}\label{eq:genf1}
\ptt \g2(t)=h(t)\diamond_t\g2(t) + X(t)\lrcorner \psi(t)
\end{align}
where $h(t)$ is a family of symmetric $2$-tensors on $M$, $X(t)$ is a family of vector fields on $M$, both of which are usually second order operators in $\g2(t)$ and $\diamond_t$ is a map which takes symmetric $2$-tensors on $M$ to $3$-forms lying in $\Omega_{1\oplus 27}$ component (see \textsection~\ref{sec:prelims} for the definition of the $\diamond$ operator). Here $\psi(t)=*_t\g2(t)$. Under the flow \eqref{eq:genf1}, the underlying metric evolves by $\ptt g(t)=2h(t)$, see \cite{flows1}.

\medskip 

If  a family $\g2(t)$ of $\G2$-structures is evolving by $\ptt \g2(t)=-\Ric(t) \diamond_t \g2(t)$ then the underlying metrics evolve by $\ptt g=-2\Ric$ which is the Ricci flow of metrics. We call the former flow as \emph{Ricci flow of} $\G2$-\emph{structures}. Given the enormous success of the Ricci flow in solving outstanding problems in geometry and topology it seems reasonable to study flows of $\G2$-structures whose underlying metric are evolving by the Ricci flow. If, however, we simply evolve $\g2(t)$ by the Ricci flow of $\G2$-structures then the flow is neglecting possible evolutions and degenerations of the $\G2$-structures governed by the vector fields which correspond to the deformations of $\g2(t)$ in the $\Omega^3_7$ direction (see \textsection~\ref{sec:prelims} for the decomposition of $3$-forms). 	
	
	\medskip
	
A proposal for studying the ``best'' flow of $\G2$-structures with deformations only in the $\Omega^3_7$	direction was put forward in \cite{dgk-isometric}, \cite{grigorian-isometric} and \cite{loubeau-saearp}. We note that the deformations in $\Omega^3_7$ direction do not change the underlying metric. The \emph{isometric} or \emph{harmonic} flow of $\G2$-structures defined in the above works is the flow of $\G2$-structures on a compact $(M^7, \g2_0)$ given by 
\begin{align}\label{eq:isometricflow}
	\frac{\pt \g2}{\pt t} = \Div T\lrcorner \psi, \ \ \ \g2(0)=\g2_0,
\end{align}
where $T$ is the torsion of the $\G2$-structure and is a $2$-tensor and $\Div T$ is the vector field defined by $(\Div T)_{i}=\del_jT_{ji}$. The harmonic flow \eqref{eq:isometricflow} is the negative gradient of the natural energy functional
\begin{align*}
	E(\g2)=\frac 12\int_M |T_{\g2}|^2_{g_{\g2}} \vol_{\g2},
\end{align*}
when restricted to the \emph{isometry class} $[[\g2_0]]$, i.e., to the space of those $\G2$-structures all of which induce the same Riemannian metric as $\g2_0$ (see \cite[\textsection 2]{dgk-isometric} for more details). In this way, the harmonic (or isometric) flow finds the ``best'' $\G2$-structure in the isometry class of the initial $\G2$-structure. 

\medskip

Since the Ricci flow is the best way to deform a family of metrics and the harmonic/isometric flow of $\G2$-structures is the most optimum way of deforming a $\G2$-structure in an isometry class, it is reasonable to attempt to define a flow of $\G2$-structures which is a ``coupling'' of both the flows. This suggests the following flow of $\G2$-structures 
\begin{align*}
	\frac{\partial \g2}{\partial t} = -\Ric \diamond \g2 + \Div T\lrcorner \psi.
\end{align*}

While the above flow of $\G2$-structures is well-posed on compact manifolds (see Theorem~\ref{thm:RHFste}), it doesn't take into account any lower-order terms in the $\G2$-structure, of which there can be possibly many (see \cite[\textsection 5]{dgk-flows} for some of the possible lower-order terms which could appear in a flow of $\G2$-structures). These lower order terms can change the analytic and geometric behaviour of the flow. So, on the one hand, it would be desirable to have lower order terms in the definition of the geometric flow which couples the Ricci flow of the metric and the harmonic flow of $\G2$-structures up to the highest order, on the other hand, we would like those terms to appear ``naturally'' or in a "geometric" setting. In fact, this was a question asked in \cite{dgk-flows} that can we choose suitable lower order terms in the coupling of the Ricci flow and the harmonic flow which makes the evolution equations of torsion and the curvature "nicer". 

\medskip

We remedy this situation by proposing a general flow of $\G2$-structures, where general refers to the fact that we put no extra conditions on the family of $\G2$-structures, which up to the highest order is the coupled Ricci flow and the harmonic flow of $\G2$-structures but with explicit lower order terms which are geometrically motivated.  We propose the following flow.

\begin{definition}\label{heatflowdefn}
	Let $(M^7, \g2_0)$ be a compact manifold with a $\G2$-structure $\g2_0$. The \emph{Ricci-harmonic flow} for the family of $\G2$-structures $\g2(t)$ is the following initial value problem
	\begin{align} 
		\label{ricci-harmonic flow} 
		\left\{\begin{array}{rl} 
			& \dfrac{\pt \g2}{\pt t} = \left(-\Ric+3 T^tT -|T|^2g \right) \diamond \g2 + \Div T\lrcorner \psi, \\
			& \g2(0) =\g2_0.
		\end{array}\right.
	\end{align}
Here $T^t$ denotes the transpose of the torsion $2$-tensor and $\Div T$ is the divergence of the torsion.
\end{definition}
The choice of explicit lower order terms in \eqref{ricci-harmonic flow} is motivated in \textsection~\ref{sec:motivation} where these terms are shown to be contractions of the second order terms in the Taylor series expansion of a $\G2$-structure in $\G2$-adapted normal coordinates. As mentioned before, this is based on Hamilton's philosophy of proposing the Ricci flow as the flow of metrics along the term which is obtained by contracting the second order terms in the Taylor series expansion of the metric in normal coordinates. In this way, we can view the Ricci-harmonic flow \eqref{ricci-harmonic flow} as the heat equation for $\G2$-structures and an analog of the Ricci flow of metrics. We explain this viewpoint in detail in \textsection~\ref{sec:motivation}.

\medskip

We prove in \textsection~\ref{sec:motivation} that the stationary points of the Ricci-harmonic flow on compact manifolds are exactly torsion-free $\G2$-structures which also reinforces the natural appearance of our lower order terms. We prove the short-time existence and uniqueness of solutions of the flow in Theorem~\ref{thm:RHFste} as an application of the general theorem proved by the author with P. Gianniotis and S. Karigiannis in \cite{dgk-flows}. As further evidence for \eqref{ricci-harmonic flow} being an analog for the Ricci flow of metrics, we use the ideas in \cite{panos-georgeG2Hilbert} to show that the Ricci-harmonic flow cannot be the gradient flow of any diffeomorphism invariant functional of $\G2$-structures which is similar to the Ricci flow of metrics. Similarly, we prove in Proposition~\ref{prop:RHFonNG2} that \emph{nearly} $\G2$-structures, whose underlying metrics are positive Einstein, are \emph{shrinking solutions} to the flow which is precisely analogous to the behaviour of positive Einstein metrics along the Ricci flow. 

\medskip
We develop foundational theory for the Ricci-harmonic flow in \textsection~\ref{sec:evoleqns} and \ref{sec:ltecompact}. We prove both global Shi-type estimates in Theorem~\ref{thm:shiest} and local Shi-type estimates (which we remark, also follow from the work of Gao Chen \cite{gaochen-shi}). We then prove a modified local Shi-type estimates for the Ricci-harmonic flow in Theorem~\ref{thm:strlocder} which has a stronger assumptions of bounded derivatives of the Riemann curvature tensor and the torsion of the \emph{initial} $\G2$-structure but also gives stronger bounds on all order derivatives of the curvature and tensor at all later times. A consequence of these stronger estimates is a uniform bound on derivatives of curvature and the torsion tensor in geodesic balls along the Ricci-harmonic flow even close to the initial time. These modified local Shi-type estimates are based on similar ideas of Peng Lu for the Ricci flow and this is the first instance of such estimates being proved for flows of $\G2$-structures. Such estimates are used, for instance, in the proof of existence of standard solutions of the Ricci flow on $\bR^3$ and also in Ricci flow with surgery (see \cite{morgan-tian}). A long term goal would be to emulate possible approaches and results for the Ricci-harmonic flow of $\G2$-structures.

\medskip
In \textsection~\ref{sec:ltecompact}, we study the criterion for long-time existence of the Ricci-harmonic flow. We prove in Theorem~\ref{thm:lte1} that the Ricci-harmonic flow will exist on a compact manifold as long as 
\begin{align}\label{eq:lambintro}
\underset{(x,t)\in M\times [0,\tau)}{\text{sup}}\ \Lambda(x,t) = \underset{(x,t)\in M\times [0,\tau)}{\text{sup}}\ \left(|\Rm(x,t)|^2+|\del T(x,t)|^2+|T(x,t)|^4\right)^{\frac 12},
\end{align}
remains bounded. We follow this by proving a compactness theorem for solution of the flow in Theorem~\ref{thm:cmptthmflows}. We prove that for a sequence of pointed solutions to the Ricci-harmonic flow if the quantity $\Lambda$ defined in \eqref{eq:lambintro} is uniformly bounded and there is a uniform injectivity radius lower bound away from $0$ then there exists a subsequence which converges in the Cheeger--Gromov sense to a solution of the Ricci-harmonic flow. Using our modified local Shi-estimates Theorem~\ref{thm:strlocder} we prove a stronger version of the compactness theorem in Theorem~\ref{thm:strccompactnessthm}. These results, while interesting in their own right, are very useful in the analysis of the singularities of the flow. Our hope is that the stronger regularity results and the compactness theorem will be used in the same way as it has been used for the analysis of the solutions and singularities of the Ricci flow. As another application of our characterization of the singular time and the compactness theorem, we prove in Theorem~\ref{thm:velbound} that a solution to the Ricci-harmonic flow will keep on existing as long as the velocity of the flow, i.e., $|\Ric|+|T|^2+|\Div T|$ remains bounded. This is done by following the methods of Šešum \cite{sesum} for the Ricci flow case and those of Lotay--Wei \cite{lotay-wei-gafa} for the Laplacian flow of closed $\G2$-structures case. 

\medskip

Section \textsection~\ref{sec:solitons} studies solitons or self-similar solutions of the Ricci-harmonic flow. We prove in Proposition~\ref{prop:solmainprop} that there are no compact expanding solitons of the Ricci-harmonic flow and the only compact steady solitons for the flow are torsion-free $\G2$-structures. We then prove various identities and properties of solitons of the flow. Among other things, we prove an identity for compact \emph{gradient} Ricci-harmonic solitons in \eqref{eq:soliden2} which is the analog of Hamilton's identity for gradient Ricci solitons which has far-reaching applications in the study of Ricci solitons. We expect similar applications from \eqref{eq:soliden2}. We give a criterion for a Ricci-harmonic soliton to be trivial in Corollary~\ref{cor:soltrivial}. Finally, we prove an integral formula for the potential function of a gradient Ricci-harmonic soliton in Lemma~\ref{lem:intformsol}.

\medskip

There are special geometric structures on manifolds in dimension $8$ which are also related to $\bO$ and these are called Spin(7)-structures. These are prescribed by a $4$-form $\Phi$ on $M$ and the existence of such a $\Phi$ is a purely topological condition. A Spin(7)-structure $\Phi$ on $M$ induces a Riemannian metric $g_{\Phi}$ and an orientation and $\Phi$ is {\emph{torsion-free}} if $\del \Phi=0$, where $\del$ is the Levi-Civita connection of $g_{\Phi}$. Thus, similar to the $\G2$-case, the existence of a torsion-free Spin(7)-structure is obtained by solving a highly nonlinear PDE so one approach to find such a structure is by the means of a suitable geometric flow. Unlike the $\G2$-case, there have not been many proposed flows for Spin(7)-structures. A flow of isometric Spin(7)-structures, called the \emph{harmonic flow of Spin(7)-structures} was proposed in \cite{dle-isometric} and was shown to have various nice analytic and geometric properties. The harmonic flow was the negative gradient flow of the $L^2$-norm of the torsion functional, $\Phi \mapsto \frac 12 \int_M |T_{\Phi}|^2\vol$ with $T$ being the torsion of $\Phi$, when \emph{restricted to an isometry class} $[[\Phi_0]]_{\text{iso}}$. A flow where the metric is also varying was proposed by the author in \cite{dwivedi-spin7} where the flow was again the negative gradient flow of the $L^2$-norm of the torsion functional but now the Spin(7)-structures were allowed to vary over the space of \emph{all Spin(7)-structures.} Other second order flows of Spin(7)-structures were studied by Krasnov in \cite{krasnov-spin7}.

\medskip 

Motivated by the idea to obtain the Ricci-harmonic flow for $\G2$-structures, we first prove a Taylor series expansion for a Spin(7)-structure $\Phi$ in Theorem~\ref{thm:taylorspin7}. This result, although a straightforward emulation of the result for $\G2$-structures, is new and allows us to define the natural coupling of the Ricci flow and the harmonic flow of Spin(7)-structures with explicit lower order terms. We compute the Laplacian of the components of $\Phi$ by contracting the expression in the second order term of the Taylor series expansion \eqref{eq:taylorspin7} and get the Ricci-harmonic flow of Spin(7)-structures which is given in Definition~\ref{defn:rhfspin7}. The short-time existence and uniqueness of solutions to the flow is proved in Theorem~\ref{thm:srhfste} and essentially follows from the analysis of the principal symbols of the operators $\Ric_{\Phi}$ and $\Div_{\Phi}T_{\Phi}$ developed in an earlier work of the author \cite{dwivedi-spin7}. Finally in \textsection~\ref{sec:questions} we mention some problems for future directions. The methods in this paper are quite general and will work for the corresponding "Ricci-harmonic flow" with lower order terms obtained in the manner we describe here for any $H$-structures. For example, these methods provide natural flows of $\mathrm{SU}(2)$-structures (some of which were studied in \cite{fowdar-saearp}) and will give flows of $\mathrm{SU}(3)$-structures. These will be addressed in future works. 

\medskip 

We have also included an Appendix~\ref{sec:appendix} where we prove the Taylor series expansion of the dual $4$-form $\psi$ and follow the methods in \textsection~\ref{sec:motivation} to write the potential "heat equation" for $4$-forms which we hope will shed light on flows of $4$-forms.

\medskip
 
Another proposal for a general flow of $\G2$-structures with behaviour similar to the Ricci flow of metrics is described in the recent work for Gianniotis--Zacharopoulos \cite{panos-georgeG2Hilbert}. The flow in that paper is based on the variation of a modified Einstein-Hilbert-functional \cite[eq. (4.2)]{panos-georgeG2Hilbert} which the authors introduce and which is more suited for $\G2$-structures by taking into account the torsion of the $\G2$-structure. The flow in \eqref{ricci-harmonic flow} is different from the flows studied in \cite{panos-georgeG2Hilbert}, although, torsion-free $\G2$-structures are stationary points in both the cases.

	\medskip
	
\begin{ack}
We would like to thank Panagiotis Gianniotis, Spiro Karigiannis, Ragini Singhal and Thomas Walpuski for helpful discussions.
\end{ack}
	
	
	\section{Preliminaries on $\G2$ and Spin(7)-structures}\label{sec:prelims}
	
In this section, we briefly recall some preliminaries on $\G2$ and Spin(7)-structures and set up notations for the rest of the paper.

Throughout the paper, we compute in a local orthonormal frame, so all indices are subscripts and any repeated indices are summed over all values from $1$ to $7$ in the $\G2$-case and $1$ to $8$ in the Spin(7)-case. We have the Ricci identity which, for instance, for a $2$-tensor $S$ reads as
\begin{align}\label{eq:ricciiden}
	\del_i\del_jS_{kl}-\del_j\del_iS_{kl}=-R_{ijkm}S_{ml}-R_{ijlm}S_{km}. 
\end{align}

Let $M^7$ be an oriented smooth manifold. We say that $\g2\in \Omega^3(M)$ is a $\G2$-structure if it is nondegenerate, which means that it determines a Riemannian metric $g_\g2$ and an orientation which is given by the formula
\begin{align}\label{eq:metricfromphi}
(X\lrcorner \g2)\wedge (Y\lrcorner \g2)\wedge \g2 = -6 g_{\g2}(X,Y)\vol_{\g2},\ \ \ X,\ Y\in \Gamma(TM).
\end{align}
The metric and orientation determines the Hodge star operator and we denote by $\psi=*_{\g2}\g2$. The space of differential forms decompose further into irreducible $\G2$-representation. We have the following orthogonal decomposition with respect to $g$,
\begin{align*}
\Omega^2=\Omega^2_7\oplus \Omega^2_{14}\ \ \ \ \ \text{and}\ \ \ \ \ \Omega^3=\Omega^3_1\oplus \Omega^3_7\oplus \Omega^3_{27},
\end{align*}
where $\Omega^k_l$ has pointwise dimension $l$. Let $\gamma \in \Omega^k$. Given any $2$-tensor $A$ on $M$, we define
\begin{align}\label{eq:diadefn}
(A\diamond \gamma)_{i_1i_2\cdots i_k}=A_{i_1p}\gamma_{pi_2\cdots i_k}+A_{i_2p}\gamma_{i_1pi_3\cdots i_k}+\cdots + A_{i_kp}\gamma_{i_1i_2\cdots i_{k-1}p}.	
\end{align}	
So, for instance, $g\diamond \gamma = k\gamma$. The operation $\diamond$ is the infinitesimal action of the group $\mathrm{GL}(7, \bR)$ on the space of differential forms. If $\cT^2$ denote the space of $2$-tensors, then we get a linear map 
\begin{align*}
\diamond : \cT^2 \rightarrow \Omega^3.
\end{align*}

Recall that $\cT^2= C^{\infty}(M)g \oplus S^2_0 \oplus \Omega^2_7 \oplus \Omega^2_{14}$, where $S^2_0$ denote the space of symmetric traceless $2$-tensors. It's easy to prove that $\ker(\diamond) \cong \Omega^2_{14}$ and we have isomorphisms
\begin{align*}
C^{\infty}(M)\cong \Omega^3_1,\ \ \ \ \ \Omega^2_7 \cong \Omega^3_7 \cong  \Omega^1,\ \ \ \ \  \ S^2_0\cong \Omega^3_{27}.
\end{align*}
As a result, any $3$-form $\sigma\in \Omega^3(M)$ can be described by a pair $(h,X)$ where $h$ is a symmetric $2$-tensor and $X$ a vector field on $M$. We have
\begin{align*}
\sigma = h\diamond \g2+ X\lrcorner \psi.
\end{align*}
We have the following identities for contractions between $\g2$ and $\psi$.
\begin{align}
\g2_{ijk}\g2_{abk}&=g_{ia}g_{jb}-g_{ib}g_{ja}-\psi_{ijab}, \label{eq:phiwithphi1} \\
\g2_{ijk}\psi_{abck}&=g_{ia}\g2_{jbc}+g_{ib}\g2_{ajc}+g_{ic}\g2_{abj}-g_{ja}\g2_{ibc}-g_{jb}\g2_{aic}-g_{jc}\g2_{abi}. \label{eq:phiwithpsi1}
\end{align}

The torsion of a $\G2$-structure is a $2$-tensor $T$ and is given by
\begin{align}\label{eq:delphi}
	\del_m\g2_{ijk}= T_{mp}\psi_{pijk}.
\end{align}	
Using the identities \eqref{eq:phiwithphi1}-\eqref{eq:phiwithpsi1}, we can derive the following formulas,
\begin{align}
T_{pq}=\frac{1}{24}\del_p\g2_{ijk}\psi_{qijk}, \ \ \ \del_p\psi_{ijkl}=-T_{pi}\g2_{jkl}+T_{pj}\g2_{ikl}-T_{pk}\g2_{ijl}+T_{pl}\g2_{ijk}. \label{eq:delpsi}
\end{align}
Since $T$ is a $2$-tensor, we can write
\begin{align*}
T=T_1+T_{27}+T_7+T_{14}\ \ \ \ \ \ \ \text{where}\ \ T_1=\dfrac{\tr T}{7}g.
\end{align*}
These $T_i's$ are called \emph{intrinsic torsion forms} of the $\G2$-structure. We record here that 
\begin{align}
T^t = T_1+T_{27}-T_7-T_{14},\ \  |T|^2=|T_1|^2+|T_{27}|^2+|T_7|^2+|T_{14}|^2,\ \ (\tr T)^2=7|T_1|^2.
\end{align}
Since $\Omega^2_7\cong \Omega^1$ which, using the metric, can be identified with the vector fields, we can view the $T_7$ component of the torsion as a vector field (or a $1$-form). We denote it by $\Vop{T}$ with $(\Vop{T})_k=T_{ij}\g2_{ijk}$.

\medskip

The covariant derivative of the torsion of $\g2$ and the Riemann curvature tensor of the underlying metric are both second order in $\g2$ and they are related. In fact, since the torsion is a diffeomorphism invariant tensor, it satisfies a Bianchi-type identity which gives the relation between $\del T$ and the Riemann curvature tensor. Such an identity is called the $\G2$-Bianchi identity which  was first proved by Karigannis \cite[Thm. 4.2]{flows1} and is the following
\begin{align}\label{eq:G2Bianchi}
	\del_iT_{jq}-\del_jT_{iq}&=T_{ia}T_{jb}\g2_{abq}+\frac  12 R_{ijab}\g2_{abq}.
\end{align}	

A simple consequence of the $\G2$-Bianchi identity is an expression of the scalar curvature $R$ in terms of the intrinsic torsions which is given by
\begin{align}\label{eq:scalarcurv}
R=6|T_1|^2-|T_{27}|^2+5|T_7|^2-|T_{14}|^2 - 2\Div(VT).
\end{align}
We also have the following expression for the Lie derivative of $\g2$ in the direction of a vector field $Y$,
	\begin{align}\label{eq:liederivativephi}
	\cL_Y \g2= \frac 12 \cL_Yg \diamond \g2 + \left(-\frac 12 \curl Y + Y\lrcorner T\right)\lrcorner \psi,\ \ \ \ \text{where}\ \ (\curl Y)_k=\del_iY_j\g2_{ijk}.
	\end{align}
	
We now turn to Spin(7)-structures on $8$-dimensional manifolds. The reader can consult \cite{joycebook}, \cite{karigiannis-spin7} and \cite{dwivedi-spin7} for more details on Spin(7)-structures.

\medskip

A Spin(7)-structure on $M$ is a particular type of $4$-form $\s7$ on $M$. The existence of such a structure is a \emph{topological condition}. Concretely, an $8$-manifold admits a Spin(7)-structure if and only if it is orientable, spinnable, conditions which are equivalent to the vanishing of the first and second Stiefel--Whitney classes respectively, and for some orientation on M,
\begin{align*}
	p_1^2-4p_2+8\chi=0,
\end{align*}
where $p_i$ is the i-th Pontryagin class and $\chi$ is the Euler class of $M$. The space of $4$-forms which determine a Spin(7)-structure on $M$ is a subbundle $A$ of $\Omega^4(M)$, called the bundle of \emph{admissible} $4$-forms. This is \emph{not} a vector subbundle and it is not even an open subbundle, unlike the case for $\G2$-structures. For $p\in M$, the subbundle $A_p(M)$ is of codimension $27$ in $\Lambda^4(T^*_pM)$.

\medskip

\noindent
A Spin(7)-structure $\s7$  determines a Riemannian metric and an orientation on $M$. 

\begin{definition}
	Let $\del$ be the Levi-Civita connection of the metric $g_{\s7}$. The pair $(M^8, \s7)$ is a \emph{Spin(7)-manifold} if $\del \s7=0$. This is a non-linear partial differential equation for $\s7$, since $\del$ depends on $g$, which in turn depends non-linearly on $\s7$. A Spin(7)-manifold has Riemannian holonomy contained in the subgroup $\S7\subset \SO(8)$. Such a parallel Spin(7)-structure is also called \emph{torsion free}. 
\end{definition}

The existence of a Spin(7)-structure $\s7$ induces a decomposition of the space of differential forms on $M$  into irreducible Spin(7)-representations. We have the following orthogonal decompositions, with respect to $g_\s7$:
\begin{align*}
	\Omega^2=\Omega^2_{7}\oplus \Omega^2_{21},\ \ \ \ \ \ \ \Omega^3=\Omega^3_{8}\oplus \Omega^3_{48}, \ \ \ \ \ \ \ \ \  \Omega^4=\Omega^4_{1}\oplus \Omega^4_{7}\oplus \Omega^4_{27}\oplus \Omega^4_{35},
\end{align*}
where $\Omega^k_l$ has pointwise dimension $l$. 

\medskip

Since we are interested in flows of Spin(7)-structures, we only need the decomposition of $4$-forms for our current purposes. We again use the analog of the $\diamond$ operator for Spin(7)-structures. 
Given $A\in \Gamma(T^*M\otimes TM)$, define
\begin{align}
	\label{eq:diadefn1}
	A\diamond \s7= \frac{1}{24}(A_{ip}\s7_{pjkl}+A_{jp}\s7_{ipkl}+A_{kp}\s7_{ijpl}+A_{lp}\s7_{ijkp})e^i\wedge e^j\wedge e^k\wedge e^l,    
\end{align}
and hence 
\begin{align}
	\label{eq:diadefn2}
	(A\diamond \s7)_{ijkl}=  A_{ip}\s7_{pjkl}+A_{jp}\s7_{ipkl}+A_{kp}\s7_{ijpl}+A_{lp}\s7_{ijkp}. 
\end{align}

Just like the $\G2$-case, we have the following proposition. See \cite[Prop. 2.4]{dwivedi-spin7} for a proof.
\begin{proposition}\label{prop:diaproperties1}
	The kernel of the map $A\mapsto A\diamond \s7$ is isomorphic to the subspace $\Omega^2_{21}$. The remaining three summands $\Omega^0,\ S_0$ and $\Omega^2_7$ are mapped isomorphically onto the subspaces $\Omega^4_1,\ \Omega^4_{35}$ and $\Omega^4_7$ respectively.
\end{proposition}

We now describe the \emph{torsion} of a Spin(7)-structure. Given $X\in \Gamma(TM)$, we know from \cite[Lemma 2.10]{karigiannis-spin7} that $\del_X\s7$ lies in the subbundle $\Omega^4_7\subset\Omega^4$. 
\begin{definition}
The \emph{torsion tensor} of a Spin(7)-structure $\s7$ is the element of $\Omega^1_8\otimes \Omega^2_7$ defined using $\del \s7$.
Since $\del_X\s7\in \Omega^4_7$, by Proposition \ref{prop:diaproperties1}, $\del \s7$ can be written as
\begin{align}
	\label{Tdefneqn}
	\del_m\s7_{ijkl}=(T_m\diamond \s7)_{ijkl}=T_{m;ip}\s7_{pjkl}+T_{m;jp}\s7_{ipkl}+T_{m;kp}\s7_{ijpl}+T_{m;lp}\s7_{ijkp}    
\end{align}
where $T_{m;ab}\in\Omega^2_7$, for each fixed $m$. This defines the torsion tensor $T$ of a Spin(7)-structure, which is an element of $\Omega^1_8\otimes \Omega^2_7$.
\end{definition}

\noindent
In terms of $\del \s7$, the torsion $T$ is given by
\begin{align}
\label{eq:Texpress}
T_{m;ab} 
=\frac{1}{96}(\del_m\s7_{ajkl})\s7_{bjkl}    
\end{align}
since $T$ is an element of $\Omega^1_8\otimes \Omega^2_7$.
\begin{remark}
The notation $T_{m;ab}$ should not be confused with taking two covariant derivatives of $T_m$. The torsion tensor T is an element of $\Omega^1_8\otimes \Omega^2_7$ and thus, for each fixed index $m$, $T_{m;ab}\in \Omega^2_7$. \demo
\end{remark} 
\noindent
We write $\Div T$ for the divergence of the torsion which is an element of $\Omega^2_7$ and is given by
\begin{align}\label{eq:divT}
(\Div T)_{jk}=\del_mT_{m;jk}.
\end{align}

%

\medskip

\noindent
The Riemann curvature tensor $\Rm$ of the metric of a Spin(7)-structure and the covariant derivate of the torsion $\del T$ are also related and they satisfy a ``Bianchi-type identity''. This was first proved by Karigiannis \cite[Theorem 4.2]{karigiannis-spin7} using the diffeomorphism invariance of the torsion tensor and a different proof using the Ricci identity \eqref{eq:ricciiden} was given in \cite[Theorem 3.9]{dle-isometric}.

\begin{theorem}\label{thm:spin7bianchi}
The torsion tensor $T$ satisfies the following ``Bianchi-type'' identity
\begin{align}\label{spin7bianchi}
	\del_iT_{j;ab}-\del_jT_{i;ab}=2T_{i;am}T_{j;mb}-2T_{j;am}T_{i;mb}+\frac 14R_{jiab}-\frac 18R_{jimn}\s7_{mnab}.    
\end{align}
\qed
\end{theorem}
Using the Riemannian Bianchi identity, we see  that
\begin{align*}
R_{ijkl}\s7_{ajkl}=-(R_{jkil}+R_{kijl})\s7_{ajkl}=-R_{iljk}\s7_{aljk}-R_{ikjl}\s7_{akjl}    
\end{align*}
and hence we have the fact that 
\begin{align}\label{rmspin7}
R_{ijkl}\s7_{ajkl}=0.    
\end{align}
Using this and contracting \eqref{spin7bianchi} on $j$ and $b$ gives the expression for the Ricci curvature of a metric induced by a Spin(7)-structure. Precisely,
\begin{align}\label{ricci}
R_{ij}=4\del_iT_{a;ja}-4\del_aT_{i;ja}-8T_{i;jb}T_{a;ba}+8T_{a;jb}T_{i;ba}    
\end{align}
which also proves that the metric of a torsion free Spin(7)-structure is Ricci-flat, a result originally due to Bonan. Taking the trace of \eqref{ricci} gives the expression of the scalar curvature $R$
\begin{align}
R&=4\del_iT_{a;ia}-4\del_aT_{i;ia}+8|T_8|^2+8T_{a;jb}T_{j;ba}.\label{scalar1}
\end{align}

\section{Ricci-harmonic flow of $\G2$-structures}\label{sec:motivation}
We explain our motivation for the flow and why it can be regarded as a ``heat flow" for $\G2$-structures. We start with the discussion of the Ricci flow of metrics. Recall that if $(M^n,g)$ is a Riemannian manifold and $x\in M$ then we can write the Taylor series expansion of the components of the metric in Riemannian normal coordinates centred at the point $x$. This well-known expansion is given by 
		\begin{align}\label{eq:tsmetric}
		g_{ij}(x^1,\ldots, x^n)=\delta_{ij}+\frac 16 \left(R_{piqj}+R_{pjqi}\right)x^px^q+O(||x||^3).
		\end{align} 
One motivation for defining the Ricci flow as the evolution of a family of metrics $g(t)$ by $\ptt g(t)=-2\Ric(g(t))$, as described by Hamilton \cite[\textsection 1]{hamilton-singularities}, is that we compute the Laplacian of the components of the metric in normal coordinates, i.e.,
\begin{align*}
\Delta g_{ij}=g^{pq}\frac{\pt^2}{\pt x^p\pt x^q}g_{ij}=-\frac 23 R_{ij},
\end{align*}
where we used \eqref{eq:tsmetric} and hence we view the Ricci flow as heat equation for the metric. It can also be shown that in \emph{harmonic coordinates}, that is, coordinate system $(x^i)$ with $\Delta x^i=0$ for all $i$, the Laplacian of the components of the metric is precisely $-2\Ric$ (see \cite[Lemma 3.32]{Chow-Knopf} for a proof). Thus, from both the discussions we can view the Ricci flow as the heat flow for Riemannian metrics. As a result, we expect the ``smoothening'' of the metrics as they evolve by the Ricci flow.

We follow the same ideas for other special geometric structures, which in the present paper are $\G2$ and Spin(7)-structures. Let $(M^7, \g2)$ be a manifold with a $\G2$-structure. For $x\in M$, we choose our local orthonormal frame $\{e_1,\ldots,e_7\}$ of $T_xM$ to be $\G2$-\emph{adapted} which means that at the point $x$, the components $\varphi_{ijk}$ agree with those of the standard flat model on $\bR^7$. We recall the following Taylor series expansion of a $\G2$-structure $\g2$ from \cite{dgk-flows}.
	
	\begin{theorem} \label{thm:taylor-g2}
		\emph{\cite[Thm. 2.25]{dgk-flows}}
		Let $(x^1, \ldots, x^7)$ be $\G2$-adapted Riemannian normal coordinates centred at $x \in M$. The components $\g2_{ijk}$ of $\g2$ have Taylor expansions about $0$, which is the point in $\bR^7$ corresponding to $x \in M$, given by
		\begin{equation} \label{eq:taylor-g2}
			\g2_{ijk} (x^1, \ldots, x^7) = \g2_{ijk} + (T_{qm} \psi_{mijk}) x^q + \prescript{\g2}{}{\! \mathcal Q_{pq \, ijk}} x^p x^q + O(\|x\|^3),
		\end{equation}
		where
		\begin{align}\nonumber 
			\prescript{\g2}{}{\! \mathcal Q_{pq \, ijk}} & = \tfrac{1}{2} \del_{p} T_{qm} \psi_{mijk} - \tfrac{1}{2} (TT^t)_{pq} \g2_{ijk} \\ \nonumber
			& \qquad  + \tfrac{1}{2} T_{pm} (T_{qi} \g2_{mjk} + T_{qj} \g2_{mki} + T_{qk} \g2_{mij}) \\ 
			& \qquad {} + \tfrac{1}{6} (R_{piqm} \g2_{mjk} + R_{pjqm} \g2_{mki} + R_{pkqm} \g2_{mij}).\label{eq:taylor-ph-B}
		\end{align}
		Here all coefficient tensors on the right hand side are evaluated at $0$.
	\end{theorem}

Thus, following the same ideas as Hamilton, the Laplacian of the components of the $\G2$-structure $\g2$ at the point $x$ are given by 
\begin{align*}
	\Delta(\g2)_{ijk}&= \frac 12 ((\Div T)\lrcorner \psi)_{ijk} - \frac 12 |T|^2 \g2_{ijk}+\frac 12T_{pm}(T_{pi}\g2_{mjk}+T_{pj}\g2_{imk}+T_{pk}\g2_{ijm}) \\
	& \qquad {} - \frac 16(R_{im}\g2_{mjk}+R_{jm}\g2_{imk}+R_{km}\g2_{ijm})
\end{align*}	
which can be re-written in terms of the $\diamond$ operation as
\begin{align}\label{eq:heateqpre}
	\Delta(\g2)_{ijk}&=-\frac 16 (\Ric \diamond \g2)_{ijk} + \frac 12 (T^tT\diamond \g2)_{ijk} - \frac 16(|T|^2g\diamond \g2)_{ijk} + \frac 12 (\Div T\lrcorner \psi)_{ijk}.
\end{align}

As we want to describe the ``heat equation'' for the $\G2$-structure, we propose the following flow of $\G2$-structures, which we call the ``Ricci-harmonic flow of $\G2$-structures'' based on the discussion in \textsection\ref{sec:intro}.
\begin{definition}\label{heatflowdefn}
Let $(M^7, \g2_0)$ be a compact manifold with a $\G2$-structure $\g2_0$. The Ricci-harmonic flow for the family of $\G2$-structures $\g2(t)$ is the following initial value problem
\begin{align} 
	\label{heatfloweqn} 
	\left\{\begin{array}{rl} 
		& \dfrac{\pt \g2}{\pt t} = \left(-\Ric+3 T^tT -|T|^2g \right) \diamond \g2 + \Div T\lrcorner \psi, \\
		& \g2(0) =\g2_0.
		\tag{RHF}
	\end{array}\right. 
\end{align}
\end{definition}
\noindent
In coordinates, the flow reads as
\begin{align}\label{eq:RHFincoord}
	\frac{\pt}{\pt t}\g2_{ijk}&= (-R_{ip}+3T_{mi}T_{mp}-|T|^2g_{ip})\g2_{pjk}+(-R_{jp}+3T_{mj}T_{mp}-|T|^2g_{jp})\g2_{ipk}+(-R_{kp}+3T_{mk}T_{mp} \nonumber \\
	& \qquad  -|T|^2g_{kp})\g2_{ijp} + (\Div T)_l\psi_{lijk}.
\end{align}

We remark that we do not put \emph{any extra conditions} on either the initial $\G2$-structure $\g2_0$ or the evolving $\G2$-structures $\g2(t)$, for instance, (co)-closedness, isometric etc. 

The reader might have noticed that we multiplied the expression in the second order term of the Taylor series expansion by a factor of $+6$ so that the highest order term in the evolution of the underlying metric becomes $-2\Ric$ (see \eqref{eq:RHFgevol}) but we kept the vector field part in \eqref{heatfloweqn} as $\Div T$ instead of $3 \Div T$. The only reason we made this choice is so that the principal symbol of the differential operator on the right-hand side of \eqref{heatfloweqn}, after a DeTurck's type trick, becomes exactly the Laplacian in both the $\Omega^3_{1\oplus 27}$ and the $\Omega^3_7$ part, see \cite[Prop. 6.6]{dgk-flows} for the expression for the principal symbols of the operator $\Ric$ and $\Div T$ and \cite[\textsection 6.4]{dgk-flows} for the DeTurck's trick.  The choice of the factor $1$ instead of $3$ in the $\Div T$ term does not effect any results which we obtain in the paper (see also Remarks~\ref{rmk;3divste} and \ref{rmk:3divgrad}).

\medskip 

{\bf{Stationary points of the Ricci-harmonic flow.}} We now look at the stationary points of the Ricci-harmonic flow \eqref{heatfloweqn}. Since $T=0 \implies \Ric =0$ hence torsion-free $\G2$-structures are stationary points of the flow \eqref{heatfloweqn}. Let $M$ be compact  and suppose $\g2$ is a stationary point for the flow. Then $-\Ric + 3T^tT-|T|^2g=0$ and $\Div T=0$. Taking the trace of the first equation gives
\begin{align}
	R-3|T|^2+7|T|^2 =R+4|T|^2&=0 \nonumber 
\end{align}	
which on using the expression for the scalar curvature \eqref{eq:scalarcurv} gives
\begin{align}
	6|T_1|^2-|T_{27}|^2+5|T_7|^2-|T_{14}|^2 - 2\Div(\Vop T)+4(|T_1|^2+|T_7|^2+|T_{14}|^2+|T_{27}|^2)&=0 \nonumber 
\end{align}	
which implies that
\begin{align}
	10|T_1|^2+9|T_7|^2+3|T_{14}|^2+3|T_{27}|^2-2\Div(\Vop T) &=0, \label{eq:critpnts} 
	\end{align}
which on integrating over compact $M$ proves that $T_i=0,\ \text{for}\ i=1, 7, 14, 27$ and hence $\g2$ is torsion-free. {\bf{Thus, we see that stationary points of the flow \eqref{heatfloweqn} on compact manifolds are precisely torsion-free $\G2$-structures.}} We emphasize that this also shows that the lower order terms in our geometric flow are suitable which, in turn, were obtained by following the philosophy of describing the heat equation for $\G2$-structures. 

\begin{remark}
Suppose the evolving $\G2$-structures $\g2(t)$ are closed, i.e., $d\g2=0$. For instance, this is the setting of Bryant's Laplacian flow of closed $\G2$-structures. In that case, $T_{14}\in \Omega^2_{14}$ is the only non-vanishing torsion component. We have $d^*\g2=2T_{14}$ which gives
\begin{align*}
d^*T_{14}&=\frac 12 d^*(d^*\g2)=0	
\end{align*}
and hence for closed $\G2$-structures, $\Div T=\Div T_{14}=0$ (this is a well-known fact). Using the fact that $T$ is skew-symmetric, we get
\begin{align*}
	-\Ric + 3T^tT-|T|^2g&= -\Ric -3T^2-|T|^2g
\end{align*}
where $(T^2)_{ij}=T_{ip}T_{pj}$. The heat flow \eqref{heatfloweqn}, in this case, reads as
\begin{align*}
	\frac{\pt \g2}{\pt t}&= (-\Ric -3T^2-|T|^2g)\diamond \g2.
\end{align*}
The equation for the Laplacian flow of closed $\G2$-structures in \cite[eqs. (3.1)--(3.4)]{lotay-wei-gafa} are
\begin{align*}
\ptt \g2 = (-\Ric -2T^2-\frac 13|T|^2g)\diamond \g2,
\end{align*}
which shows that \eqref{heatfloweqn} differs from the Laplacian flow of closed $\G2$-structures only by constants on the lower order terms involved in the definition of the flow. It is tempting to modify the constants in the lower-order terms in \eqref{heatfloweqn} to match the Laplacian flow of closed $\G2$-structures so that the Ricci-harmonic flow reduces to the former for closed $\G2$-structures. We, however, see that such a modification will result in having classes other than torsion-free $\G2$-structures as stationary points of the flow and for this reason, we do not make the modification. \rightline{$\blacktriangle$}
\end{remark}

\begin{remark}\label{rem:RHFcoclos}
Suppose the evolving $\G2$-structures $\g2(t)$ are \emph{co-closed}, i.e., $d \psi(t)=0$ (recall that $\psi(t)=*_{\g2(t)}\g2(t))$. In this case, $T=T_1g +T_{27} \in \Omega^0g\oplus \Omega^3_{27}$ and hence the torsion is symmetric and so $T^t=T$. Moreover, contracting the $\G2$-Bianchi identity \eqref{eq:G2Bianchi} on the indices $i$ and $q$, we get
\begin{align*}
(\Div T^t)_j-\del_j \tr T&= T_{ia}T_{jb}\g2_{abi}+\frac 12 R_{ijab}\g2_{abi}\nonumber \\
&=0,
\end{align*}
where we used the fact that $T$ is symmetric and the Riemannian first Bianchi identity. Thus, we have $\Div T^t=\Div T=\del \tr T$. Thus, the Ricci-harmonic flow \eqref{heatfloweqn} in the case of co-closed $\G2$-structures is
\begin{align}\label{eq:RHFcoclos}
\ptt \g2(t)= \left(-\Ric +3T^2-|T|^2g\right)\diamond \g2 + \Div T\lrcorner \psi = \left(-\Ric +3T^2-|T|^2g\right)\diamond \g2 + \del \tr T\lrcorner \psi. 
\end{align}
This has very different lower order terms than, for instance, the Laplacian co-flow of co-closed $\G2$-structures in \cite{kmt} or its modification in \cite{grigorian-modified}. In fact, it is different from the natural heat flow for $4$-forms which one can derive using the Taylor series expansion of the $4$-form $\psi$ considered in Appendix~\ref{sec:appendix}. On the other hand, the flow in \eqref{eq:RHFcoclos} has resemblance to the geometric flow of $\G2$-structures obtained by looking at a $\G2$-Einstein-Hilbert functional in \cite{panos-georgeG2Hilbert}, precisely, look at the flows in \cite[eqs. (4.20) and (4.21)]{panos-georgeG2Hilbert}. This is very curious and should be investigated further. \demo
\end{remark}

\begin{remark}
We used compactness of $M$ to get rid of the divergence term in \eqref{eq:critpnts} and get that torsion-free $\G2$-structures are the only stationary points. The reader might have noticed that we have not used all the conditions of $\g2$ being a stationary point, in particular, the condition of $\Div T=0$ was not used at all in the computations leading to \eqref{eq:critpnts}. It's conceivable that one can prove the result on stationary points on noncompact manifolds as well using the condition of vanishing of $\Div T$.\demo
\end{remark}

To prove short-time existence and uniqueness of solutions to the flow \eqref{heatfloweqn}, we use the very general theorem from \cite{dgk-flows} which gives sufficient conditions for the short-time existence and uniqueness of solutions to general flows of $\G2$-structures, We recall the result below.

\begin{theorem}\label{thm:dgkmainthm}
\emph{\cite[Thm. 6.15]{dgk-flows}.}
Let $(M,\g2_0)$ be a compact $7$-manifold with a $\G2$-structure $\g2_0$. Consider the flow
\begin{equation} \label{eq:g2flow_general}
	\begin{aligned}
		\frac{\partial}{\partial t} \g2(t) & = (- \Ric + a \mathcal{L}_{\Vop T} g + \lambda F) \diamond \g2 + (b_1 \Div T+ b_2 \Div T^t) \lrcorner \psi + \text{lower\ order\ terms}, \\
		\g2(0) & = \g2_0,
	\end{aligned}
\end{equation}
where $F$ is the symmetric $2$-tensor given by $F_{ij}=R_{pqrs}\g2_{pqi}\g2_{rsj}$. Suppose that $0 \leq b_1 - a-1 < 4$, $b_1 + b_2 \geq 1$ and $|\lambda| < \frac{1}{4} c$, where $c = 1 - \frac{1}{4} (b_1 - a-1) > 0$.
\noindent
Then there exists $\varepsilon > 0$ and a unique smooth one-parameter family of $\G2$-structures $\g2(t)$ for $t \in [0,\varepsilon)$, solving~\eqref{eq:g2flow_general}.	
\end{theorem}

Comparing the expression of \eqref{eq:g2flow_general} with \eqref{heatfloweqn} we have $a=0, \lambda=0, b_1=1$ and $b_2=0$ which satisfy the conditions of the above theorem. Thus, we immediately get,

\begin{theorem}\label{thm:RHFste}
Given an initial $\G2$-structure $\g2_0$, the Ricci-harmonic flow of $\G2$-structures \eqref{heatfloweqn} admits a unique solution for a short time $[0, \varepsilon)$ on a compact manifold $M$ with $\varepsilon$ depending on the initial conditions. \qed
\end{theorem}

\begin{remark}\label{rmk;3divste}
We remark that Theorem~\ref{thm:RHFste} still remains true if one chooses $3\Div T$ instead of $\Div T$ for the family of vector fields defining the Ricci-harmonic flow. \demo
\end{remark}

Thus, the Ricci-harmonic flow of $\G2$-structures is well-posed. A desirable property of a geometric flow is if it is a gradient flow of some functional. This was the case for the harmonic flow of $\G2$-structures \cite{dgk-isometric} (and harmonic flow of Spin(7)-structures as well \cite{dle-isometric}), the Laplacian flow of closed $\G2$-structures \cite{bryant-remarks}, \cite{lotay-wei-gafa}, the heat flow of Weiss--Witt \cite{weiss-witt} and other families of flows of $\G2$-structures considered in \cite{dgk-flows}. This is unlike the situation with the Ricci flow of metrics which is not a gradient flow of a functional on the space of Riemannian metrics and one of the major breakthroughs of Perelman \cite{perelman-1} was to prove that the Ricci flow can indeed be viewed as a gradient flow but on an enlarged space. We show below that there is \emph{no} diffeomorphism invariant functional on the space $\Omega^3_{+}$ of $\G2$-structures whose gradient is the Ricci-harmonic flow \eqref{heatfloweqn}. This essentially follows from a result of Gianniotis--Zacharopoulos \cite{panos-georgeG2Hilbert}. We first restate their result.

\begin{proposition}
	\emph{\cite[Prop. 3.3]{panos-georgeG2Hilbert}.}
	Let $M^7$ be compact with a $\G2$-structure $\g2$. Suppose that $\mathcal{F} : \Omega^3_{+} \rightarrow \bR$ is a diffeomorphism invariant functional, that is, $\mathcal{F}(\Phi^*\g2)=\mathcal{F}(\g2)$ for any diffeomorphism $\Phi$ of $M$, and that there are second order quasilinear differential operators $Q_1 : \Omega^3_+ \rightarrow S^2$ and $Q_2 : \Omega^3_+ \rightarrow \Omega^1$ such that for any variation $(\g2(t))_{t\in (-\epsilon, \epsilon)}$ of $\G2$-structures with $\g2(0)=\g2$ and $\left. \frac{d\g2}{dt}\right |_{t=0} = h\diamond \g2 + X\lrcorner \psi$ and
	\begin{align}\label{eq:diffinvfunctionaleq1}
		\left.\frac{d\mathcal{F}(\g2(t))}{dt}\right|_{t=0}= \int_M \langle h, Q_1(\g2)\rangle + \langle X, Q_2(\g2)\rangle \vol_{\g2}.
	\end{align}
Then, we have 
\begin{align}\label{eq:diffinvfunctionaleq2}
	Q_1(\g2) &= \alpha \Ric + \beta Rg + \gamma \cL_{\Vop T}g + \zeta F + \text{l.o.ts} \nonumber \\
	Q_2(\g2)&= \delta \Div T + \epsilon \del \tr T + \text{l.o.ts.},
\end{align}
where $R$ is the scalar curvature and $l.o.ts.$ mean lower order terms and the coefficients must satisfy 
\begin{align}\label{eq:diffinvfunctionaleq3}
\frac{\alpha}{2}+\beta -\frac{\gamma}{2}+\frac{\delta}{4}=0 \ \ \ \ \ \ \text{and}\ \ \ \ \gamma+\frac{\delta}{2}=0.	
\end{align}
\end{proposition}

We refer the readers to \cite{panos-georgeG2Hilbert} for a proof of the proposition but we mention the main idea. The expressions for $Q_1$ and $Q_2$ up to lower order terms follow from the classification all second order differential invariants of a $\G2$-structure which can be made into a $3$-form in \cite{dgk-flows}. For deriving \eqref{eq:diffinvfunctionaleq3} we notice that since $\cF$ is diffeomorphism invariant, we must have $D \cF_{\g2}(\cL_{Y}\g2)=0$ for any vector field $Y$ on $M$. It, then, follows from \eqref{eq:liederivativephi} that 
\begin{align*}
	\int_M \left\langle Q_1, \frac 12 \cL_{Y}g \right\rangle  + \left \langle Q_2, -\frac 12 \curl Y + Y\lrcorner T \right \rangle \vol =0,
\end{align*}
and then one can compute and compare the highest order terms to arrive at \eqref{eq:diffinvfunctionaleq3}. Comparing \eqref{eq:diffinvfunctionaleq2} with the Ricci-harmonic flow \eqref{heatfloweqn} we see that $\alpha=-1, \beta=0, \gamma=0, \zeta =0, \delta =1$ and $\epsilon=0$ and they do not satisfy \eqref{eq:diffinvfunctionaleq3}. As a result, we have the following,

\begin{proposition}\label{prop:RHFisnotgradient}
	The Ricci-harmonic flow of $\G2$-structures defined in \eqref{heatfloweqn} is not a gradient flow of any diffeomorphism invariant functional on the space of $\G2$-structures. \qed
	\end{proposition}

The fact that the Ricci-harmonic flow is not a gradient flow of any diffeomorphism invariant functional on the space of $\G2$-structures is similar to the case of the Ricci flow of metrics which is not a gradient flow of any functional on just the space of Riemannian metrics and one needs to enlarge the space in consideration to view the Ricci flow as a gradient of functionals which was understood after the pioneering work of Perelman \cite{perelman-1}. This leads to the following interesting question.

\begin{question}
 Can one define analogs of Perelman's $\cF$ and $\mathcal{W}$ functionals, possibly on an enlarged space containing $\Omega^3_+(M)$, so that the Ricci-harmonic flow of $\G2$-structures can be viewed as the gradient flow of such functionals and the functionals are monotonic along the flow?
\end{question}

The $\cF$ and $\mathcal{W}$-functionals of Perelman have origins in the study of steady and shrinking gradient Ricci solitons respectively so maybe similar ideas might be helpful in answering the above question. In \textsection~\ref{sec:solitons}, we study solitons of the Ricci-harmonic flow and derive various important identities which we believe will be helpful in answering the previous question.
	
\begin{remark}\label{rmk:3divgrad}
We again remark that Proposition~\ref{prop:RHFisnotgradient} remains true even if we choose $3\Div T$ instead of $\Div T$ in \eqref{heatfloweqn}.\demo
\end{remark}

We recall that a $\G2$-structure is called \emph{nearly $\G2$} if $d\g2=\lambda \psi$ and $d\psi=0$, where $\lambda$ is a constant. In this case, $T_1$ is the only non-zero torsion component and it is a constant. Suppose $\g2_0$ is a nearly $\G2$-structure with $T_{\g2_{0}}=cg_0$. We have the following proposition.

\begin{proposition}\label{prop:RHFonNG2}
Let $(M^7, \g2_0)$ be a nearly $\G2$-manifold and let $T_{\g2_{0}}=cg_0$. The Ricci-harmonic flow \eqref{heatfloweqn} with $\g2_0$ as initial condition exists for finite time interval $[0, \tau]$ where $\tau=\tfrac{1}{20c^2}$, and remains nearly $\G2$ for its time of existence. As such, nearly $\G2$-structures are {\bf{shrinking}} solutions of the flow.	
\end{proposition}

\begin{remark}
It can be computed that $c$ in the statement of the proposition is $\frac{\tr T}{7}$ and $d\g2_0=\tfrac 47 \tr T \psi$.\demo
\end{remark}	

\begin{proof}
It is well-known (see for instance, \cite[\textsection 2]{dwivedi-singhal}) that the Ricci curvature of $\g2_0$ is given by $\Ric_{g_0} = 6c^2g_0$. Consider the family of $\G2$-structures 
\begin{align}\label{eq:ng2aux1}
\g2(t)= \left(1-20c^2t\right)^{\frac 32} \g2_0.
\end{align}
Clearly $\g2(0)=\g2_0$. Since $\g2(t)$ is a scaling of $\g2_0$, it follows that 
\begin{alignat}{6}
	&g(t)&&= \left(1-20c^2t\right)  g_0, \ \ \    &&g^{-1}(t) &&=  \left(1-20c^2t\right) ^{-1}g_{0}^{-1}  \ \ \ &\Ric(g(t))&&=6c^2g_0,\nonumber \\
	&T(t)&&= c \left(1-20c^2t\right)^{\frac 12} g_0,  \ \ \  &&|T(t)|^2g(t)&&=7c^2g_0,\ \ \text{and}\ \  & \Div T(t)&&=0. \nonumber
	\end{alignat}
Thus, we see that 
\begin{align*}
\ptt \g2(t) = -30c^2 \left(1-20c^2t\right)^{\frac 12}\g2_0,
\end{align*}
and 
\begin{align*}
(-\Ric + 3T^tT-|T|^2g)\diamond \g2 = (-10c^2g_0)\diamond_t \g2(t)=-10c^2\left(1-20c^2t\right)^{\frac 12}(g_0\diamond_0\g2_0)= -30c^2 \left(1-20c^2t\right)^{\frac 12}\g2_0.
\end{align*}

\end{proof}

\begin{remark}
Nearly $\G2$-manifolds are positive Einstein so the behaviour of nearly $\G2$-structures along the Ricci-harmonic flow is similar to the behaviour of positive Einstein metrics along the Ricci flow where the latter are shrinking solutions of the flow. \demo
\end{remark}


\section{Evolution equations and Shi-type estimates}\label{sec:evoleqns}

We start this section  by computing the evolution of the Riemann curvature tensor, the torsion tensor and the covariant derivative of the torsion tensor along the Ricci-harmonic flow. Throughout this section (unless stated otherwise), we will use the following convention. If $A$ and $B$ are two tensors on $(M, \g2, g)$ then $A*B$ will denote any quantity obtained from $A\otimes B$ by contracting using the metric $g$  or its inverse $g^{-1}$  and multiplication by constants depending only on the ranks of $A$ and $B$ and the dimension of the manifold which in our case is $7$. We will use the same letter $C$ for various constants which differ from line to line to declutter the text but all the constants will depend on the same quantities and will be prescribed beforehand. 

\medskip

It was proved by Karigiannis \cite[\textsection 3]{flows1} that if a family of $\G2$-structures $\g2(t)$ evolve by
\begin{align*}
	\frac{\pt}{\pt t}\g2(t)=(h(t)\diamond \g2(t)) + X(t)\lrcorner \psi(t)
\end{align*}
for a family of symmetric $2$-tensors $h(t)$ and a family of vector fields $X(t)$ on $M$ then we have the following evolution equations:
\begin{align}\label{eq:genevoleqns}
	\begin{split}
	\frac{\pt}{\pt t} g(t)_{ij}&= 2h(t)_{ij}, \\
	\frac{\pt}{\pt t}g^{-1}(t)^{ij} &= - 2h^{ij},\ \ \ \text{with}\ \ h^{ij}=g^{ia}g^{jb}h_{ab},  \\
	\frac{\pt}{\pt t}\vol_{g(t)}&= \tr_{g(t)} h(t) \vol_{g(t)},  \\
	\frac{\pt}{\pt t} \psi_{ijkl} &= h_{im}\psi_{mjkl}+h_{jm}\psi_{imkl}+h_{km}\psi_{ijml}+h_{lm}\psi_{ijkm}-X_{i}\g2_{jkl}+X_j\g2_{ikl}-X_k\g2_{ijl}+X_l\g2_{ijk} \\
	\frac{\pt }{\pt t} T_{ij} &=T_{im}h_{mj} + T_{im}X_n\g2_{nmj}+\del_mh_{ni}\g2_{mnj}+\del_iX_j. 
\end{split}
\end{align}

From \eqref{eq:genevoleqns}, along the Ricci-harmonic flow, we have
\begin{align}
	\frac{\pt}{\pt t}g_{ij}& = -2R_{ij}+6 T_{pi}T_{pj} -2|T|^2g_{ij} \label{eq:RHFgevol}, \\
	\frac{\pt}{\pt t}g^{ij}&= 2R^{ij}-6\tensor{T}{_p^i}\tensor{T}{_p^j}+2|T|^2g^{ij}, \label{eq:RHFginvevol}
\end{align}
and the volume form evolves as
\begin{align}
	\frac{\pt}{\pt t} \vol &= \tr( -R_{ij}+3 T_{pi}T_{pj} -|T|^2g_{ij} ) \vol \nonumber \\
	&=(-R+3|T|^2-7|T|^2) \vol \nonumber \\
	&=-(R+4|T|^2) \vol \nonumber \\
	&\overset{\eqref{eq:scalarcurv}}{=} -\left(10|T_1|^2+9|T_7|^2+3|T_{14}|^2+3|T_{27}|^2-2\Div(\Vop T) \right)\vol. \label{eq:RHFvolevol}
\end{align}

Integrating \eqref{eq:RHFvolevol} on compact $M$ proves that the volume of $M$ with respect to $g(t)$ will decrease along the Ricci-harmonic flow \eqref{heatfloweqn}.

\medskip

\subsection{Evolution of curvature quantities}\label{eq:subseccurvevol}
We recall that if a family of metrics $g(t)$ on a manifold evolve by
\begin{align}\label{eq:gengevol}
	\frac{\pt g(t)}{\pt t}=h(t)
\end{align}
for some time-dependent symmetric $2$-tensor $h(t)$, then the Riemann curvature tensor, the Ricci tensor and the scalar curvature of the underlying metric evolve by (see, for instance, \cite[Chapter 3]{Chow-Knopf})
\begin{align}
	\frac{\pt }{\pt t}R_{ijk}^{\phantom{ijk}l} &= \frac 12 g^{lp}\left(\del_i\del_kh_{jp}+\del_j\del_ph_{ik}-\del_i\del_ph_{jk}-\del_j\del_kh_{ip}-R_{ijk}^{\phantom{ijk}q}h_{qp}-R_{ijp}^{\phantom{ijp}q}h_{kq}\right), \label{eq:genRiemevol} \\
	\frac{\pt}{\pt t} R_{jk}&= \frac 12 g^{pq}\left(\del_q\del_jh_{kp}+\del_q\del_kh_{jp}-\del_q\del_ph_{jk}-\del_j\del_kh_{pq} \right), \label{eq:genRicevol} \\
	\frac{\pt}{\pt t} R&= - \Delta \tr_g (h) + \Div (\Div h) - \langle h, \Ric \rangle. \label{eq:genscalevol}
\end{align}
Here $\Delta$ is the ``analyst's Laplacian'' and its expression in local coordinates is $\Delta=\del_i\del^i$, and $\Div (h)_{i}= \del^jh_{ji}$. 

Using \eqref{eq:RHFgevol} and \eqref{eq:genRiemevol}, we have
\begin{align*}
	\frac{\pt}{\pt t}R_{ijk}^{\phantom{ijk}l}&= -\del_i\del_kR_{j}^l-\del_j\del^lR_{ik}+\del_i\del^lR_{jk}+\del_j\del_kR_{i}^l + \left(R_{ijkq}\tensor{R}{_q^l}+\tensor{R}{_i_j^l_q}R_{kq}\right)  \nonumber \\
	& \quad -3\left(R_{ijkq}T_{pq}T_{p}^{\phantom{p}l}+\tensor{R}{_i_j^l_q}T_{pk}T_{pq}\right) \nonumber \\
	& \quad - \left(\del_i\del_k |T|^2 \tensor{g}{_j^l} + \del_j\del^l |T|^2g_{ik}-\del_i\del^l |T|^2g_{jk}-\del_j\del_k |T|^2\tensor{g}{_i^l}\right) \nonumber \\
	& \quad  + 3 \left(\del_i\del_k(T_{pj}\tensor{T}{_p^l})+\del_j\del^l(T_{pi}T_{pk})-\del_i\del^l(T_{pj}T_{pk})-\del_j\del_k(T_{pi}\tensor{T}{_p^l})\right).
\end{align*}
The first six terms in the above equation are contributions from the $-2\Ric$ term and hence they can be analysed in the same way as the evolution of the Riemann curvature tensor in the Ricci flow case (see \cite[Lemma 6.13]{Chow-Knopf}) where we apply the Riemannian first and second Bianchi identities and the Ricci identity \eqref{eq:ricciiden}, to get
\begin{align}\label{eq:Riemevol1}
\frac{\pt}{\pt t}R_{ijk}^{\phantom{ijk}l}&=\Delta R_{ijk}^{\phantom{ijk}l}+\left(R_{ijpr}R_{rpk}^{\phantom{rpk}l}-2R_{pikr}R_{jpr}^{\phantom{jpr}l}+2R_{pir}^{\phantom{pir}l}R_{jpkr}\right) -R_{ip}R_{pjk}^{\phantom{pjk}^l}-R_{jp}R_{ipk}^{\phantom{ipk}^l}-R_{kp}R_{ijp}^{\phantom{ijp}^l} \nonumber \\
& \quad +R_{p}^{\phantom{p}l}R_{ijkp}  -3\left(R_{ijkq}T_{pq}T_{p}^{\phantom{p}l}+\tensor{R}{_i_j^l_q}T_{pk}T_{pq}\right) \nonumber \\
& \quad  - \left(\del_i\del_k |T|^2 \tensor{g}{_j^l} + \del_j\del^l |T|^2g_{ik}-\del_i\del^l |T|^2g_{jk}-\del_j\del_k |T|^2\tensor{g}{_i^l}\right) \nonumber \\
& \quad  + 3 \left(\del_i\del_k(T_{pj}\tensor{T}{_p^l})+\del_j\del^l(T_{pi}T_{pk})-\del_i\del^l(T_{pj}T_{pk})-\del_j\del_k(T_{pi}\tensor{T}{_p^l})\right).
\end{align}

The above equation can be schematically written as 
\begin{align}
	\frac{\pt}{\pt t}\Riem&= \Delta \Riem + \Riem * \Riem + \Riem * T*T+ \del T* \del T+ \del^2T * T. \label{eq:Riemevolschem}  
\end{align}

\medskip

Since for any tensor $A$ on $M$ we have $\Delta |A|^2 = 2\langle A, \Delta A\rangle + 2|\del A|^2$, using \eqref{eq:RHFginvevol} and \eqref{eq:Riemevolschem}, we have
\begin{align}
	\frac{\pt }{\pt t}|\Riem|^2&= \frac{\pt}{\pt t}\left(g^{ia}g^{jb}g^{kc}g^{ld}R_{ijkl}R_{abcd}\right) \nonumber \\
	& = \Riem * \Riem * \left(\Ric + T*T\right) + 2\left \langle  \Riem , \frac{\pt}{\pt t} \Riem\right \rangle \nonumber \\
	& \leq \Delta |\Riem |^2 - 2|\del \Riem|^2 + C|\Riem|^3+C|\Riem|^2|T|^2 + C|\Riem||\del  T|^2 + C|\Riem||T||\del^2T|, \label{eq:Rimesqschemevol} 
\end{align}
where $C$ is some universal constant.

\medskip

Using \eqref{eq:genRicevol}, we have
\begin{align}
	\frac{\pt}{\pt t}R_{jk}&= \Delta\left( R_{jk}-3T_{pj}T_{pk}+|T|^2g_{jk}\right) + \del_j\del_k  \left(R+4|T|^2 \right) - \del_p\del_j\left(R_{kp}-3T_{sk}T_{sp}+|T|^2g_{kp}\right) \nonumber \\
	& \qquad -\del_p\del_k \left(R_{jp}-3T_{sj}T_{sp}+|T|^2g_{jp}\right), \label{eq:RHFRicevol}
\end{align}
and the evolution of the scalar curvature is
\begin{align}
	\frac{\pt}{\pt t} R&=-\Delta ( \tr(-2R_{ij}+6 T_{pi}T_{pj} -2|T|^2g_{ij}))+\del_j\del_i( -2R_{ij}+6 T_{pi}T_{pj} -2|T|^2g_{ij}) \nonumber \\
	& \qquad-   R_{ij}(-2R_{ij}+6 T_{pi}T_{pj} -2|T|^2g_{ij} ) \nonumber \\
	&= \Delta R +  6\Delta |T|^2 + 6 \del_j  \del_i (T_{pi}T_{pj}) + 2|\Ric|^2 - 6 R_{ij}T_{pi}T_{pj} + 2|T|^2R, \label{eq:RHFscalevol}
\end{align}
where we used the twice contracted Riemannian second Bianchi identity.

\medskip

\subsection{Evolution of torsion}\label{subsec:torsionevol}

We compute the evolution of the torsion tensor along the Ricci-harmonic flow. Using \eqref{eq:RHFincoord} in the evolution of torsion in \eqref{eq:genevoleqns}, we have
\begin{align}
\frac{\pt}{\pt t}T_{ij}&= T_{im}(-R_{mj}+3T_{pm}T_{pj}-|T|^2g_{mj})+T_{im}\del_aT_{an}\g2_{nmj} + \del_i\del_aT_{aj} \nonumber \\
& \quad + \del_m(-R_{ni}+3T_{pn}T_{pi}-|T|^2g_{ni})\g2_{mnj} \nonumber \\
& = -T_{im}R_{mj}+3T_{im}T_{pm}T_{pj}-|T|^2T_{ij}+T_{im}\del_aT_{an}\g2_{nmj}+\del_i\del_aT_{aj} \nonumber \\
& \quad - \del_mR_{ni}\g2_{mnj}+3\del_m(T_{pn}T_{pi})\g2_{mnj}-\del_m|T|^2\g2_{mij}. \label{eq:torevolaux1}
\end{align}

We use the $\G2$-Bianchi identity \eqref{eq:G2Bianchi} and the Ricci identity \eqref{eq:ricciiden} to compute the Laplacian of the torsion tensor. We have
\begin{align}
\Delta T_{ij}&=\del_a\del_aT_{ij} \nonumber \\
	&=\del_a(\del_iT_{aj}+T_{ab}T_{ic}\g2_{bcj}+\frac 12R_{aibc}\g2_{bcj}) \nonumber \\
	&=\del_a\del_i T_{aj}+(\del_aT_{ab})T_{ic}\g2_{bcj} + T_{ab}(\del_aT_{ic})\g2_{bcj}+T_{ab}T_{ic}T_{al}\psi_{lbcj} + \frac 12 \del_aR_{aibc}\g2_{bcj}+ \frac 12 R_{aibc}T_{al}\psi_{lbcj} \nonumber \\
	&= \del_i\del_aT_{aj}+R_{iq}T_{qj}-R_{aijq}T_{aq}+(\del_aT_{ab})T_{ic}\g2_{bcj} + T_{ab}(\del_aT_{ic})\g2_{bcj} + \frac 12(\del_cR_{ib}-\del_bR_{ic})\g2_{bcj} \nonumber \\
	& \qquad + \frac 12 R_{aibc}T_{al}\psi_{lbcj} \nonumber \\
	&= \del_i\del_aT_{aj}+R_{iq}T_{qj}-R_{aijq}T_{aq}+(\del_aT_{ab})T_{ic}\g2_{bcj} + T_{ab}(\del_aT_{ic})\g2_{bcj} +  \del_cR_{bi}\g2_{bcj}+ \frac 12 R_{aibc}T_{al}\psi_{lbcj} \label{eq:torevolaux2}
\end{align}
where  we used the fact that  $T_{ab}T_{ic}T_{al}$ is symmetric in $b, l$ while $\psi_{lbcj}$ is skew in $b, l$ for the fourth term in the third equality and the once contracted second Bianchi identity in the fourth equality.

Using \eqref{eq:torevolaux1} and \eqref{eq:torevolaux2}, we have
\begin{align}
	\frac{\pt}{\pt t}T_{ij}-\Delta T_{ij}&= -T_{im}R_{mj}+3T_{im}T_{pm}T_{pj}-|T|^2T_{ij}+T_{im}\del_aT_{an}\g2_{nmj}+\del_i\del_aT_{aj} \nonumber \\
	& \quad - \del_mR_{ni}\g2_{mnj}+3\del_m(T_{pn}T_{pi})\g2_{mnj}-\del_m|T|^2\g2_{mij} \nonumber \\
	& \quad -[\del_i\del_aT_{aj}+R_{iq}T_{qj}-R_{aijq}T_{aq}+(\del_aT_{ab})T_{ic}\g2_{bcj} + T_{ab}(\del_aT_{ic})\g2_{bcj} +  \del_cR_{bi}\g2_{bcj} \nonumber \\
	& \qquad + \frac 12 R_{aibc}T_{al}\psi_{lbcj}] \nonumber \\
	&=-T_{im}R_{mj}-R_{im}T_{mj}+3T_{im}T_{pm}T_{pj}-|T|^2T_{ij} + 3\del_m(T_{pn}T_{pi})\g2_{mnj}-\del_m|T|^2\g2_{mij} \nonumber \\
	& \quad + R_{aijq}T_{aq}  -T_{ab}(\del_aT_{ic})\g2_{bcj} -\frac{1}{2}R_{aibc}T_{al}\psi_{lbcj}, \nonumber
\end{align}
and so we get
\begin{align}
	\frac{\pt}{\pt t}T_{ij}& = \Delta T_{ij} -T_{im}R_{mj}-R_{im}T_{mj}+3T_{im}T_{pm}T_{pj}-|T|^2T_{ij} + 3\del_m(T_{pn}T_{pi})\g2_{mnj}-\del_m|T|^2\g2_{mij} \nonumber \\
	& \quad+ R_{aijq}T_{aq}  -T_{ab}(\del_aT_{ic})\g2_{bcj} -\frac{1}{2}R_{aibc}T_{al}\psi_{lbcj}. \label{eq:torevol}
\end{align}

The above equation can be schematically written as
\begin{align}
\frac{\pt}{\pt t}T& = \Delta T + \Riem*T + T*T*T+ \del T*T*\g2 +\Riem*T*\psi. \label{eq:torevolsch}
\end{align}

\begin{remark}
The schematic evolution of the torsion tensor along the Ricci-harmonic flow is the same as the corresponding equation for the Laplacian flow for closed $\G2$-structures \cite[eq. (3.13)]{lotay-wei-gafa}. \demo
\end{remark}

\subsection{Global Shi-type estimates for the Ricci-harmonic flow}\label{subsec:shiest}

In this section, we develop the regularity theory for the solutions of the Ricci-harmonic flow. We recall the following definition from \cite[Def. 1.1]{gaochen-shi}.

\begin{definition}
A flow of $\G2$-structures of the form $\pt_t \g2 = h\diamond \g2 + C\lrcorner \psi$ is called a \emph{reasonable flow} if the flow has a unique solution for short-time and it satisfies the following conditions.
\begin{enumerate}
	\item The underlying metric evolves by 
	\begin{align}\label{eq:genreaf1}
		\frac{\pt}{\pt t}g_{ij} = -2R_{ij} + C + L(T)+T*T.
	\end{align} 
\item The family of vector fields $X$ must be of the form
\begin{align}
	X=C+L(T)+L(\Riem)  + L(\del T) + T*T. \label{eq:genreaf2}
\end{align}	
\item The torsion tensor evolves by
\begin{align}
	\frac{\pt}{\pt t}T = \Delta T + L(T) + L(\del T) + \Riem *T + \del T*T + T*T + T*T*T. \label{eq:genreaf3}
\end{align}
\end{enumerate}
Here. $L$ and $*$ denote linear maps and multi-linear maps in variables other than $\g2$, $\psi$ and $g$ respectively.
\end{definition}

\medskip

Since $X= \Div T$ for the Ricci-harmonic flow, comparing \eqref{eq:genreaf1} with \eqref{eq:RHFgevol}, \eqref{eq:genreaf3} with \eqref{eq:torevol} and using Theorem~\ref{thm:RHFste}, we see that the Ricci-harmonic flow is a reasonable flow of $\G2$-structures. As a consequence, the following local derivative  estimates  for the solutions of the Ricci-harmonic flow follows from \cite[Thm. 2.1]{gaochen-shi}.

\begin{theorem}\label{thm:locderest}
Let $K>0$ and $r>0$. Let $p\in M$ and suppose $\g2(t)$ is a solution to the Ricci-harmonic flow with $t\in\left[0, \frac 1K\right] $, on an open neighbourhood $U\ni p$ with $B_{g(0)}(p, r) \subset U$. If
\begin{align}\label{eq:locderestassump}
	|\Riem|+|\del T|+|T|^2 \leq K
\end{align}	\label{eq:locderestconc}
for all $x\in B_{g(0)}(p, r)$ and $t\in \left[0, \frac 1K\right]$, then for all $k\in \mathbb{N}$, there  exists a constant $C=C(K, r, k)$ such that
\begin{align}
|\del^k \Riem|+|\del^{k+1}T|\leq C(K, r, k)t^{-\frac k2},
\end{align}
on $B_{g(0)}(p, \frac r2) \times \left[0, \frac 1K\right] $.

\end{theorem}

\medskip

In fact, we prove the following ``doubling-time estimate'' which shows why the assumption in \eqref{eq:locderestassump} is a reasonable one. Suppose $\g2(t)$ is a Ricci-harmonic flow on a compact manifold $M$ and define

\begin{align}\label{eq:Lambdadefn}
	\Lambda(x,t)= \left(|\Riem(x,t)|^2+|\del T(x,t)|^2+|T(x,t)|^4\right)^{\frac 12}.
\end{align}
Let $\Lambda(t)= \underset{M}{\text{sup}}\  \Lambda(x,t)$. 

\begin{lemma}[Doubling-time estimate]\label{lem:dte}
Let $\g2(t)$ be a Ricci-harmonic flow on a compact manifold $M$ for $t\in[0, \tau]$. Then there exists a constant $C>0$ and a $\delta>0$ such that 
\begin{align*}
\Lambda(t)\leq 2\Lambda(0) \ \ \ for\ all\ 0\leq t\leq \delta,
\end{align*}
with $\delta \leq$ \emph{min}$\left\{\tau, \frac{1}{C\Lambda(0)} \right\}$.
\end{lemma}

\begin{proof}
The idea of the proof is similar to the proof of such estimate for the Ricci flow of metrics or the Laplacian flow of closed $\G2$-structures \cite[Prop. 4.1]{lotay-wei-gafa}. We will first derive a differential inequality for $\Lambda(t)$ along the Ricci-harmonic flow	and then use the scalar maximum principle to deduce the estimate.

\medskip

We first compute the evolution of $|T|^4$. Using \eqref{eq:torevolsch}, we have
\begin{align}
\frac{\pt}{\pt t}|T|^4 &=\frac{\pt}{\pt t}(g^{ia}g^{jb}g^{kc}g^{ld}T_{ij}T_{ab}T_{kl}T_{cd}) \nonumber \\
& = 2|T|^2 \left\langle2 \frac{\pt}{\pt t}T, T \right \rangle + T*T*T*T*\left(\Ric + T*T\right)  \nonumber \\
&= 2|T|^2(2T*(\Delta T + \Riem*T + T*T*T+ \del T*T*\g2 +\Riem*T*\psi ))+T*T*T*T*(\Ric+T*T) \nonumber \\
& \leq \Delta |T|^4-4|T|^2|\del T|^2+C|\Riem||T|^4+C|T|^6+C|\del T||T|^4, \label{eq:normTsch}
\end{align} 
where $C$ is some universal constant and we are using the fact that $|\g2|^2=|\psi|^2=7$.

We now compute the evolution of $|\del T|^2$. Recall that if a family of metrics is evolving by $\pt_t g =h$ then the Christoffel symbols evolve by
\begin{align*}
	\frac{\pt}{\pt t}\tensor{\Gamma}{^k_i_j}&= \frac 12 g^{kl}\left(\del_ih_{jl}+\del_jh_{il}-\del_lh_{ij}\right)
\end{align*}
and as a result, any  time-dependent tensor $Q(t)$ evolve by
\begin{align}\label{eq:Qtenevol}
\frac{\pt}{\pt t}\del Q = \del \frac{\pt}{\pt t}Q + Q* \frac{\pt}{\pt t}\Gamma,
\end{align}
and hence, using \eqref{eq:torevolsch}, we see that

\begin{align}
\frac{\pt}{\pt t} \del T&= \del \ptt T + T*(\del(\Ric + T*T)) \nonumber \\
&= \del (\Delta T + \Riem*T + T*T*T+ \del T*T*\g2 +\Riem*T*\psi) + T*(\del \Ric + 2\del T*T) \nonumber \\
&= \Delta \del T +  \del \Riem * T+\Riem * \del T+\del T*T*T+ \del^2T*T*\g2 + \del T*\del T*\g2 \nonumber \\
& \quad+ \del T*T*T*\psi + \del \Riem *T*\psi + \Riem * \del T*\psi + \Riem*T*T*\g2, \label{eq:nabTevolsch}
\end{align}
where we have used the Ricci identity \eqref{eq:ricciiden} to have
\begin{align*}
	\del \Delta T &=\Delta \del T+\Riem * \del T+ \del \Riem*T,
\end{align*}
and the facts  that $\del \g2 = T*\psi$ and $\del \psi = T*\g2$ from \eqref{eq:delphi} and \eqref{eq:delpsi} respectively.

We now use \eqref{eq:nabTevolsch} to compute
\begin{align}
\ptt |\del T|^2&= 2\left \langle \del T ,\ptt \del T\right \rangle + \del T*\del T*\ptt g^{-1} \nonumber \\
& = 2 \del T* \left( \Delta \del T +  \del \Riem * T+\Riem * \del T+\del T*T*T+ \del^2T*T*\g2 + \del T*\del T*\g2 \right. \nonumber \\
& \qquad \qquad \quad \left.+ \del T*T*T*\psi + \del \Riem *T*\psi + \Riem * \del T*\psi + \Riem*T*T*\g2\right) \nonumber \\
& \quad+ \del T*\del T* (\Ric + T*T) \nonumber \\
& \leq \Delta |\del T|^2 - 2|\del^2T|^2 + C|\del T||\del \Riem||T|+C|\del T|^2|\Riem| + C|\del T|^2|T|^2+|\del^2T|\del T||T|+ C|\del T|^3 \nonumber \\
& \quad +   C|\del T||\Riem||T|^2. \label{eq:normnabTevolsch}
\end{align}

We use \eqref{eq:Rimesqschemevol}, \eqref{eq:normTsch} and \eqref{eq:normnabTevolsch} to compute
\begin{align}
\ptt \Lambda(x,t)^2&= \Delta(|\Riem|^2+|\del T|^2+|T|^4) - 2|\del \Riem|^2 - 2|\del^2T|^2 + C|\Riem|^3+C|\Riem|^2|T|^2 + C|\Riem||\del  T|^2  \nonumber \\
& \quad+ C|\Riem||T||\del^2T|  -4|T|^2|\del T|^2+C|\Riem||T|^4+C|T|^6+C|\del T||T|^4 + C|\del T||\del \Riem||T| \nonumber \\
& \quad+ C|\del T|^2|T|^2+C|\del^2T|\del T||T|+ C|\del T|^3  +   C|\del T||\Riem||T|^2. \label{eq:dteaux1}
\end{align}

Notice that the terms $|\Riem|^3, |\del T|^3, |T|^6, |\Riem|^2|\del T|, |\Riem||\del T|^2, |\Riem|^2|T|^2, |\Riem||T|^4, |\del T|^2|T|^2, |\del T||T|^4\  \text{and}$ $|\Riem||\del T||T|^2$ are all bounded by some constant times $\Lambda^3 = (|\Riem|^2+|\del T|^2+|T|^4)^{\frac{3}{2}}$, and so we only need to take care  of the terms $C|\Riem||T||\del ^2T|, C|\del T||\del \Riem||T|$ and $C|\del^2T||\del T||T|$. We use the Young's inequality $ab\leq \frac{1}{2\varepsilon}a^2 + \frac{\varepsilon}{2}b^2$ for all $\varepsilon>0$ and $a, b\ge 0$ to estimate
\begin{align}
|\Riem||T||\del^2T|&\leq  	\frac{1}{2\varepsilon} |\Riem|^2|T|^2 + \frac{\varepsilon}{2}|\del^2T|^2, \label{eq:dteaux2} \\
|\del T||\del \Riem||T| & \leq \frac{1}{2\varepsilon} |\del T|^2|T|^2 + \frac{\varepsilon}{2}|\del \Riem|^2, \label{eq:dteaux3} \\
|\del^2T||\del T||T| & \leq \frac{1}{2\varepsilon}|\del T|^2|T|^2 + \frac{\varepsilon}{2}|\del^2T|^2. \label{eq:dteaux4}
\end{align} 

Using equations~\eqref{eq:dteaux2}-\eqref{eq:dteaux4} in \eqref{eq:dteaux1}, we get
\begin{align*}
\ptt \Lambda(x,t)^2 &\leq \Delta \Lambda(x,t)^2 + (C\varepsilon-2)(|\del \Riem|^2+|\del^2T|^2) + C\Lambda(x,t)^3,
\end{align*}
which on choosing $\varepsilon$ so that $C\varepsilon\leq 1$, yields
\begin{align}
\ptt \Lambda(x,t)^2 \leq \Delta \Lambda(x,t)^2 -(|\del \Riem|^2+|\del^2T|^2)+ C\Lambda (x,t)^3. \label{eq:dteaux5}
\end{align}

Since $\Lambda(t)= \underset{M}{\text{sup}} \Lambda(x,t)$ is a Lipschitz function, applying the maximum principle to \eqref{eq:dteaux5}, we get
\begin{align}
\ddt \Lambda (t) \leq \frac C2 \Lambda(t)^2,
\end{align} 
in the sense of the lim sup of forward difference quotients. We conclude that, for $t\leq \text{min} \left\{
\tau, \frac{2}{C\Lambda(0)}\right\}$, we have 
\begin{align}
\Lambda(t)\leq \dfrac{\Lambda(0)}{1-\frac 12 C\Lambda(0)t},
\end{align}
and hence $\Lambda(t)\leq 2\Lambda(0)$ for all $0\leq t\leq \delta$ if we take $\delta= \text{min} \left\{
\tau, \frac{1}{C\Lambda(0)}\right\}$.
\end{proof}

Thus, we see that the quantity $\Lambda(x,t)$ is well-behaved and cannot blow-up quickly along the Ricci-harmonic flow and hence the assumption of bounded $\Lambda(t)$ along the Ricci-harmonic flow is well-justified. In fact, the proof of the doubling-time estimate shows why we need to define $\Lambda$ by combining $|\Riem|^2, |\del T|^2$ and $|T|^4$ terms. Firstly, due to scaling reasons, the powers of the norm of the tensors involved are the correct one. Secondly, the terms like $|\del T||\Riem||T|^2, |\del T||\del \Riem||T|, |\Riem||T||\del^2T|$ and others in the evolution equations of $|\Riem|^2$ and $|\del T|^2$ are bad terms so combining them together allow us to use the good gradient terms $|\del \Riem|^2$ and $|\del^2T|^2$ to kill parts of the contribution. The remaining parts of contributions of such terms are subsumed in $\Lambda(x,t)^3$ term once we also include $|T|^4$ in the definition of $\Lambda$.

\medskip

Based on the computations done above for the evolution equations of the torsion and its covariant derivative and the Riemann curvature tensor, we can also prove global derivative estimates for solutions of the Ricci-harmonic flow, i.e., Shi-type estimates for the flow. The proof of the following theorem is inspired from the derivative estimates  for the Ricci flow by Shi \cite{Shi} and Hamilton \cite[Thm. 7.1]{hamilton-singularities} and is similar, in the $\G2$-specific case, to similar estimates for the Laplacian flow of closed $\G2$-structures by \cite[Thm. 4.2]{lotay-wei-gafa} or for the isometric flow \cite[Thm. 3.3]{dgk-isometric}.

\begin{theorem}\label{thm:shiest}
Suppose that  $K>0$ is a constant and $\g2(t)$ is  a solution to the Ricci-harmonic flow on a closed manifold $M^7$ with $t\in \left[0, \frac 1K\right]$. For all $m\in \mathbb{N}$, there exists a constant $C_m$ such that  if 
\begin{align}\label{eq:shiestassum}
\Lambda(x,t)\leq K \ \ \ \ \text{on}\ M\times \left[0,\frac 1K\right],
\end{align}
then for all $t\in \left[0, \frac 1K\right]$ we have
\begin{align}\label{eq:shiestconc}
|\del^{m}\Riem|+|\del^{m+1}T|\leq C_m t^{-\frac m2}K.
\end{align}
\end{theorem}

\begin{proof}
As in the Ricci flow case, the proof is by induction on $m$. The idea of the proof is to define a suitable function which satisfies a parabolic differential inequality and then apply the maximum principle.  Since the proof is similar to derivative estimates for other geometric flows, with only the actual computations being different which  depend on the evolution of the torsion and its derivatives and the Riemann curvature tensor and its derivatives which we have explicitly derived, we only show the details for the base case in the induction step. The general case of induction is done for the modified local Shi-type estimates in Theorem~\ref{thm:strlocder} and is similar to the case here.  

\medskip

For $m=1$ case, we define the function
\begin{align}
f = t(|\del \Riem|^2+|\del^2T|^2)+ \beta(|\Riem|^2+|\del T|^2+|T|^4), \label{eq:shiestf1}
\end{align}
where $\beta$ is a constant to be determined later. Note that from he assumption \eqref{eq:shiestassum}, we have $f(x,0)\leq \beta K^2$. We want to compute the evolution of $f$ and so we compute the evolution of $|\del \Riem|^2$ and $\del^2T|^2$. 

We use \eqref{eq:RHFgevol} and \eqref{eq:Riemevolschem} to compute the evolution of $|\del \Riem|^2$. Using \eqref{eq:Qtenevol}, we have
\begin{align}
\ptt \del \Riem&= \del \ptt \Riem + \Riem * \ptt \Gamma \nonumber \\
& = \del (\Delta \Riem + \Riem * \Riem + \Riem * T*T+ \del T* \del T+ \del^2T * T) + \Riem *(\del (\Riem + T*T)) \nonumber \\
&= \Delta \del \Riem + \del \Riem*\Riem + \del \Riem*T*T+ \Rm*\del T*T+ \del^2T*\del T+ \del^3T*T, \label{eq:nabRmsch}
\end{align}
where we used the Ricci identity \eqref{eq:ricciiden} in the third equality to write
\begin{align*}
\del \Delta \Rm = \Delta \del \Rm + \del \Rm*\Rm.
\end{align*}

We use \eqref{eq:nabRmsch} to compute
\begin{align}
\ptt |\del \Rm|^2 &= 2\left\langle \del \Rm, \ptt \del \Rm\right\rangle + \del \Rm*\del \Rm * \ptt g^{-1} \nonumber \\
&= \del \Rm*\left(\Delta \del \Riem + \del \Riem*\Riem + \del \Riem*T*T+ \Rm*\del T*T+ \del^2T*\del T+ \del^3T*T\right) \nonumber \\
& \quad + \del \Rm*\del \Rm * \Rm + \del \Rm*\del \Rm*T*T \nonumber \\
& \leq \Delta |\del \Rm|^2-2|\del^2 \Rm|^2+ C|\del \Rm|^2|\Rm| + C|\del \Rm|^2|T|^2+ C|\del \Rm||\Rm||\del T||T| \nonumber \\
& \quad + C|\del \Rm||\del^2T||\del T|+ C|\del \Rm||\del^3T||T|. \label{eq:nabRmsqsch}
\end{align}

Similarly, we use \eqref{eq:Qtenevol} and \eqref{eq:nabTevolsch} to compute
\begin{align}
\ptt \del^2T & = \del \ptt \del T+ \del T*\ptt \Gamma \nonumber \\
& = \del \left(\Delta \del T +  \del \Riem * T+\Riem * \del T+\del T*T*T+ \del^2T*T*\g2 + \del T*\del T*\g2 \right.\nonumber \\
& \qquad\quad \left.+ \del T*T*T*\psi \right. \left.+ \del \Riem *T*\psi + \Riem * \del T*\psi + \Riem*T*T*\g2\right) + \del T*\del \Rm \nonumber \\
& \quad +\del T*\del T*T \nonumber \\
\begin{split}
&= \Delta \del^2T + \del \Rm*\del T+ \Rm*\del^2T + \del^2\Rm *T+ \del^2T*T*T + \del T*\del T*T + \del^3T*T*\g2 \nonumber \\
& \quad + \del^2T*\del T* \g2 + \del^2T*T*T*\psi + \del T*\del T*T*\psi + \del T*T*T*T*\g2 + \del^2\Rm *T*\psi \nonumber \\
& \quad + \del \Rm *\del T*\psi + \del \Rm*T*T*\g2 + \Rm*\del^2T*\psi + \Rm*\del T*T*\g2+\Rm*T*T*T*\psi,
\end{split}
\end{align}
where we again used the Ricci identity for the $\del \Delta \del T$ term. Using this, we find
\begin{align}
\ptt |\del^2T|^2 & \leq \Delta |\del^2T|2-2|\del^3T|^2 + C|\del^2T||\del \Rm||\del T|+C|\del^2T|^2|\Rm|+C|\del^2T||\del^2\Rm||T| +|\del^2T|^2|T|^2 \nonumber \\
& \quad +C|\del^2T||\del T|^2|T| ++C|\del^2T||\del^3 T||T|++C|\del^2T|^2|\del T|++C|\del^2T||\del T||T|^3  \nonumber \\
& \quad +C|\del^2T||\del \Rm||T|^2 + +C|\del^2T||\del T||\Rm||T|+C|\del^2T||\Rm||T|^3. \label{eq:nab2Tsqsch}  
\end{align}

We use Young's inequality to estimate some of the terms on the right hand side of the differential inequality \eqref{eq:nabRmsqsch} and \eqref{eq:nab2Tsqsch}. For all $\varepsilon>0$, the fifth, sixth and the seventh term on the RHS of \eqref{eq:nabRmsqsch} can be estimated as
\begin{align*}
2|\del \Rm||\Rm||\del T||T|&\leq |\del \Rm||T|(|\Rm|^2+|\del T|^2), \\
2|\del \Rm||\del^2T||\del T| & \leq  |\del T|(|\del \Rm|^2+|\del^2T|^2),\\
2|\del \Rm||\del^3T||T|&\leq \frac{1}{\varepsilon}|\del \Rm|^2|T|^2+ \varepsilon |\del^3T|^2.	
\end{align*}

Similarly, for all $\varepsilon>0$, the third and eleventh terms, fifth term, eighth term, tenth, twelfth and the thirteenth terms on the RHS of \eqref{eq:nab2Tsqsch} can be estimated (in that order) as
\begin{align*}
2|\del \Rm||\del^2T|(|\del T|+|T|^2)& \leq(|\del \Rm|^2+|\del^2T|^2)(|\del T|+|T|^2), \\
2|\del^2T||\del^2\Rm||T|& \leq \frac{1}{\varepsilon}|\del^2T|^2|T|^2+\varepsilon |\del^2\Rm|^2,\\
2|\del^2T||\del^3T||T|& \leq \frac{1}{\varepsilon}|\del^2T|^2|T|^2+\varepsilon |\del^3T|^2,\\
2|\del^2T||\del T||T|^3&\leq |\del^2T||T|(|\del T|^2+|T|^4),\\
2|\del ^2T||\del T||\Rm||T| & \leq |\del^2T||T|(|\Rm|^2+|\del T|^2),\\
2|\del^2T||\Rm||T|^3& \leq |\del^2T||T|(|\Rm|^2+|T|^4).
\end{align*} 

We can combine \eqref{eq:nabRmsqsch} and  \eqref{eq:nab2Tsqsch}, along with all the estimates using the Young's inequality above and suitably chosen $\varepsilon>0$, to obtain
\begin{align}
\ptt (|\del \Rm|^2|+|\del^2T|^2)&\leq \Delta(|\del \Rm|^2+|\del^2T|^2)-(|\del^2\Rm|+|\del^3T|^2) \nonumber \\
& \quad +C(|\del \Rm|^2+|\del^2T|^2)(|\Rm|+|\del T|+|T|^2) \nonumber \\
& \quad + C|T|(|\del \Rm|+|\del^2T|)(|\Rm|^2+|\del T|^2+|T|^4). \label{eq:nabRmTsqsch}
\end{align}

Having obtained these evolution equations and recalling the definition of the function $f$ from \eqref{eq:shiestf1}, in combination with \eqref{eq:dteaux5}, allows us to compute
\begin{align}
\ptt f &\leq \Delta f+ (1-\beta) (|\del \Rm|^2+|\del^2T|^2) +Ct(|\del \Rm|^2+|\del^2T|^2)(|\Rm|+|\del T|+|T|^2) \nonumber \\
& \quad + Ct|T|(|\del \Rm|+|\del^2T|)(|\Rm|^2+|\del T|^2+|T|^4) \nonumber \\
& \quad + C\beta (|\Rm|^2+|\del T|^2+|T|^4)^{\frac 32}. \label{eq:fevol1}
\end{align}
Note that we have the  assumption of $\Lambda(t)=\underset{M}{\text{sup}}\ \Lambda(x,t) \leq K$ and $t\in \left[0, \frac 1K\right]$ and hence $tK\leq 1$. We use the previous fact, the definition of $\Lambda$ and again use the Young's inequality for the third and the fourth term on the RHS of \eqref{eq:fevol1} to get
\begin{align*}
\ptt f\leq \Delta f + (C-\beta)(|\del \Rm|^2+|\del^2T|^2)+C\beta K^3,
\end{align*}
which on choosing $\beta$ sufficiently large gives
\begin{align}
\ptt f\leq \Delta f+C\beta K^3.
\end{align}

Applying the maximum principle to the above inequality implies that 
\begin{align*}
\underset{M}{\text{sup}}\ f(x,t) \leq f(x,0)+C\beta K^3t \leq \beta K^2+C\beta K^2 \leq CK^2.
\end{align*}

Thus, we obtain
\begin{align*}
|\del \Rm|+|\del^2T|\leq C_1t^{-\frac 12}K
\end{align*}
which proves \eqref{eq:shiestconc} for $m=1$ and the base case of the induction is done.

We don't write the estimates for completing the induction step as it is similar to Shi-type estimates for other flows of $\G2$-structures, for instance, see \cite{lotay-wei-gafa} for the case of the Laplacian flow for closed $\G2$-structures. Below, we derive the expressions for the evolution of the $m$-th order derivative of the Riemann curvature tensor and the $m+1$-th order derivative of the torsion which are to be used for completing the induction step and will also be used in the proof of Theorem~\ref{thm:strlocder}. Since for any time-dependent tensor $Q(t)$, we have
\begin{align}
\ptt \del^mQ-\del^m \ptt Q&= \sum_{i=1}^{m} \del^{m-i}Q*\del^i \ptt g, \label{eq:Qdelmevol}
\end{align}
we use \eqref{eq:RHFgevol}, \eqref{eq:Riemevolschem} and the Ricci identity, to get
\begin{align}
\ptt \del^m\Rm&=\del^m \ptt \Rm + \sum_{i=1}^{m} \del^{m-i}\Rm * \del^i (\Ric + T*T)\nonumber \\
&= \del^m (\Delta \Riem + \Riem * \Riem + \Riem * T*T+ \del T* \del T+ \del^2T * T)  + \sum_{i=1}^{m} \del^{m-i}\Rm * \del^i (\Ric + T*T)\nonumber \\
&=\Delta \del^m\Rm + \sum_{i=0}^{m} \del^{m-i}\Rm * \del^i(\Rm +T*T)+\sum_{i=0}^{m+1}\del^iT*\del^{m+2-i}T. \label{eq:delmRmevol}
\end{align}
Using the previous equation, we get
\begin{align}
\ptt |\del^m\Rm|^2&= \Delta |\del^m\Rm|^2-2|\del^{m+1}\Rm|^2+\sum_{i=0}^{m} \del^m\Rm*\del^{m-i}\Rm * \del^i(\Rm +T*T) \nonumber \\
& \quad +\sum_{i=0}^{m+1}\del^m\Rm*\del^iT*\del^{m+2-i}T.\label{eq:delmRmsqevol}
\end{align}

\noindent
Similarly, for the torsion tensor we use \eqref{eq:torevolsch}, to get
\begin{align*}
\ptt \del^{m+1}T&= \del^{m+1}\ptt T + \sum_{i=1}^{m+1}\del^{m+1-i}T*\del^i\ptt g \\
&= \Delta \del^{m+1}T+ \sum_{i=0}^{m+1} \del^{m+1-i}T*\del^i\Rm + \sum_{i=0}^{m+1}\del^{m+1-i}T*\del^i(T*T) \\
& \quad + \sum_{i=0}^{m+1}\del^{m+1-i}(\Rm*T)*\del^i\psi + \sum_{i=0}^{m+1}\del^{m+1-i}(\del T*T)*\del^i\g2
\end{align*}
and hence

\begin{align}
\ptt |\del^{m+1}T|^2&=\Delta|\del^{k+1}T|^2-2|\del^{k+2}T|^2 + \sum_{i=0}^{m+1}\del^{m+1}T*\del^{m+1-i}T*\del^i(\Rm +T*T) \nonumber \\
& \quad + \sum_{i=0}^{m+1}\del^{m+1}T*\del^{m+1-i}(\Rm*T)*\del^i\psi + + \sum_{i=0}^{m+1}\del^{m+1}T*\del^{m+1-i}(\del T*T)*\del^i\g2. \label{eq:delmTsqevol}
\end{align}
Finally, for completing the proof, we would also need the expression of $\del^m\psi$ and $\del^m\g2$ which can be obtained from successively differentiating \eqref{eq:delpsi} and \eqref{eq:delphi}, respectively and using the induction hypothesis.
\end{proof}

\subsection{Modified local Shi-type estimates}\label{subsec:lmodshi}
In this section, we prove our modified local Shi-type estimates for the Ricci-harmonic flow. Since the Ricci-harmonic flow is a reasonable flow of $\G2$-structures we obtain local derivative estimates for the solutions of the flow from \cite[Thm. 2.1]{gaochen-shi} (and we have also computed all the evolution equations required to derive them). We state the result below and state and prove a stronger form of the result which is Theorem~\ref{thm:strlocder}.

\begin{theorem}\label{thm:locderest}
Let $\g2(t)$ be a solution to the Ricci-harmonic flow on $[0,\tau]$. Let $B_{g(0)}(p, r)$ be the ball of radius $r$ centred at the point $p\in M$. If 
\begin{align}\label{eq:locderesthyp}
\left(|\Rm|^2+|\del T|^2+|T|^4\right)^{\frac 12} \leq K \ \ \ \ \text{on}\ \ \  B_{g(0)}(p, r)\times [0,\tau],
\end{align}
then
\begin{align}
|\del^m\Rm| + |\del^{m+1}T| \leq C(m, r, K)t^{-\frac m2} \ \ \ \ \text{on} \ \ \ B_{g(0)}(p, \frac r2) \times (0,\tau]. \label{eq:locderestconc}
\end{align}
\end{theorem}

\medskip

We now state a stronger form of the previous theorem which has the latter as a special case. The assumptions for the modified local derivative estimates are stronger, in particular, we assume bounds on higher order derivatives of the Riemann curvature tensor and the torsion \emph{for the initial $\G2$-structure}, and as a result the conclusions are also stronger and the higher order derivatives of $\Rm$ and $T$ are then bounded including at the time $t=0$. The modified local Shi-type estimates will allow us to prove a modified Cheeger--Gromov--Hamilton compactness theorem for the flow (Theorem~\ref{thm:modcomp}), which we believe would be useful in further analysis of the singularities of the Ricci-harmonic flow in the same way as similar result is immensely useful for the Ricci flow. The theorem and its proof below is based on ideas of Peng Lu for the case of the Ricci flow, see \cite[Thm. 14.16]{chow-analytic} and \cite[Theorem 3.29]{morgan-tian}. 

\begin{theorem}\label{thm:strlocder}
Let $\g2(t), t\in [0, \tau]$ be a solution to the Ricci-harmonic flow \eqref{heatfloweqn} on $M^7$. Let $a_+=\text{max}\{a,0\}$. For any integers $l\geq 0$ and $m\geq 1$ and any positive numbers $\alpha, K, K_l$ and $r$, there exists constant $C<\infty$ with $C=C(\alpha, K, K_l, r, l, m)$ such that if $p\in M$ and $0< \tau\leq \frac{\alpha}{K}$ and if
\begin{align}
|Rm(x, t)|+|\del T(x,t)|+|T(x,t)|^2 \leq K \ \ \ for\ all\ x\in B_{g(0)}(p,r)\ \ and\ \ t\in [0,\tau], \label{eq:modderesthyp1} \\
|\del^{\beta} \Rm (x,0)|+|\del^{\beta+1}T(x,0)|\leq K_l\ \ \ for\ all\ x\in B_{g(0)}(p,r)\ \ and\ \ \beta \leq l, \label{eq:modderesthyp2}
\end{align}
then
\begin{align}\label{eq:modderestconc}
|\del^m\Rm(y,t)|+|\del^{m+1}T(y,t)| \leq  \frac{C}{t^{\frac{(m-l)_+}{2}}},
\end{align}
for all $y\in \bar{B}_{g(0)}(p, \frac r2)$ and $t\in (0, \tau]$. In particular, if $m\leq l$, then we have the uniform bound
\begin{align}\label{eq:locunibound}
|\del^m\Rm(y,t)|+|\del^{m+1}T(y,t)| \leq C
\end{align}
on $\bar{B}_{g(0)}(p, \frac r2) \times [0, \tau]$.
\end{theorem}

\begin{remark}
Clearly the assumption \eqref{eq:modderesthyp2} is stronger than the assumptions in the local derivative estimates Theorem~\ref{thm:locderest} of Chen but the conclusion \eqref{eq:locunibound} of uniform bounds are also stronger. \demo
\end{remark}

\begin{remark}
We also notice that $l=0$ is precisely Theorem~\ref{thm:locderest} and so Theorem~\ref{thm:strlocder} is somewhat a generalization of the former theorem. \demo
\end{remark}

\begin{proof}
The proof again is by induction on $m$ but to complement the proof of the global derivative estimates in Theorem~\ref{thm:shiest}, we prove the induction step only. First we notice that proving Theorem~\ref{thm:strlocder} for one value of $r$ implies it for all $r'>2r$ because if $y\in B_{g(0)}(p, \frac{r'}{2})$ and $B_{g(0)}(y,r)\subset B_{g(0)}(p, r')$ then a curvature bound on $B_{g(0)}(p, r')$ will imply one on $B_{g(0)}(y,r)$ and hence will imply higher derivative estimates at the point $y$.

We also recall (see also \cite[eq. (2.33)]{gaochen-shi}) that the right functions for proving local derivatives estimates for the Ricci-harmonic flow are
\begin{align}\label{eq:Fmdefn}
F_m =(C+t^m(|\del^m\Rm|^2+|\del^{m+1}T|^2))t^{m+1}((|\del^{m+1}\Rm|^2+|\del^{m+2}T|^2)).
\end{align}
The idea to prove the theorem then is to suitably modify the functions $F_m$ in \eqref{eq:Fmdefn} by lowering the powers of $t$ but still keeping them non-negative. We define the functions
\begin{align}\label{eq:Fmldefn}
F_{m,l}&= \left(\tilde{C}+t^{(m-l)_+}(|\del^m\Rm|^2+|\del^{m+1}T|^2)\right)t^{(m-l+1)_+}(|\del^{m+1}\Rm|^2+|\del^{m+2}T|^2),
\end{align}
where $\tilde{C}$ is a constant which will be chosen later.

As in the Ricci flow case, the main inequality which we need to prove the theorem is given by the following Lemma.

\begin{lemma}\label{lem:modshie}
Let $l\in \mathbb{N}\cup \{0\}$ be a fixed integer and assume that by induction on $m$, we have
\begin{align}
|\del^j\Rm(y,t)|+|\del^{j+1}T(y,t)| \leq \frac{C}{t^{\frac{(j-l)_+}{2}}}, \label{eq:lemhyp}
\end{align}
for all $y\in B_{g(0)}(p, \frac{r}{2^j})$, $t\in (0, \tau]$ and $j=1,\ldots, m$. Then for $\tilde{C}< \infty$ sufficiently large, there exist constants $c>0$ and $C<\infty$ such that
\begin{align}\label{eq:lemconc}
\ptt F_{m,l}\leq \Delta F_{m,l}- \frac{c}{t^{\sgn((m-l+1)_+)}}(F_{m,l})^2 + \frac{C}{t^{\sgn((m-l+1)_+)}},	
\end{align}	
where 
\begin{align*}
 \sgn(x) = \begin{cases} 
	1 & x> 0, \\
	0 & x=0, \\
	-1 & s<0, 
\end{cases}
\end{align*}
is the signum function.
\end{lemma}

Since the proof of the lemma is quite lengthy and involved, we postpone it for later and we first prove the theorem assuming Lemma~\ref{lem:modshie}. Let $m\in \mathbb{N}$ and we assume the induction hypothesis in the Lemma, i.e., \eqref{eq:lemhyp} holds and hence the functions $F_{m,l}$ satisfy the differential inequality \eqref{eq:lemconc}. We notice that if we bound the function $F_{m,l}$ then we are done. Let $U$ be an open set in $M$ containing the ball $B_{g(0)}(p, \frac{r}{2^{m+1}})$ and let $\eta:U\rightarrow [0,1]$ be a cut-off function which satisfy
\begin{enumerate}
	\item $\eta=1$ on $B_{g(0)}(p, \frac{r}{2^{m+1}})$.
	\item supp $(\eta)\subset B_{g(0)}(p, \frac{r}{2^{m}})$.
	\item $\eta^{-1}|\del \eta|^2_{g(t)}-\Delta_{g(t)}\eta \leq \bar{C}$ for some $\bar{C}< \infty$. 
\end{enumerate}
Such a cut-off function exists, for a proof, see for instance \cite[Chapter 14, \textsection 1.3]{chow-analytic}. Now we divide the proof into two cases.

\noindent
{\bf{Case 1.}} If $m-l+1\leq 0$ then $\sgn(m-l+1)_+=0$ and hence \eqref{eq:lemconc}, in this case, implies
\begin{align*}
\left(\ptt -\Delta\right)F_{m,l}\leq -c(F_{m,l})^2+C.
\end{align*}
Thus, we have
\begin{align}
\left(\ptt -\Delta\right)(\eta F_{m,l})&\leq \eta\left(-c(F_{m,l})^2+C\right) -F_{m,l}\Delta \eta-2\del \eta\cdot \del F_{m,l}. \label{eq:thmmodshi1}
\end{align}
Let  $(x_0,t_0)\in B_{g(0)}(p, \frac{r}{2^{m+1}})\times [0,\tau]$ where the function $\eta F_{m,l}$ attains its maximum. Such a point exists because the ball $B_{g(0)}(p, \frac{r}{2^{m+1}})\subset B_{g(0)}(p, \frac{r}{2^m})$ and hence the former has compact closure in the open set $U$, and also because of our assumptions on the boundedness of the derivatives of $\Rm$ and $T$ of the initial metric. If $t_0=0$ then $\eta F_{m,l}\leq (\tilde{C}+K_l^2)K_l^2$ and hence the estimate follows. If $t_0>0$ then by the first derivative test, we must have $\del (\eta\cdot F_{m,l})=F_{m,l}\del \eta + \eta \del F_{m,l}=0$ at $(x_0, t_0)$ and $\left(\ptt-\Delta\right)\eta F_{m,l}\geq 0$ at $(x_0, t_0)$ and hence, at $(x_0,t_0)$, we have
\begin{align*}
0&\leq -c\eta( F_{m,l})^2+C\eta-F_{m,l}\Delta \eta +2\del \eta \cdot \eta^{-1}F_{m,l}\del \eta \\
&=-c(\eta F_{m,l})^2+C\eta^2+\left(-\Delta \eta  +2\eta^{-1}|\del \eta|^2\right)\eta F_{m,l}
\end{align*}
and since $(-\Delta \eta +2\eta^{-1}|\del \eta|^2)$ is bounded by the choice of the cut-off function, we obtain a bound for $\eta F_{m,l}$ and thus for $F_{m,l}$ on the ball, thereby completing the proof of the theorem in this case.

\medskip

\noindent
{\bf{Case 2.}} If $m-l+1>0$ then $\sgn(m-l+1)_+=1$ and \eqref{eq:lemconc} in this case implies the differential inequality
\begin{align*}
\left(\ptt -\Delta\right)F_{m,l}\leq -\frac{c}{t}(F_{m,l})^2+\dfrac{C}{t}.
\end{align*}
As a result, we have
\begin{align*}
\left(\ptt -\Delta\right)\eta F_{m,l}\leq -\frac{c\eta}{t}(F_{m,l})^2+\dfrac{C\eta}{t}-F_{m,l}\Delta \eta -2\del \eta\cdot \del F_{m,l}	
\end{align*}
which  on multiplying both sides by $\eta$ gives,
\begin{align*}
\eta \left(\ptt -\Delta\right)\eta F_{m,l}\leq -\frac{c}{t}(\eta F_{m,l})^2+\dfrac{C\eta^2}{t}+(-\Delta \eta +2\eta^{-1}|\del \eta|^2) \eta  F_{m,l} -2\del \eta \cdot \del (\eta F_{m,l}).	
\end{align*}
Again looking at a point $(x_0,t_0)$ where $\eta F_{m,l}$ attains its maximum and assuming $t_0>$ (as $t_0=0$ case is same as in Case 1), we get
\begin{align*}
0&\leq -c(\eta F_{m,l})^2+C\eta^2+t(-\Delta \eta +2\eta^{-1}|\del \eta|^2) \eta  F_{m,l}, 
\end{align*}
which again by the choice of the cut-off function allows us to conclude that $\eta F_{m,l}$ is bounded. As a result.
\begin{align*}
\eta t^{(m-l+1)_+}(|\del^{m+1}\Rm|^2+|\del^{m+2}T|^2) \leq C\ \ \text{on}\ \ B_{g(0)}(p, \frac{r}{2^m})\times [0,\tau],
\end{align*}
and since $\eta=1$ on $B_{g(0)}(p, \frac{r}{2^{m+1}})$, we conclude that
\begin{align*}
|\del^{m+1}\Rm|^2+|\del^{m+2}T|^2\leq \dfrac{C}{t^{(m-l+1)_+}}\ \ \text{on}\ \ B_{g(0)}(p, \frac{r}{2^{m+1}})\times (0, \tau],
\end{align*}
which completes the proof of the theorem.
\end{proof} 

We now prove the pending Lemma.

\noindent
{\bf{Proof of Lemma~\ref{lem:modshie}}}. Let $l\geq 0$ be fixed and assume that by induction we have, for $j=1,\ldots, m$, constant $C_j=C_j(\alpha, K, K_l, l, r, m)$ such that for all $x\in B_{g(0)}(p, \frac{r}{2^j})$ and $t\in [0,\tau]$, we have
\begin{align*}
t^{\frac{(j-l)_+}{2}}|\del^j\Rm|+|\del^{j+1}T|\leq C_j. 
\end{align*} 
This is a long and tedious computation so we proceed step by step. The computations and ideas are inspired from the Ricci flow case in \cite{chow-analytic} and \cite{morgan-tian} but here we have extra complications due to the torsion terms so we need to estimate each term carefully. We have from \eqref{eq:delmRmsqevol} and \eqref{eq:delmTsqevol},
\begin{align}
\ptt \left(t^{(k-l)_+}(|\del^k\Rm|^2+|\del^{k+1}T|^2)\right)&\leq \Delta(t^{(k-l)_+}(|\del^k\Rm|^2+|\del^{k+1}T|^2))-2t^{(k-l)_+}(|\del^{k+1}\Rm|^2+|\del^{k+2}T|^2)\nonumber \\
& +Ct^{(k-l)_+}\sum_{i=0}^{k} |\del^k\Rm||\del^{k-i}\Rm| |\del^i(\Rm +T*T)| \nonumber \\ 
& +Ct^{(k-l)_+}\sum_{i=0}^{k+1}|\del^k\Rm||\del^iT||\del^{k+2-i}T| \nonumber \\
&+ Ct^{(k-l)_+}\sum_{i=0}^{k+1}|\del^{k+1}T||\del^{k+1-i}T||\del^i(\Rm +T*T)| \nonumber \\
&  +Ct^{(k-l)_+} \sum_{i=0}^{k+1}|\del^{k+1}T||\del^{k+1-i}(\Rm*T)||\del^i\psi |\nonumber \\
&  +Ct^{(k-l)_+} \sum_{i=0}^{k+1}|\del^{k+1}T||\del^{k+1-i}(\del T*T)||\del^i\g2| \nonumber \\
&+ (k-l)_+t^{(k-l)_+-1}(|\del^k\Rm|^2+|\del^{k+1}T|^2). \label{eq:modshieaux1}
\end{align}
We let $m_l=(m-l+1)_+$ and $\widehat{m}_l=(m-l)_+$. Note that if $m_l\neq 0$ then $m_l-\widehat{m}_l=1$ and if $m_l=0$ then $m_l=\widehat{m}_l=0$. We first put $k=m+1$ in the above equation and estimate each term step by step by keeping the followings in mind: 
\begin{enumerate}
	\item in all the estimates below, we will use the induction hypothesis \eqref{eq:lemhyp} for $0\leq j\leq m$ and we will keep using the letter $C$ for different constants as long as they depend only on $\alpha, K, K_l, l, r$ and $j$ when we apply the induction hypothesis \eqref{eq:lemhyp}, 
	\item we want to apply Lemma~\ref{lem:modshie} to obtain Theorem~\ref{thm:locderest} and hence we will assume \eqref{eq:modderesthyp1} and \eqref{eq:modderesthyp2}
	 \item  for any $t\in [0,\tau]$, $tK < \alpha$ where $K$ is the constant in the assumption \eqref{eq:modderesthyp1}. 
	\end{enumerate}
	We now estimate each term on the right hand side of \eqref{eq:modderesthyp2}. We start with 

\begin{align}
\text{\bf{3rd\  term}} &= Ct^{\ml}\Big(|\del^{m+1}\Rm|^2(|\Rm|+|T|^2)+|\Rm||\del^{m+1}\Rm||\del^{m+1}(\Rm +T*T) \nonumber \\
& \qquad  \qquad + |\del^{m+1}\Rm|\sum_{i=1}^{m}|\del^{m+1-i}\Rm||\del^i(\Rm +T*T)|\Big)\nonumber \\
&\leq Ct^{m_l}\Big(K|\del^{m+1}\Rm|^2 +|\del^{m+1}\Rm||\Rm||T||\del^{m+1}T| +C|\del^{m+1}\Rm||\Rm| \nonumber \\
& \qquad \qquad + |\Rm||\del^{m+1}\Rm|\sum_{i=0}^{m+1}|\del^{m+1-i}T||\del^{i}T| + |\del^{m+1}\Rm|t^{-\frac{m_l}{2}}\Big)\nonumber \\
& \leq Ct^{m_l}K|\del^{m+1}\Rm|^2+Ct^{\frac{m_l}{2}}|\del^{m+1}\Rm| \nonumber\\
& \leq Ct^{\mhatl}|\del^{m+1}\Rm|^2+Ct^{\frac{m_l}{2}}|\del^{m+1}\Rm|. \label{eq:modshiaux2}
\end{align}

\begin{align}
\text{{\bf{4th\ term}}}&=Ct^{m_l}\sum_{i=0}^{m+2}|\del^{m+1}\Rm||\del^iT||\del^{m+3-i}T| \nonumber \\
&=Ct^{m_l}\Big(|\del^{m+1}\Rm||T||\del^{m+3}T|+|\del^{m+1}\Rm||\del^{m+2}T||\del T|+ |\del^{m+1}\Rm|\sum_{i=2}^{m+1}|\del^{i}T||\del^{m+3-i}T|\Big) \nonumber \\
& \leq Ct^{m_l}K^{\frac 12}|\del^{m+1}\Rm||\del^{m+3}T|+Ct^{\mhatl}(|\del^{m+1}\Rm|^2+|\del^{m+2}T|^2)+Ct^{\frac{m_l}{2}}|\del^{m+1}\Rm|, \label{eq:modshiaux3}
\end{align}
where we used Young's inequality in the last inequality.

\noindent
We now move on to the 5th term on the right hand side of \eqref{eq:modshieaux1}. We have
\begin{align}
\text{{\bf{5th\ term}}}&=Ct^{m_l}\sum_{i=0}^{m+2}|\del^{m+2}T||\del^{m+2-i}T||\del^i(\Rm +T*T)| \nonumber \\
&\leq Ct^{m_l}\Big( K|\del^{m+2}T|^2+|\del^{m+2}T||T||\del^{m+2}(\Rm+T*T)|+|\del^{m+2}T||\del T||\del^{m+1}(\Rm +T*T)| \nonumber \\
& \qquad \quad \quad  + \sum_{i=1}^{m}|\del^{m+2}T||\del^{m+2-i}T||\del^i\Rm| + \sum_{i=1}^{m}|\del^{m+2}T||\del^{m+2-i}T||\del^{i}(T*T)|\Big) \nonumber \\
& \leq  Ct^{m_l}\Big(K|\del^{m+2}T|^2  + |T||\del^{m+2}T||\del^{m+2}\Rm| +|T|^2|\del^{m+2}T|^2+t^{-\frac{m_l}{2}}|\del^{m+2}T|\nonumber \\
& \qquad \qquad +C|\del T||\del^{m+2}T||\del^{m+1}\Rm|\Big) \nonumber \\
& \leq Ct^{\mhatl}|\del^{m+1}T|^2 + Ct^{m_l}|T||\del^{m+2}T||\del^{m+2}\Rm|+Ct^{\frac{m_l}{2}}|\del^{m+2}T|+Ct^{\mhatl}|\del^{m+2}T||\del^{m+1}\Rm|. \label{eq:modshiaux4}
\end{align}

\noindent
For the 6th and the 7th term we need the expressions for $\del^{i}\psi$ and $\del^i\g2$. Both of these can be expressed in terms of torsion and its derivatives by using induction and starting from \eqref{eq:delpsi} and \eqref{eq:delphi}. The expression in the context of proving Shi-type estimates has already been derived by Lotay--Wei and we use \cite[eq. (4.29)]{lotay-wei-gafa} with appropriate changes to the estimate using our induction hypothesis \eqref{eq:lemhyp} to get,
\begin{align*}
|\del^i\psi| \leq C_i\sum_{j=0}^{i-2}t^{\frac{j-i+2}{2}}.
\end{align*}
Thus, we get
\begin{align}
{\text{\bf{6th term}}}&=Ct^{m_l} \sum_{i=0}^{m+2}|\del^{m+2}T||\del^{m+2-i}(\Rm*T)||\del^i\psi| \nonumber \\
&=Ct^{m_l}\Big(|\del^{m+2}T||\del^{m+2}(\Rm+T*T)|\del \psi| + |\del^{m+2}T||\del^{m+1}(\Rm +T*T)|\del \psi|\nonumber \\
& \qquad \qquad +\sum_{i=2}^{m+2}|\del^{m+2}T||\del^{m+2-i}(\Rm + T*T)||\del^i\psi|\Big) \nonumber \\
& \leq Ct^{m_l}\Big(|T||\del^{m+2}T||\del^{m+2}\Rm|+|\del^{m+2}T||\del^{m+1}\Rm||\del T|+t^{-\frac{m_l}{2}}|\del^{m+2}T|\nonumber \\
&\qquad \qquad +|T|^2|\del^{m+2}T||\del^{m+1}\Rm| \Big) \nonumber \\
&\leq Ct^{m_l}|T||\del^{m+2}T||\del^{m+2}\Rm|+Ct^{\mhatl}|\del^{m+2}T||\del^{m+1}\Rm|+Ct^{\frac{m_l}{2}}|\del^{m+2}T|. \label{eq:modshiaux5}
\end{align}
The 7th term can also be estimated in the same manner and the right hand side of the inequality looks exactly the same as the right hand side of \eqref{eq:modshiaux5}. We now use eqs. \eqref{eq:modshiaux2}-\eqref{eq:modshiaux5} in \eqref{eq:modshieaux1} to get
\begin{align}
\ptt \left(t^{m_l}(|\del^{m+1}\Rm|^2+|\del^{m+2}T|^2)\right)&\leq \Delta(t^{m_l}(|\del^{m+1}\Rm|^2+|\del^{m+2}T|^2))-2t^{m_l}(|\del^{m+2}\Rm|^2+|\del^{m+3}T|^2)\nonumber \\
& \quad +\underbrace{Ct^{m_l}K|\del^{m+1}\Rm|^2}_{\text{A}} + \underbrace{Ct^{\frac{m_l}{2}}|\del^{m+1}\Rm|}_{\text{B}}+\underbrace{Ct^{m_l}|T||\del^{m+1}\Rm||\del^{m+3}T|}_{\text{C}} \nonumber \\
& \quad +\underbrace{Ct^{\mhatl}(|\del^{m+1}\Rm|^2+|\del^{m+2}T|^2)}_{\text{A}}+\underbrace{Ct^{\mhatl}|\del^{m+2}T|^2}_{\text{A}} \nonumber \\
& \quad +\underbrace{Ct^{m_l}|T||\del^{m+2}\Rm||\del^{m+2}T|}_{\text{C}}+\underbrace{Ct^{\frac{m_l}{2}}|\del^{m+2}T|}_{\text{B}}+\underbrace{Ct^{\mhatl}|\del^{m+2}T||\del^{m+1}\Rm|}_{\text{A}} \nonumber \\
& \quad +m_lt^{m_l-1}(|\del^{m+1}\Rm|^2+|\del^{m+2}T|^2). \label{eq:modshiaux6}
\end{align}

We apply Young's inequality to the last $\text{A}$ term above to get
\begin{align*}
Ct^{\mhatl}|\del^{m+2}T||\del^{m+1}\Rm|\leq Ct^{\mhatl}(|\del^{m+1}\Rm|^2+|\del^{m+2}T|^2)
\end{align*} 
and hence all the terms with $\text{A}$ underneath the braces in \eqref{eq:modshiaux6} combine to give
\begin{align*}
\underbrace{\cdots}_{\text{A}}\leq Ct^{\mhatl}(|\del^{m+1}\Rm|^2+|\del^{m+2}T|^2).
\end{align*}
We apply the Young's inequality to $\text{B}$ terms in \eqref{eq:modshiaux6} to get
\begin{align*}
\underbrace{\cdots}_{\text{B}}=Ct^{\frac{m_l}{2}}(|\del^{m+1}\Rm|+|\del^{m+2}T|)&=Ct^{{\frac{m_l-1}{2}}+\frac 12}(|\del^{m+1}\Rm|+|\del^{m+2}T|) \\
&\leq Ct^{\mhatl}(|\del^{m+1}\Rm|^2+|\del^{m+2}T|^2)+t^{m_l-\mhatl},
\end{align*}
where we got $t^{m_l-\mhatl}$ term precisely because of the relation between $m_l$ and $\mhatl$. For the $\text{C}$ terms in \eqref{eq:modshiaux6}, we again apply Young's inequality
\begin{align*}
Ct^{m_l}|T||\del^{m+1}\Rm||\del^{m+3}T|&\leq Ct^{m_l}\left(\frac{1}{\epsilon}|\del^{m+1}\Rm|^2+\epsilon |\del^{m+3}T|^2\right) \\
Ct^{m_l}|T||\del^{m+2}\Rm||\del^{m+2}T|&\leq Ct^{m_l}\left(\epsilon|\del^{m+2}\Rm|^2+\frac{1}{\epsilon}|\del^{m+2}T|^2\right)
\end{align*}
to finally get
\begin{align*}
\underbrace{\cdots}_{\text{C}}\leq Ct^{m_l}\left(\frac{1}{\epsilon}|\del^{m+1}\Rm|^2+\frac{1}{\epsilon}|\del^{m+2}T|^2+\epsilon|\del^{m+2}\Rm|^2+\epsilon |\del^{m+3}T|^2\right).
\end{align*}

We use all these estimates in \eqref{eq:modshiaux6} to get
\begin{align}
\ptt \left(t^{m_l}(|\del^{m+1}\Rm|^2+|\del^{m+2}T|^2)\right)&\leq \Delta(t^{m_l}(|\del^{m+1}\Rm|^2+|\del^{m+2}T|^2))-(2-C\epsilon)t^{m_l}(|\del^{m+2}\Rm|^2+|\del^{m+3}T|^2)\nonumber \\	
& \quad + Ct^{\mhatl}(|\del^{m+1}\Rm|^2+|\del^{m+2}T|^2)+t^{m_l-\mhatl}. \label{eq:lemevolm}
\end{align}

Using $k=m$ in \eqref{eq:modshieaux1} and following the same steps we get
\begin{align}
\ptt \left(t^{\mhatl}(|\del^{m}\Rm|^2+|\del^{m+1}T|^2)\right)&\leq \Delta(t^{\mhatl}(|\del^{m}\Rm|^2+|\del^{m+1}T|^2))-(2-C\epsilon)t^{\mhatl}(|\del^{m+1}\Rm|^2+|\del^{m+2}T|^2)\nonumber \\	
& \quad + \mhatl t^{\mhatl-1}(|\del^{m}\Rm|^2+|\del^{m+1}T|^2)+C. \label{eq:lemevolmminusone}
\end{align}

Recalling the definition of the functions $F_{m,l}$ from \eqref{eq:Fmldefn} and using \eqref{eq:lemevolm}, \eqref{eq:lemevolmminusone}, we have
\begin{align*}
\left(\ptt - \Delta \right)F_{m,l}=\left(\ptt-\Delta\right)\left(\tilde{C}+t^{\mhatl}(|\del^m\Rm|^2+|\del^{m+1}T|^2)\right)t^{m_l}(|\del^{m+1}\Rm|^2+|\del^{m+2}T|^2)	\leq
\end{align*}
\begin{align}
&\left(\tilde{C}+t^{\mhatl}(|\del^m\Rm|^2+|\del^{m+1}T|^2)\right)\left(-\frac 32 t^{m_l}(|\del^{m+2}\Rm|^2+|\del^{m+3}T|^2) + Ct^{\mhatl}(|\del^{m+1}\Rm|^2+|\del^{m+2}T|^2)+C\right)\nonumber \\
& +\left(-\frac 32t^{\mhatl}(|\del^{m+1}\Rm|^2+|\del^{m+2}T|^2)	
 + \mhatl t^{\mhatl-1}(|\del^{m}\Rm|^2+|\del^{m+1}T|^2)+C \right)t^{m_l}(|\del^{m+1}\Rm|^2+|\del^{m+2}T|^2) \nonumber \\
 & -2t^{m_l+\mhatl}\del(|\del^m\Rm|^2+|\del^{m+1}T|^2)\del (|\del^{m+1}\Rm|^2+|\del^{m+2}T|^2), \label{eq:modshiaux7}
\end{align}
where we chose $\epsilon$ so that $2-C\epsilon < \frac 32$ and used the fact that for $t\in [0,\tau], t^{m_l}-t^{\mhatl}\leq C$ for some constant. Recall that we have the independence to chose the constant $\tilde{C}$ in the definition of $F_{m,l}$. We choose $\tilde{C}$ such that 
\begin{align*}
7t^{\mhatl}(|\del^{m}\Rm|^2+|\del^{m+1}T|^2)\leq \tilde{C},
\end{align*}
which we can do because of our induction hypothesis for $j=m$. We continue our estimates with the right hand side of \eqref{eq:modshiaux7} being
\begin{align}
&\leq -12t^{m_l+\mhatl}(|\del^m\Rm|^2+|\del^{m+1}T|^2)(|\del^{m+2}\Rm|^2+|\del^{m+3}T|^2) \nonumber \\
& \quad +\left(\tilde{C}+t^{\mhatl}(|\del^m\Rm|^2+|\del^{m+1}T|^2)\right)\left( Ct^{\mhatl}(|\del^{m+1}\Rm|^2+|\del^{m+2}T|^2)+C\right) \nonumber \\
&\quad -\frac 32t^{m_l+\mhatl}(|\del^{m+1}\Rm|^2+|\del^{m+2}T|^2)^2 \nonumber \\
& \quad +\left(\mhatl t^{\mhatl-1}(|\del^{m}\Rm|^2+|\del^{m+1}T|^2)+C \right)(t^{m_l}(|\del^{m+1}\Rm|^2+|\del^{m+2}T|^2))\nonumber \\
& \quad +8t^{m_l+\mhatl}(|\del^m\Rm||\del^{m+1}\Rm|+|\del^{m+1}T||\del^{m+2}T|)(|\del^{m+1}\Rm||\del^{m+2}\Rm|+|\del^{m+2}T||\del^{m+3}T|).\label{eq:modshiaux8}
\end{align}
Applying Young's inequality to the last term in \eqref{eq:modshiaux8} as
\begin{align*}
&\frac{8\sqrt{24}}{\sqrt{24}}t^{m_l+\mhatl}(|\del^m\Rm||\del^{m+1}\Rm|+|\del^{m+1}T||\del^{m+2}T|)(|\del^{m+1}\Rm||\del^{m+2}\Rm|+|\del^{m+2}T||\del^{m+3}T|)\leq \\
&\qquad 12 t^{m_l+\mhatl}\left((|\del^m\Rm|^2+|\del^{m+1}T|^2)(|\del^{m+2}\Rm|^2+|\del^{m+3}T|^2)\right)+\frac 43 t^{m_l+\mhatl}(|\del^{m+1}\Rm|^2+|\del^{m+2}T|^2)^2,
\end{align*}
implies that \eqref{eq:modshiaux8} becomes
\begin{align}
&\leq -\frac 16t^{m_l+\mhatl}(|\del^{m+1}\Rm|^2+|\del^{m+2}T|^2)^2 + \frac 87\tilde{C}\left( Ct^{\mhatl}(|\del^{m+1}\Rm|^2+|\del^{m+2}T|^2)+C\right) \nonumber \\
& \quad \quad +\left(\frac 17\tilde{C}\mhatl t^{m_l-1}+Ct^{m_l}\right)(|\del^{m+1}\Rm|^2+|\del^{m+2}T|^2) \nonumber \\
&\leq  -\frac 16t^{m_l+\mhatl}(|\del^{m+1}\Rm|^2+|\del^{m+2}T|^2)^2 + \frac 87\tilde{C}C + \left(\left(\frac 87\tilde{C}C+\frac 17\tilde{C}\mhatl\right)t^{\mhatl}+Ct^{m_l}\right)(|\del^{m+1}\Rm|^2+|\del^{m+2}T|^2)  \nonumber \\
& \leq  -\frac 16t^{m_l+\mhatl}(|\del^{m+1}\Rm|^2+|\del^{m+2}T|^2)^2 + \frac 87\tilde{C}C +Ct^{\mhatl}(|\del^{m+1}\Rm|^2+|\del^{m+2}T|^2) \nonumber \\
& \leq  -\frac{1}{12}t^{m_l+\mhatl}(|\del^{m+1}\Rm|^2+|\del^{m+2}T|^2)^2 + \frac 87\tilde{C}C +Ct^{\mhatl-m_l}, \label{eq:modshiaux9}
\end{align}
where we used $7t^{\mhatl}(|\del^{m}\Rm|^2+|\del^{m+1}T|^2)\leq \tilde{C}$ in the first and the second inequality, the fact that for $t\in [0, \tau], t^{m_l}\leq Ct^{\mhatl}$ in going from the second to the third inequality and Young's inequality on the last term in the third inequality to obtain the final inequality. Recalling that $t^{\sgn(m_l)}=t^{m_l-\mhatl}$ and hence $t^{m_l+\mhatl}=t^{2m_l}\cdot t^{-\sgn(m_l)}$, \eqref{eq:modshiaux9} and \eqref{eq:modshiaux7} finally give
\begin{align}
\left(\ptt -\Delta\right)F_{m,l}&\leq   -\frac{1}{12}t^{m_l+\mhatl}(|\del^{m+1}\Rm|^2+|\del^{m+2}T|^2)^2 + \frac 87\tilde{C}C +Ct^{\mhatl-m_l}\nonumber \\
&\leq -\frac{c}{12t^{\sgn(m_l)}}\left(t^{m_l}(|\del^{m+1}\Rm|^2+|\del^{m+2}T|^2)^2 \right) + \dfrac{C}{t^{\sgn(m_l)}} \nonumber \\
& \leq  -\frac{c}{12t^{\sgn(m_l)}}\left(F_{m,l}\right)^2 + \dfrac{C}{t^{\sgn(m_l)}}, \label{eq:modshiauxfinal}
\end{align}
where we again used the induction hypothesis for $j=m$ and the condition for $\tilde{C}$. This proves the Lemma. \qed

\section{Long time existence, and Compactness theorem}\label{sec:ltecompact}
\subsection{Criteria for long time existence}\label{subsec:lte1}

As is the case for the Ricci flow, we have the following characterization of the maximal existence time of a solution of the Ricci-harmonic flow.

\begin{theorem}\label{thm:lte1}
Let $\g2(t)$ be a solution of the Ricci-harmonic flow on a closed $7$-manifold $M$ on a maximal time interval $[0, \tau)$. Then
\begin{align}\label{eq:lte1}
\lim_{t\nearrow \tau} \Lambda(t)=\infty. 
\end{align}	
Moreover, the quantity $\Lambda(t)$ blows-up at the following rate,
\begin{align}\label{eq:Lambdablowuprate}
\Lambda(t)\geq \dfrac{C}{\tau-t},
\end{align}
where $C>0$ is a constant.
\end{theorem}

\begin{proof}
As is standard in these type of arguments, the proof is by contradiction. Suppose $\g2(t)$ is a solution of the Ricci-harmonic flow such that \eqref{eq:lte1} does not hold. In other words, there exists a constant $K>0$ such that 
\begin{align}\label{eq:lteaux1}
\underset{[0, \tau]}{\text{sup}}\ \Lambda(t) = \underset{M\times [0, \tau]}{\text{sup}}\ \left(|\Rm(x,t)|^2+|\del T(x,t)|^2+|T(x,t)|^4\right)^{\frac 12} \leq K.
\end{align}

We will prove that in this case the solution can be extended past $\tau$, thus contradicting the \emph{maximality} of $\tau$. Since \eqref{eq:lteaux1} holds, we have
\begin{align*}
\underset{M\times [0,\tau]}{\text{sup}}\ |\Rm(x,t)|+|T(x,t)|^2\leq K,\ \ \ \text{and}\ \ \ |\del T(x,t)|\leq K
\end{align*}
which implies that
\begin{align*}
\underset{M\times [0,\tau]}{\text{sup}}\ \left|\ptt g_{ij}\right|_{g(t)} = \underset{M\times [0,\tau]}{\text{sup}}\ \left|-2R_{ij}+6T_{pi}T_{pj}-2|T|^2g_{ij} \right|_{g(t)} \leq CK,
\end{align*}
for some constant $C$. As a result, all the metrics in the family $g(t)_{t\in [0,\tau]}$ are uniformly equivalent (see \cite[Thm. 14.1]{hamilton-3manifolds} for the argument in the Ricci flow case). In fact, we see from the flow equation \eqref{heatfloweqn} and the definition of the $\diamond$ operator \eqref{eq:diadefn}, that
\begin{align}
\left| \ptt \g2\right|_{g(t)} = \left|\left(-\Ric+3 T^tT -|T|^2g \right) \diamond \g2 + \Div T\lrcorner \psi \right|_{g(t)} \leq CK.
\end{align}

The idea now is to get hold of a limiting $\G2$-structure $\g2(\tau)$ as $t\nearrow \tau$ and show that $\g2(t)\rightarrow \g2(\tau)$ smoothly. Since all the metrics $g(t)$ are uniformly equivalent, for any $0<t_1<t_2<\tau$ we have,
\begin{align}\label{eq:lteaux2}
|\g2(t_2)-\g2(t_1)|_{g(0)} \leq \int_{t_1}^{t_2} \left| \ptt \g2\right|_{g(0)} dt\leq CK(t_2-t_1), 
\end{align}
and hence, $\g2(t)$ converges to a $3$-form $\g2(\tau)$ continuously as $t\rightarrow \tau$. Similarly, the metrics $g(t)\rightarrow g(\tau)$ and the volume forms $\vol_{g(t)} \rightarrow \vol_{g(\tau)}$ continuously as $t\rightarrow \tau$. Since $\g2(t)$ is a $\G2$-structures, we know that for all $t\in [0, \tau)$ we have
\begin{align}\label{eq:lteaux3}
g_t(u,v)\vol_{g(t)}= -\frac 16 (u\lrcorner \g2(t))\wedge (v\lrcorner \g2(t))\wedge \g2(t).
\end{align}
Since the left hand side of \eqref{eq:lteaux3} converges to a positive definite $7$-form valued bilinear form, the right hand side must also do so. As a consequence, the continuous limit $\g2(\tau$) is a positive $3$-form and hence is a $\G2$-structure. Thus, if \eqref{eq:lteaux1} holds, then the solution $\g2(t)_{t\in [0,\tau)}$ of the Ricci-harmonic flow can be extended continuously to the time interval $[0, \tau]$.

We now show that the above extension is actually smooth. Let $\del(0)=\overline{\del}$ denote the Levi-Civita connection of the metric $g(0)$. We first prove the following claim.

\begin{claim}\label{lteclaim}
For all $m\in \mathbb{N}$, there exist constants $C_m$ and $C'_m$ such that 
\begin{align*}
\underset{M \times [0,\tau)}{\text{sup}}\ \left|\overline{\del}^mg(t) \right|_{g(0)}\leq C_m\ \ \ \text{and}\ \ \ \  \underset{M \times [0,\tau)}{\text{sup}}\ \left|\overline{\del}^m\g2(t) \right|_{g(0)}\leq C'_m.
\end{align*}	
\end{claim}

\noindent
\emph{Proof of Claim~\ref{lteclaim}}. We just prove the claim for the derivative of $\g2$ as the case for the metrics is similar. The proof is by induction on $m$. For $m=1$ case, let $(x,t)\in M\times [0,\tau)$ be any point. Then
\begin{align*}
\ptt \overline{\del} \g2 = \overline{\del} \ptt \g2& = \overline{\del} \left((-\Ric + 3T^tT-|T|^2g)\diamond \g2 + \Div T\lrcorner \psi\right)\\
&=\del \left((-\Ric + 3T^tT-|T|^2g)\diamond \g2 + \Div T\lrcorner \psi\right)+ A*\left((-\Ric + 3T^tT-|T|^2g)\diamond \g2 + \Div T\lrcorner \psi\right)
\end{align*} 
where denotes the tensor $A=\overline{\del}-\del$. In a fixed chart around the point $x$, we have
\begin{align*}
\ptt \tensor{A}{^k_i_j}&=-\ptt \tensor{\Gamma}{^k_i_j}\\
&=-\frac 12 g^{kl}\left(\del_i(\ptt g_{jl})+\del_j(\ptt g_{il})-\del_l(\ptt g_{ij})\right)\\
\end{align*}
and so
\begin{align*}
\ptt A = \del (\Ric+T*T).
\end{align*}

We integrate the above in time, to get
\begin{align}
|A(t)|_{g(0)} &\leq C|A(0)|_{g}+ C\int_{0}^{t} \left|\frac{\pt}{\pt s} A\right|_{g}ds \nonumber \\
&\leq  C|A(0)|_{g}+ C(|\del \Ric|+|\del T||T|)t \leq C, \label{eq:lteaux4}
\end{align}
where we used the fact that $t<\tau$ is finite and because of the assumption \eqref{eq:lteaux1}, we have the bounds \eqref{eq:shiestconc} and hence the integrand is bounded.

\noindent
Since $\del \left((-\Ric + 3T^tT-|T|^2g)\diamond \g2 + \Div T\lrcorner \psi\right)$ is bounded because of \eqref{eq:shiestconc}, using \eqref{eq:lteaux4}, we get that 
\begin{align*}
\left|\ptt \overline{\del} \g2\right|_{g(0)}\leq C,
\end{align*}
and hence
\begin{align}
|\overline{\del} \g2|_{g(0)}\leq |\overline{\del} \g2(0)|_{g(0)}+\int_{0}^{t}\left| \frac{\pt}{\pt  s}\overline{\del} \g2(s)\right|_{g(0)}ds |\overline{\del} \g2(0)|_{g(0)}+C\tau \leq C_1,
\end{align}
where we again used that  $\tau < \infty$. This proves the claim for $m=1$. 

We note that following the same reasoning which is used to obtain \eqref{eq:lteaux4} and the bounds \eqref{eq:shiestconc} from the Shi-type estimates, we get
\begin{align*}
|\pt_t \overline{\del} ^kA(t)|_{g(0)} = |\overline{\del}^k \pt_t A(t)|_{g(0)} = |\overline{\del}^k(g^{-1}(\Ric + T*T))|_{g(0)}
\end{align*}
and hence by the induction hypothesis, we have
\begin{align}\label{eq:lteaux5}
|\overline{\del}^k A(t)|_{g(0)}\leq C\ \ \ \text{for}\ 0\leq k\leq m-1.
\end{align}

We now prove the induction step. We have
\begin{align}
\left|\ptt \overline{\del}^m \g2\right|_{g(0)} &=|\overline{\del}^m\left((-\Ric + 3T^tT-|T|^2g)\diamond \g2 + \Div T\lrcorner \psi\right)|_{g(0)} \nonumber \\
&\leq C \sum_{i=0}^{m}|A|^i|\del^{m-i}\left((-\Ric + 3T^tT-|T|^2g)\diamond \g2 + \Div T\lrcorner \psi\right)| \nonumber \\
& \quad +C\sum_{i=1}^{m-1}|\overline{\del}^iA||\del^{m-1-i}\left((-\Ric + 3T^tT-|T|^2g)\diamond \g2 + \Div T\lrcorner \psi\right)|. \label{eq:lteaux6}
\end{align}
The  terms in the first quantity on the right hand side of \eqref{eq:lteaux6} are bounded due to the compactness of $M$ and the Shi-type estimates \eqref{eq:shiestconc}. The terms in the second quantity are bounded due to \eqref{eq:lteaux5} and the Shi-type estimates. Thus we get,
\begin{align}
\left|\ptt \overline{\del}^m\g2\right|_{g(0)}\leq C. \label{eq:lteaux7}
\end{align}
The Claim~\ref{lteclaim} follows from integrating \eqref{eq:lteaux7} in time and using the fact that $\tau<\infty$.

\medskip

We now continue the proof of Theorem~\ref{thm:lte1}. We have the continuous limit $\G2$-structure $\g2(\tau)$. Let $U$ be the domain of a fixed local coordinate chart. We have
\begin{align}
\g2_{ijk}(\tau)=\g2_{ijk}(t) + \int_{t}^{\tau} \left((-\Ric(s) + 3T^tT(s)-|T|^2g(s))\diamond \g2(s) + \Div T(s)\lrcorner \psi (s)\right)_{ijk}ds. \label{eq:lteaux8}
\end{align}
Let $\alpha=\{a_1, \ldots a_r\}$ be any multi-index with $|\alpha|=a_1+a_2+\cdots + a_r=m\in \mathbb{N}$. We know from Claim~\ref{lteclaim} and~\eqref{eq:lteaux7} that
\begin{equation*}
	\frac{\pt^m}{\pt x^{\alpha}}\g2_{ijk} \qquad \text{and} \qquad \frac{\pt^m}{\pt x^{\alpha}}\left((-\Ric + 3T^tT-|T|^2g)\diamond \g2 + \Div T\lrcorner \psi \right)_{ijk}
\end{equation*}
are uniformly bounded on $U\times [0, \tau)$. So from~\eqref{eq:lteaux8} we have that $\frac{\pt^m}{\pt x^{\alpha}}\g2_{ijk}(\tau)$ is bounded on $U$ and hence $\g2(\tau)$ is a smooth $\G2$-structure. Moreover, from~\eqref{eq:lteaux8} we have
\begin{equation*}
	\Big | \frac{\pt^m}{\pt x^{\alpha}}\g2_{ijk}(\tau)-\frac{\pt^m}{\pt x^{\alpha}}\g2_{ijk}(t) \Big | \leq C(\tau-t)
\end{equation*}
and thus $\g2(t)\rightarrow \g2(\tau)$ uniformly in any $C^m$ norm as $t\rightarrow \tau$, for $m\geq 2$.

Now we use the short-time existence and uniqueness result to obtain a contradiction. Since $\g2(\tau)$ is smooth, we use $\bar{\g2}(0)=\g2(\tau)$ as our initial condition in Theorem~\ref{thm:RHFste} to obtain a solution $\bar{\g2}(t)$ of the Ricci-harmonic flow for a short time $0\leq t < \varepsilon$. Since $\g2(t)\rightarrow \g2(\tau)$ smoothly as $t\rightarrow \tau$, it follows that
\begin{equation*}
	\bar{\g2}(t) = 
	\begin{cases}
		\g2(t) & 0\leq t < \tau \\
		\bar{\g2}(t-\tau) & \tau \leq t < \tau + \varepsilon
	\end{cases}
\end{equation*}
is a solution of the Ricci-harmonic flow which is smooth and satisfies $\bar{\g2}(0)=\g2(0)$. This contradicts the maximality of $\tau$ and the contradiction is arising because of our assumption \eqref{eq:lteaux1}. Thus we indeed have
\begin{equation} \label{eq:lteaux9}
	\lim_{t\nearrow \tau} \Lambda(t)=\infty,
\end{equation}
and thus, if $\lim_{t\nearrow \tau}  \Lambda(t)$ exists, it must be $\infty$.

Next we show that in fact~\eqref{eq:lte1} is true. Suppose not. Then there exists $K_0< \infty$ and a sequence of times $t_i \nearrow \tau$ such that $\Lambda(t_i) \leq K_0$. By the doubling time estimate in Lemma~\ref{lem:dte}, we get that
\begin{equation*}
	\Lambda(t)\leq 2\Lambda(t_i)\leq 2K_0
\end{equation*}
for all times $t\in [t_i, \min\{ \tau, t_i+\frac{1}{C K_0} \}]$. Since $t_i\nearrow \tau$ as $i\rightarrow \infty$, there exists $i_0$ large enough such that $t_{i_0}+\frac{C}{K_0} \geq \tau$. (Here, as before, we crucially use the fact that $\tau$ is assumed to be finite.) But this implies that
\begin{equation*}
	\sup_{M\times [t_{i_0}, \tau]} \Lambda(x,t) \leq 2K_0
\end{equation*} 
which cannot happen as we have already shown above that this leads to a contradiction to the maximality of $\tau$. This completes the proof of~\eqref{eq:lte1}.

We complete the proof of the theorem by computing the lower bound of the blow-up rate~\eqref{eq:Lambdablowuprate}. We apply the maximum principle to~\eqref{eq:dteaux5} to get
\begin{equation*}
	\frac{d}{dt} \Lambda(t)^2 \leq C\Lambda(t)^3
\end{equation*}
which implies that
\begin{equation} \label{eq:lteaux10}
	\frac{d}{dt} \Lambda(t)^{-1} \geq -\frac C2.
\end{equation}
Since we proved above that $\lim_{t\rightarrow \tau} \Lambda(t)=\infty$, we have
\begin{equation*}
	\lim_{t\rightarrow \tau} \Lambda(t)^{-1} =0.
\end{equation*}
Integrating~\eqref{eq:lteaux10} from $t$ to $t_0\in (t, \tau)$, we have
\begin{align*}
\Lambda(t_0)^{-1}-\Lambda(t)^{-1}\geq -\frac C2(t_0-t)
\end{align*}
and hence, taking the limit as $t_0\rightarrow \tau$ gives
\begin{equation*}
	\Lambda(t)\geq \cfrac{2}{{C(\tau -t)}}.
\end{equation*}
This completes the proof of Theorem~\ref{thm:lte1}.
\end{proof}


Since we obtained the blow-up rate of $\Lambda(t)$ in \eqref{eq:Lambdablowuprate}, we can make the following definitions of various type of singularities of \eqref{heatfloweqn}, just like the Ricci flow case.

\begin{definition}\label{def:sing.types}
Suppose that $(M^7,\varphi(t))$ is a solution of the Ricci-harmonic flow of $\G2$-structures on a closed manifold on a maximal time interval $[0,\tau)$ and let $\Lambda(t)$ be as in \eqref{eq:Lambdadefn}.
	
	\medskip
	
\noindent	
If we have a finite-time singularity, i.e.~$\tau<\infty$, we say that the solution forms  
	\begin{itemize}
		\item a \emph{Type I singularity} (rapidly forming) if $\sup_{t\in[0,\tau)}(\tau-t)\Lambda(t) <\infty$; and otherwise 
		\item a \emph{Type IIa singularity} (slowly forming) if $\sup_{t\in [0,\tau)}(\tau-t)\Lambda(t)=\infty$.
	\end{itemize}
	
\noindent	
If we have an \emph{infinite-time} singularity, where $\tau=\infty$, then it is 
	\begin{itemize}
		\item a \emph{Type IIb singularity} (slowly forming) if $\sup_{t\in  [0,\infty)}t\Lambda(t)=\infty$; and otherwise  
		
		\item a \emph{Type III singularity} (rapidly forming) if $\sup_{t\in  [0,\infty)}t\Lambda(t)<\infty$.
	\end{itemize}
\end{definition}


Since Type I singularities play a crucial role in understanding the behaviour of the singularities of the Ricci flow, it would be interesting to understand the long time behaviour of Type I Ricci-harmonic flows and the structure of Type I singularities. 

\subsection{Compactness}

In this section, we prove a Cheeger--Gromov--Hamilton type compactness theorem for solutions of the Ricci-harmonic flow. 

Recall that a sequence $(M^7_i, \g2_i, p_i), p_i\in M_i$ of complete, pointed $7$-manifolds with $\G2$-structures is said to converge to $(M^7, \g2, p)$ with $p\in M$ and $\g2$ a $\G2$-structure, if there exists a sequence of compact subsets $\Omega_i\subset M$ exhausting $M$ with $p\in \text{int}(\Omega_i)$, for all $i$ and a sequence of diffeomorphisms $F_i:\Omega_i\rightarrow F_i(\Omega_i)\subset M_i$ with $F(p)=p_i$ such that for every $\varepsilon >0$, there exists $k_0=k_0(\varepsilon)$ such that for all $k\geq k_0$,
\begin{align}\label{eq:cmpstrdefn}
\underset{0\leq \alpha \leq p}{\text{sup}}\ \underset{x\in K}{\text{sup}}\ |\del^{\alpha}(F_{i}^*\g2_i-\g2)|_{g} < \varepsilon,
\end{align}
for all $p\in \mathbb{N}$ and for any compact subset $K$ of $M$. Here $g$ is any Riemannian metric on $M$ and $\del$ is its Levi-Civita connection.

We also recall the following very general compactness theorem for $\G2$-structures proved by Lotay--Wei  in~\cite[Theorem 7.1]{lotay-wei-gafa}.

\begin{theorem} \label{thm:strccompactnessthm}
	Let $M_i$ be a sequence of smooth $7$-manifolds and for each $i$ we let $p_i\in M_i$ and $\g2_i$ be a $\G2$-structure on $M_i$ such that the metric $g_i$ on $M_i$ induced by $\g2_{i}$ is complete on $M_i$. Suppose that 
	\begin{equation} \label{strccompactnessthmeqn}
		\sup_i \sup_{x\in M_i} \big (|\del_{g_i}^{k+1}T_{i}(x)|^2_{g_i}+|\del^k_{g_i} \Rm_{g_i}(x)|^2_{g_i} \big )^{\frac 12} < \infty
	\end{equation}
	for all $k\geq 0$ and
	\begin{equation*}
		\inf_i \inj (M_i,g_i,p_i)>0,
	\end{equation*}
	where $T_i$, $\Rm_{g_i}$ are the torsion and the Riemann curvature tensor of $\g2_i$ and $g_i$ respectively and $\inj (M_i, g_i, p_i)$ denotes the injectivity radius of $(M_i,g_i)$ at $p_i$.
	
	Then there exists a $7$-manifold $M$, a $\G2$-structure $\g2$ on $M$ and a point $p\in M$ such that, after passing to a subsequence, we have 
	\begin{equation*}
		(M_i, \g2_i, p_i) \rightarrow (M, \g2, p) \qquad \text{as } i\rightarrow \infty.
	\end{equation*} 
\end{theorem}

Following the same ideas as in the Ricci flow case by Hamilton \cite{hamilton-compactness} and for the Laplacian flow of closed $\G2$-structures by Lotay--Wei \cite[Thm. 7.2]{lotay-wei-gafa}, we prove the following compactness theorem for the solutions of the Ricci-harmonic flow.

\begin{theorem}\label{thm:cmptthmflows}
Let $M_i$ be a sequence of compact $7$-manifolds and let $p_i\in M_i$ for each $i$. Let $\g2_i(t)$ be a sequence of solutions to the Ricci-harmonic flow~\eqref{heatfloweqn} for $\G2$-structures on $M_i$ for $t\in (a,b)$, where $-\infty \leq a<0<b\leq \infty$. Suppose that
\begin{equation} \label{eq:cmpctthm1}
	\sup_i \sup_{x\in M_i, t\in (a,b)} \left(|\del_{g_i(t)}T_{i}(x,t)|^2_{g_i(t)}+|\Rm_{g_i(t)}(x,t)|^2_{g_i(t)}+|T_i(x,t)|^4_{g_i(t)}\right)^{\frac 12} < \infty
\end{equation}
where $\Rm_i(t)$ and $T_i(t)$ denote the Riemann curvature tensor and the torsion of $\g2_i(t)$ respectively, and the injectivity radius satisfies
\begin{equation} \label{eq:cmpctthm2}
	\inf_i \inj(M_i, g_i(0), p_i)>0.
\end{equation}

Then there exists a $7$-manifold $M$, a point $p\in M$ and a solution $\g2(t)$ of the flow~\eqref{heatfloweqn} on $M$ for $t\in (a,b)$ such that, after passing to a subsequence, 
\begin{equation*}
	(M_i, \g2_i(t), p_i) \rightarrow (M, \g2(t), p) \qquad \text{as }\  i\rightarrow \infty.
\end{equation*}
\end{theorem}
\begin{proof}
Since the proof is similar to the proof of compactness theorem for other flows of $\G2$-structures, we only sketch the main ideas. From the Shi-type estimates in Theorem~\ref{thm:shiest} and the assumption \eqref{eq:cmpctthm1}, we get
\begin{align}\label{eq:cptthmaux1}
|\del^k_{g_i(t)} \Rm_i(x,t)|_{g_i(t)} + |\del^{k+1}_{g_i(t)} T_i(x,t)|_{g_i(t)} \leq C_k.
\end{align}
Using the assumption \eqref{eq:cmpctthm2} on the injectivity radius, we can apply Theorem~\ref{thm:strccompactnessthm} to extract a subsequence of $(M_i, \g2_i(0), p_i)$ which converges to a complete limit $(M, \tilde{\g2}_{\infty}(0), p)$ where $p\in M$ and $\tilde{\g2}_{\infty}(0)$ is a $\G2$-structure on $M$. We use the diffeomorphisms $F_i$ obtained from applying Theorem~\ref{thm:strccompactnessthm} above to define 
\begin{align*}
\tilde{\varphi}_i(t)=F_i^*\g2_i(t). 
\end{align*}

Using the Claim~\ref{lteclaim} and in fact, essentially the same arguments in the proof of Theorem~\ref{thm:lte1}, we get that for any compact subset $\Omega\times [c,d] \subset M\times (a,b)$, there exists constants $C_{k,l}$ such that
\begin{align}
\underset{\Omega\times [c,d]}{\text{sup}}\ \left( \left|\frac{\pt^l}{\pt t^l} \del^k_{\tilde{g}_{\infty}(0)}\tilde{g}_i(t)\right|_{\tilde{g}_{\infty}(0)} +\left|\frac{\pt^l}{\pt t^l} \del^k_{\tilde{g}_{\infty}(0)}\tilde{\g2}_i(t)\right|_{\tilde{g}_{\infty}(0)}    \right) \leq C_{k,l}.
\end{align}

Thus, using the Arzelá-Ascoli theorem and a standard diagonalization argument, we get that there exists a subsequence of $\tilde{\g2}_i(t)$ that converges smoothly on any compact subset of $M\times (a,b)$ to a solution $\tilde{\g2}_{\infty}(t)$ of the Ricci-harmonic flow.
\end{proof}

One of the main applications of a compactness theorem for solution of any geometric flow is in singularity analysis. More precisely, suppose $M^7$ is a compact manifold and let $\g2(t)$ be a solution to the Ricci-harmonic flow on a maximal time interval $[0, \tau)$ with $\tau<\infty$. Theorem~\ref{thm:lte1} then implies that $\Lambda(t)$ defined in~\eqref{eq:Lambdadefn} satisfies $\lim_{t \nearrow \tau} \Lambda(t)=\infty$. Choose a sequence of points $(x_i, t_i)$ with $t_i\nearrow \tau$ and
\begin{equation*}
	\Lambda(x_i,t_i) = \sup_{x\in M,\ t\in [0,t_i]} \left(|\del T(x,t)|^2_{g(t)}+|\Rm(x,t)|^2_{g(t)}+|T(x,t)|^4_{g(t)}\right)^{\frac{1}{2}}.
\end{equation*}

Now consider a sequence of parabolic dilations of the Ricci-harmonic flow 
\begin{equation}
	\g2_i(t) = \Lambda(x_i, t_i)^3\g2(t_i+\Lambda(x_i,t_i)^{-2}t)
\end{equation}
and define
\begin{equation}
	\Lambda_{\g2_i}(x,t)=  \left(|\del T_i(x,t)|^2_{g_i(t)}+|\Rm_i(x,t)|^2_{g_i(t)}+|T_i(x,t)|^4_{g_i(t)}\right)^{\frac{1}{2}}.
\end{equation}
If $\widetilde{\g2} = c^3\g2$ then we have,
\begin{equation*}
	(\widetilde{\Ric}+3\tilde{T}^t\tilde{T}-|\tilde{T}|^2_{\tilde{g}}\tilde{g})\ \tilde{\diamond}\  \tilde{\g2}+ \widetilde{\Div} \tilde{T} \lrcorner \tilde{\psi} = c \left((\Ric+3T^tT-|T|^2_{g}g)\ \diamond\  \g2+ \Div  T \lrcorner \psi\right).
\end{equation*}
Hence, for each $i$, we have that $(M, \g2_i(t))$ is a solution of the Ricci-harmonic flow~\eqref{heatfloweqn} on the time interval $[-t_i\Lambda(x_i,t_i), (\tau-t_i)\Lambda(x_i,t_i)$. Note that for each $i$ and for all $t\leq 0$ we have
\begin{equation*}
	\sup_{M} |\Lambda_{\g2_i}(x,t)| = \frac{\left(|\del T_i(x,t)|^2_{g_i(t)}+|\Rm_i(x,t)|^2_{g_i(t)}+|T_i(x,t)|^4_{g_i(t)}\right)^{\frac{1}{2}}}{\Lambda(x_i,t_i)} \leq 1
\end{equation*}
by the definition of $\Lambda(x_i,t_i)$. Thus, there exists a uniform $b>0$ such that
\begin{equation*}
	\sup_{i} \sup_{M\times (a,b)} |\Lambda_{\g2_i}(x,t)| \leq 2
\end{equation*}
for any $a<0$. Thus, if we have $\inf_{i} \inj (M, g_i(0), x_i)>0$, then using the compactness Theorem~\ref{thm:cmptthmflows}, we can extract a subsequence of $(M, \g2_i(t), x_i)$ that converges to a solution $(M_{\infty}, \g2_{\infty}(t), x_{\infty})$ of the Ricci-harmonic flow.  

In fact, using our modified Shi-type estimates Theorem~\ref{thm:strlocder} we can prove the following modified compactness theorem for the Ricci-harmonic flow whose proof we skip as it is similar to the proof of Theorem~\ref{thm:cmptthmflows}.

\begin{theorem}\label{thm:modcomp}
Let $M_i$ be a sequence of compact $7$-manifolds and let $p_i\in M_i$ for each $i$. Let $\g2_i(t)$ be a sequence of solutions to the Ricci-harmonic flow~\eqref{heatfloweqn} for $\G2$-structures on $M_i$ for $t\in (a,b)$, where $-\infty \leq a<0<b\leq \infty$. Suppose that
\begin{enumerate}
	\item (derivative bounds for the initial $\G2$-structure) for some $m_0\in \mathbb{N}\cup \{\infty\}$ we have
	\begin{align}
	\underset{k\in \mathbb{N}}{\text{sup}} \ \left(\underset{x\in M_k}{\text{sup}}\ \left(|\del^m \Rm_{\g2_k(0)}(x)|+ |\del^{m+1}T_{\g2_k(0)}(x)|\right)\right) < \infty
	\end{align}
	for $0\leq m \leq m_0+1$.
	
	\item (space-time $\Lambda$-bound) we have
	\begin{equation} \label{eq:strcmpctthm1}
		\sup_i \sup_{x\in M_i, t\in (a,b)} \left(|\del_{g_i(t)}T_{i}(x,t)|^2_{g_i(t)}+|\Rm_{g_i(t)}(x,t)|^2_{g_i(t)}+|T_i(x,t)|^4_{g_i(t)}\right)^{\frac 12} < \infty.
	\end{equation}
	
	\item (injectivity radius satisfies) $\inf_i \inj(M_i, g_i(0), p_i)>0$.
	\end{enumerate}
Then there exists a $7$-manifold $M$, a point $p\in M$ and a solution $\g2(t)$ of the flow~\eqref{heatfloweqn} on $M$ for $t\in (a,b)$ such that, after passing to a subsequence, 
\begin{equation*}
	(M_i, \g2_i(t), p_i) \rightarrow (M, \g2(t), p) \qquad \text{as }\  i\rightarrow \infty.
\end{equation*}	
in the pointed $C^{m_0}$-Cheeger-Gromov topology. \qed
\end{theorem}

\subsection{Long time existence for the Ricci-harmonic flow}\label{subsec:ltevelocity}
We recall that a family of metrics $g(t)$, $t\in [0, \tau)$ is called {\bf{uniformly continuous}} if for any $\varepsilon >0$ there exists a $\delta >0$ such that for any $0\leq t_0<t<\tau$, if $t\-t_0\leq \delta$ then
\begin{align*}
|g(t)-g(t_0)|\leq \varepsilon,
\end{align*}
which implies
\begin{align}\label{eq:unicontmet}
(1-\varepsilon)g(t_0)\leq g(t)\leq (1+\varepsilon)g(t_0).	
\end{align}
Clearly, \eqref{eq:unicontmet} implies that for points $x, y\in M$ if $t-t_0\leq \delta$ then the distance function satisfies
\begin{align*}
(1-\varepsilon)^{\frac 12}d_{g(t_0)}(x,y)\leq d_{g(t)}(x,y)\leq (1+\varepsilon)^{\frac 12}d_{g(t_0)}(x,y), 
\end{align*}
and hence the geodesic balls centred at $x$ with radius $r$ satisfy
\begin{align*}
B_{g(t_0)}\left(x, \frac{r}{(1+\varepsilon)^{\frac 12}}\right) \subset B_{g(t)}(x,r)
\end{align*}
with their volumes satisfying
\begin{align}\label{eq:volbounds}
\Vol_{g(t_0)}\left(B_{g(t_0)}\left(x, \frac{r}{(1+\varepsilon)^{\frac 12}}\right) \right)\leq \Vol_{g(t)}\left(B_{g(t)}(x,r)\right).
\end{align}

In this section, we improve our long time existence result Theorem~\ref{thm:lte1} by proving that the Ricci-harmonic flow starting with an arbitrary $\G2$-structure on a compact manifold exists as long as the velocity of the flow remains bounded. This will obtained as a corollary of the following theorem which proves that if the underlying metrics $g(t)$ are uniformly continuous along the Ricci-harmonic flow and if the scalar curvature, the torsion tensor and the divergence of the intrinsic torsion $\Vop T$ satisfy the boundedness assumption then the Ricci-harmonic flow exists for all time. We state and prove the theorem. The proof is based on similar results in the Ricci flow case by Šešum \cite{sesum} (we follow the proof from \cite[Theorem 6.40]{chow-lu-ni}) and the Laplacian flow of closed $\G2$-structures case by Lotay--Wei \cite[Theorem 8.1]{lotay-wei-gafa}. 

\begin{theorem}\label{thm:ltestrong}
Let $(M^7, \g2(t))$ be a Ricci-harmonic flow on a compact manifold for $t \in [0, \tau)$ with $\tau < \infty$. Let $g(t)$ be the underlying metric of $\g2(t)$. If $g(t)$ is uniformly continuous and we have the following bounds
\begin{align}\label{eq:ltestrhyp}
\underset{M\times [0,\tau)}{\text{sup}} R+4|T|^2+2 \Div (\Vop T) < \infty,
\end{align}
where $\Vop T$ is the intrinsic torsion $T_7$ identified with a vector field, then the solution $\g2(t)$ can be extended past $\tau$.
\end{theorem}

\begin{proof}
The proof is by contradiction. Suppose \eqref{eq:ltestrhyp} holds but the flow cannot be extended beyond time $\tau$. It follows from Theorem~\ref{thm:lte1} that there exists a sequence of points $(x_i, t_i)$ in space-time with $t_i\nearrow \tau$ such that
\begin{align}\label{eq:ltestraux1}
\Lambda(x_i, t_i)= \underset{x\in M,\ t\in[0,t_i]}{\text{sup}} \left(|\Rm(x,t)|^2+|\del T(x,t)|^2+|T(x,t)|^4\right)^{\frac 12}\rightarrow \infty.
\end{align}
It follows from the discussion after the proof of Theorem~\ref{thm:cmptthmflows}, and in the same way as in \cite[Thm. 8.1]{lotay-wei-gafa} by Lotay--Wei, that we can get a sequence of flows $(M, \g2_i(t), x_i)$ defined on $[-t_i\Lambda(x_i,t_i), 0]$ with 
\begin{align*}
\underset{M\times [-t_i\Lambda(x_i,t_i), 0]}{\text{sup}} |\Lambda_{\g2_i}(x,t)|\leq 1\ \ \ \text{and}\ \ \ |\Lambda_{\g2_i}(x_i,0)|=1.
\end{align*}
We now use the uniform continuity of the metrics $g(t)$ and \eqref{eq:volbounds} on the volume bounds of the geodesic balls imply that if $r\leq \Lambda(x_i,t_i)^{\frac 12}$ then for some uniform positive constant $c$, we have
\begin{align*}
\Vol_{g_i(0)}(B_{g_i(0)}(x,r))\geq C(1+\varepsilon)^{-\frac 72}r^7.
\end{align*}
This allows us to apply a theorem of Cheeger--Gromov--Taylor \cite[Theorem 5.42]{chow-lu-ni} to get uniform injectivity radius lower bound away from zero $\inj (M, g_i(0), x_i) \leq \gamma >0$. Using our compactness theorem, Theorem~\ref{thm:cmptthmflows}, we get a subsequence converging to a limit $(M_{\infty}, \g2_{\infty}, x_{\infty})$, $t\in (-\infty, 0]$ and $|\Lambda_{\g2_{\infty}}(x_{\infty}, 0)|=1$.

Since our assumption \eqref{eq:ltestrhyp} is that $R+4|T|^2+2\Div (\Vop T)$ remains bounded along the flow and $\Lambda(x_i, t_i)\rightarrow \infty$ as $i\rightarrow \infty$, we get by scalings along $\G2$-structures that
\begin{align}
(R_i+4|T_i|^2+2\Div_i (\Vop T_i))_{g_i(t)}(x,t)=\Lambda(x_i, t_i)^{-1}\left[(R+4|T|^2+2\Div (\Vop T))_{g(t_i+\Lambda(x_i,t_i)^{-1}t)}(x,t_i+\Lambda(x_i,t_i)^{-1}t)\right] \rightarrow 0,\label{eq:ltestraux2}
\end{align}
as $i\rightarrow \infty$. Using the expression for the scalar curvature \eqref{eq:scalarcurv}, we see that $(M_{\infty}, \g2_{\infty}(t))$ has all its intrinsic torsion zero and hence is torsion free for all $tin (-\infty, 0]$. Since the metrics of torsion-free $\G2$-structures are Ricci-flat, we have $\Ric_{g_{\infty}(t)}=0$ for all $t\in (-\infty, 0]$. 

Arguing again as in \cite[Thm. 8.1]{lotay-wei-gafa}, which in turn is based on the argument in \cite{sesum}, we know that $g_{\infty}(0)$ has Euclidean volume growth. Thus, from the Bishop--Gromov volume comparison theorem, $(M_{\infty}, g_{\infty})$ must be isometric to $(\bR^7, g_{\text{Eucl}})$ as $\Ric_{g_{\infty}}=0$. But this is a contradiction as by our point-picking argument we have $|\Lambda_{\g2_{\infty}}(x_{\infty}, 0)|=1$ and hence
\begin{align*}
|\Rm_{g_{\infty}}(x_{\infty}, 0)|= |\Lambda_{\g2_{\infty}}(x_{\infty}, 0)|=1.
\end{align*}
This completes the proof of the theorem.
\end{proof}

\begin{remark}
The assumption of uniform continuity of the metrics $g(t)$ was used to get volume bounds as in \eqref{eq:volbounds} which in turn was used to get the uniform positivity of the injectivity radius so that we can get the limit manifold having Euclidean volume growth. One can drop the uniformly continuous $g(t)$ assumption if there is another way to guarantee uniform positivity of the injectivity radius, for instance, by way of a $\kappa$-non-collapsing theorem. This was done for reasonable flows of $\G2$-structures (which the Ricci-harmonic flow is) by Gao Chen under the assumption of uniformly bounded torsion tensor, \cite[Thm. 5.2, Thm. 5.4]{gaochen-shi}. Thus, another way of proving long time existence results for the Ricci-harmonic flow is by using the results of Chen.\demo
\end{remark}

As an application of the previous theorem, we get
\begin{theorem}\label{thm:velbound}
Let $(M, \g2(t))$ be a Ricci-harmonic flow on a compact manifold with $t\in [0, \tau)$ and $\tau < \infty$. if the velocity of the flow in \eqref{heatfloweqn} is bounded then the solution can be extended past time $\tau$.
\end{theorem}

\begin{proof}
If the velocity of \eqref{heatfloweqn} is bounded, that is,
\begin{align}\label{eq:velaux1}
\underset{M \times [0, \tau)}{\text{sup}}|\Ric -3T^T+|T|^2g|_{g(t)} + |\Div T|_{g(t)} < \infty,
\end{align}
then the metrics $g(t)$ are uniformly bounded along the flow. Moreover, \eqref{eq:velaux1} also implies that 
\begin{align*}
\underset{M \times [0, \tau)}{\text{sup}} (R+4|T|^2+2\Div (\Vop T)) <\infty. 	
\end{align*}	
Thus, both the conditions of Theorem~\ref{thm:ltestrong} are satisfied and we can extend the solution beyond the time $\tau$.
\end{proof}

\section{Ricci-harmonic Solitons}\label{sec:solitons}

In this section, we develop the theory of solitons for the Ricci-harmonic flow. These are special solutions of the flow which are self-similar, i.e., they move only by diffeomorphisms and scalings of a given $\G2$-structure. Ricci solitons play a very important role in Hamilton--Perelman's program on the Ricci flow and the proof of Thurston's geometrization conjecture. They appear as fixed points of various new functionals (for instance $\mathcal{F}$ and $\mathcal{W}$-functional) introduced by Perelman and as a result they appear as singularity models for the Ricci flow. Since we proved a compactness theorem for the Ricci-harmonic flow Theorem~\ref{thm:cmptthmflows} and explained how it can be used to analyze a singular point, it can be expected that the singularities of the Ricci-harmonic might be modelled on self-similar solutions. With a viewpoint towards this expectation, we prove various results and identities for the solitons of the Ricci-harmonic flow.

\medskip

Let $M$ be a $7$-manifold. A \emph{soliton} for the Ricci-harmonic flow \eqref{ricci-harmonic flow} is a triple $(\g2, Y, \lambda)$ with $Y\in \Gamma(TM)$ and $\lambda\in \bR$ such that
\begin{align}\label{sol1}
	\left(-\Ric+3 T^tT-|T|^2g\right)\diamond \g2 +  \Div T\lrcorner \psi =\lambda \g2+\cL_Y\g2.
\end{align}

The $\G2$-structures which satisfy \eqref{sol1} are special as they give rise to self-similar solutions of the flow. More precisely, if $\g2_0$ satisfies \eqref{sol1} on $M$ for some vector field $Y$ and scalar $\lambda$ then for all $t$ such that $1+\frac 23 \lambda t>0$, define
\begin{align}\label{eq:selfsimilar1}
\rho(t)=\left(1+\frac 23\lambda t\right)^{\frac 32}\ \ \ \text{and}\ \ \ Y(t)=\rho(t)^{-\frac 23}Y.
\end{align}
The powers are chosen because $\tilde{g}=fg$ if $\tilde{\g2}=f^{\frac 32}\g2$. Let $\Theta(t)$ be the diffeomorphisms generated by the vector fields $X(t)$ starting from the identity map. Define
\begin{align}\label{eq:selfsimilar2}
\g2(t) = \rho(t)\Theta(t)^*\g2_0
\end{align}
which is a self-similar to $\g2_0$ as it moves $\g2_0$ only by scalings and diffeomorphisms. Differentiating \eqref{eq:selfsimilar2} with respect to time and using the fact that $\g2_0$ satisfies \eqref{sol1}, we immediately get
\begin{align*}
\ptt \g2(t) = \left(-\Ric+3 T^tT-|T|^2g\right)(t)\diamond_t \g2(t) +  \Div T(t)\lrcorner \psi(t), 
\end{align*}
and hence we get the equivalence between \eqref{sol1}, \eqref{eq:selfsimilar1} and \eqref{eq:selfsimilar2}.


\medskip

We say a soliton $(\g2, Y, \lambda)$ is expanding if $\lambda>0$; steady if $\lambda=0$; and shrinking if $\lambda <0$.

\medskip

We now derive the condition satisfied by the metric $g$ induced by $\g2$ and $\Div T$ when $(\g2, Y, \lambda)$ is a soliton, which we expect to have further use.

\medskip

\begin{proposition}\label{prop:gsoliton}
	Let $(\g2, Y, \lambda)$ be a solitons as defined in \eqref{sol1}. Then the induced metric $g$ satisfies
	\begin{align}\label{gsoliton1}
		R_{ij}-3T_{pi}T_{pj}+|T|^2g_{ij}+\frac12(\cL_Yg)_{ij}+\frac{\lambda}{3}g_{ij}=0,
	\end{align}
	and $\Div T$ satisfies
	\begin{align}\label{gsoliton2}
		\Div T + \frac 12 \curl Y-Y\lrcorner T=0.
	\end{align}
\end{proposition}
\begin{proof}
	The definition of the $\diamond$ operator in \eqref{eq:diadefn} implies $\g2=\left(\dfrac g3 \diamond \g2 \right)$. Moreover, \eqref{eq:liederivativephi} gives 
	\begin{align*}
	\cL_Y \g2&= \frac 12 \cL_Yg \diamond \g2 + \left(-\frac 12 \curl Y + Y\lrcorner T\right)\lrcorner \psi.
	\end{align*}

Using these expressions in \eqref{sol1} and using the fact that $\Omega^3_{1+ 27}$ is orthogonal to $\Omega^3_7$ gives \eqref{gsoliton1} and \eqref{gsoliton2}. 
\end{proof}

\medskip

We can prove the following non-existence theorem for compact expanding solitons of \eqref{ricci-harmonic flow}.

\begin{proposition}\label{prop:solmainprop}
	Let $(M^7, \g2, Y, \lambda)$ be a Ricci-harmonic soliton. We have the following.
	\begin{enumerate}
		\item There are no compact expanding solitons of \eqref{ricci-harmonic flow}.
		\item The only compact steady solitons of \eqref{ricci-harmonic flow} are given by torsion-free $\G2$-structures.
	\end{enumerate}
\end{proposition}
\begin{proof}
	Taking the trace of \eqref{gsoliton1} gives
	\begin{align*}
		R+4|T|^2+ \Div Y + \frac 73 \lambda=0
	\end{align*}
	which on using the expression for the scalar curvature \eqref{eq:scalarcurv} simplifies to
	\begin{align}\label{gsol3}
	10|T_1|^2+9|T_7|^2+3|T_{14}|^2+3|T_{27}|^2-2\Div(\Vop T)+\frac 73\lambda + \Div Y=0.
	\end{align}
	Integrating \eqref{gsol3} on compact $M$ gives
	\begin{align*}
		\int_M \left(10|T_1|^2+9|T_7|^2+3|T_{14}|^2+3|T_{27}|^2\right)\vol + \frac 73 \lambda \text{Vol}(M)=0.
	\end{align*}
	So $\lambda \leq 0$ and $\lambda=0$ if and only if $T_i\equiv 0$ for $i=1, 7, 14, 27$ and hence $\g2$ must be torsion-free.
\end{proof}

\medskip

We now derive some identities for \emph{gradient} Ricci-harmonic solitons. A Ricci-harmonic soliton (RH soliton, for short) $(M, \g2, Y, \lambda)$ is gradient if the vector field $Y=\del f$ for some $f\in C^{\infty}(M)$. The metric and the torsion of a gradient Ricci-harmonic soliton satisfy
\begin{align}
R_{ij}-3T_{pi}T_{pj}+|T|^2g_{ij}+ \del_i\del_jf + \frac{\lambda}{3}g_{ij}&=0, \label{eq:gradsolmet} \\
\Div T_j  - \del_ifT_{ij}&=0. \label{eq:gradsolT}
\end{align}
Both these equations are obtained by putting $Y=\del f$ in Proposition~\ref{prop:gsoliton} with the fact that $\curl (\del f)_k=\del_i\del_jf\g2_{ijk}=0$ for any $f\in C^{\infty}(M)$.

\begin{lemma}\label{lem:soliden}
Let $(M^7, \g2, \del f, \lambda)$ be a gradient RH soliton. Then the following identities hold.
\begin{align}
R+4|T|^2+\Delta f + \frac 73 \lambda &=0. \label{eq:soliden1}\\
\del_j\left(R+6|T|^2+|\del f|^2+\frac{2\lambda}{3}f\right) &= 6T_{pj}T_{pm}\del _mf-6\del_i(T_{pi}T_{pj})-2|T|^2\del_jf. \label{eq:soliden2}
\end{align}
\end{lemma}

\begin{remark}
The equation \eqref{eq:soliden2} can be viewed as the analog of Hamilton's identity \cite[\textsection 20]{hamilton-singularities}
\begin{align*}
\del(R+|\del f|^2+\lambda  f)=0	
\end{align*}
for gradient Ricci solitons. \demo
\end{remark}

\begin{proof}
Taking the trace of \eqref{eq:gradsolmet} gives
\begin{align*}
R+4|T|^2+\Delta f + \frac 73 \lambda =0
\end{align*}
which proves \eqref{eq:soliden1}.

We take the divergence of \eqref{eq:gradsolmet} and use the twice contracted second Bianchi identity and the Ricci identity \eqref{eq:ricciiden}, to get
\begin{align}\label{eq:auxeq}
	\begin{split}
0&=\del_i\left(R_{ij}-3T_{pi}T_{pj}+|T|^2g_{ij}+ \del_i\del_jf + \frac{\lambda}{3}g_{ij}\right) \\
&= \frac 12\del_jR-3\del_i(T_{pi}T_{pj})+\del_j|T|^2+\del_j\Delta f+R_{jm}\del_mf \\
&= \frac 12\del_jR -3\del_i(T_{pi}T_{pj})+\del_j|T|^2+\del_j\left( -R-4|T|^2-\frac 73\lambda \right )+R_{jm}\del_mf\\
&=  \del_jR +6\del_j|T|^2+6\del_i(T_{pi}T_{pj})-2R_{jm}\del_mf, 
\end{split}
\end{align}
where we used \eqref{eq:soliden1}. We now use the soliton equation \eqref{eq:gradsolmet} again for the $R_{jm}$ term above in the last equality to further obtain
\begin{align*}
0&= \del_jR +6\del_j|T|^2+6\del_i(T_{pi}T_{pj})-2\left(3T_{pj}T_{pm}-|T|^2g_{jm}-\del_j\del_mf - \frac{\lambda}{3}g_{jm}\right)\del_mf \\
&= \del_jR + 6\del_j|T|^2+6\del_i(T_{pi}T_{pj})-6T_{pj}T_{pm}\del_mf +2|T|^2\del_jf +\del_j |\del f|^2+\frac{2\lambda}{3}\del_jf,
\end{align*}
which can be re-written as
\begin{align*}
\del_j\left(R+6|T|^2+|\del f|^2+\frac{2\lambda}{3}f\right) = 6T_{pj}T_{pm}\del _mf-6\del_i(T_{pi}T_{pj})-2|T|^2\del_jf,
\end{align*}
which proves \eqref{eq:soliden2}.
\end{proof}

\medskip

We now state and prove some identities for Ricci-harmonic solitons which, we believe, will be useful for studying further geometric and analytic properties of the former. We first recall the following

\begin{lemma}\label{lem:pw}
\emph{\cite[Lemma 2.1]{petersen-wylie2}.}
Let $X$ be a vector field on a Riemannian manifold $(M^n,g)$. Then
\begin{align}
\Div(\cL_Xg)(X)=\frac 12 \Delta |X|^2-|\del X|^2+\Ric(X,X)+\del_X \Div X. \label{eq:pw1}	
\end{align}	
When $X=\del f$ and $Z\in \Gamma(TM)$, then
\begin{align}\label{eq:pw2}
\Div(\cL_{\del f}g)(Z)=2\Ric(Z, \del f)+2\del_{Z} \Div \del f.	
\end{align}
\end{lemma}

We use the preceding lemma to prove the following.

\begin{lemma}\label{lem:pwsol}
Let $(M^7, \g2, X, \lambda)$ be a Ricci-harmonic soliton. Then 
\begin{align}\label{eq:pwsol1}
\frac 12 \Delta|X|^2&=|\del X|^2-\Ric (X,X)+2\del_X|T|^2+6\Div (T^t\circ T)(X),
\end{align}
and 
\begin{align}\label{eq:pwsol2}
\frac 12(\Delta-\del_X)|X|^2&=|\del X|^2+\frac{\lambda}{3}|X|^2+|T|^2|X|^2-3(T^t\circ T)(X,X) + 2\del_X|T|^2+6\Div (T^t\circ T)(X).
\end{align}
\end{lemma}

\begin{proof}
Taking the divergence of \eqref{gsoliton1} gives
\begin{align}\label{eq:pwsolaux1}
2\Div \Ric + \Div(\cL_Xg)+2 \del |T|^2 -6\Div (T^t\circ T)=0.
\end{align}
Moreover, since 
\begin{align*}
R+4|T|^2+ \Div X+\frac 73\lambda =0,
\end{align*}
we get
\begin{align}\label{eq:pwsolaux2}
\del_XR + 4\del_X|T|^2+\del_X(\Div X)=0.
\end{align}
Using \eqref{eq:pw1}, \eqref{eq:pwsolaux1}, \eqref{eq:pwsolaux2} and the twice contracted second Bianchi identity, we get
\begin{align*}
\del_X(\Div X)&=-\del_XR - 4\del_X|T|^2 \\
&= -2\Div \Ric(X) -4\del_X|T|^2\\
&= \Div (\cL_Xg)(X)+2\del_X|T|^2-6\Div(T^t\circ T)(X)-4\del_X|T|^2	\\
&=\frac 12 \Delta |X|^2-|\del X|^2+\Ric(X,X)+\del_X \Div X-2\del_X|T|^2-6\Div (T^t\circ T)(X)
\end{align*}
which gives
\begin{align*}
\frac 12 \Delta|X|^2&=|\del X|^2-\Ric (X,X)+2\del_X|T|^2+6\Div (T^t\circ T)(X),
\end{align*}
thus proving \eqref{eq:pwsol1}.

Again, using the soliton equation \eqref{gsoliton1}, we have
\begin{align*}
\Ric(X,X)=-\frac 12 (\cL_Xg)(X,X)-\frac{\lambda}{3}|X|^2-|T|^2|X|^2+3(T^t\circ T)(X,X),
\end{align*}
which on inserting in \eqref{eq:pwsol1}, implies
\begin{align*}
\frac 12 \Delta |X|^2&=|\del X|^2+\frac 12(\cL_Xg)(X,X)+\frac{\lambda}{3}|X|^2+|T|^2|X|^2-3(T^t\circ T)(X,X) + 2\del_X|T|^2+6\Div (T^t\circ T)(X) \\
&= |\del X|^2+\frac 12\del_X|X|^2+\frac{\lambda}{3}|X|^2+|T|^2|X|^2-3(T^t\circ T)(X,X) + 2\del_X|T|^2+6\Div (T^t\circ T)(X),
\end{align*}
which proves \eqref{eq:pwsol2}.

If a Ricci-harmonic soliton is compact then integrating \eqref{eq:pwsol1} from the previous lemma gives the following simple corollary.

\begin{corollary}\label{cor:soltrivial}
Let $(M^7, \g2, X, \lambda)$ be a compact Ricci-harmonic soliton. If
\begin{align}
\int_M \left(\Ric (X,X)-2\del_X|T|^2-6\Div (T^t\circ T)(X)\right) \vol \leq 0
\end{align}
then $X$ is a Killing field and the soliton is  trivial.
\end{corollary}

\end{proof}

\medskip
We also have the following integral formula for compact Ricci-harmonic gradient soliton.

\begin{lemma}\label{lem:intformsol}
Let $(M^7, \g2, \del f, \lambda)$ be a compact gradient Ricci-harmonic soliton. Then
\begin{align}\label{eq:intsol1}
\left.\int _M\right. \left|\del^2f-\frac{\Delta f}{7}g \right|^2 \vol &= \int_M \left[\frac{5}{14} \langle \del R, \del f\rangle +6T_{pj}T_{pm}\del_j\del_mf - 6\del_j\del_i(T_{pi}T_{pj})  -\frac 67|T|^2\Delta f \right] \vol.
\end{align}
\end{lemma}

\begin{proof}
We take the divergence of \eqref{eq:soliden2} to get
\begin{align*}
\Delta R+6\Delta |T|^2+\Delta |\del f|^2+\frac{2\lambda}{3}\Delta f&=6\del_j(T_{pj}T_{pm})\del_mf + 6T_{pj}T_{pm}\del_j\del_mf - 6\del_j\del_i(T_{pi}T_{pj})-2\langle \del |T|^2, \del f\rangle \nonumber \\
& \quad -2|T|^2\Delta f,
\end{align*}
which on using the Bochner formula
\begin{align*}
\Delta |\del f|^2=2\langle \del (\Delta f), \del f\rangle + 2\Ric(\del f, \del f)+2|\del^2f|^2,
\end{align*}
becomes
\begin{align}
\Delta R+6\Delta |T|^2+2\langle \del (\Delta f), \del f\rangle + 2\Ric(\del f, \del f)+2|\del^2f|^2+\frac{2\lambda}{3}\Delta f&=6\del_j(T_{pj}T_{pm})\del_mf + 6T_{pj}T_{pm}\del_j\del_mf \nonumber \\
& \quad - 6\del_j\del_i(T_{pi}T_{pj})-2\langle \del |T|^2, \del f\rangle \nonumber \\
& \quad -2|T|^2\Delta f. \label{eq:intsolaux1}
\end{align}

Using the second equality in \eqref{eq:auxeq} for the third and the fourth term on the right hand side of \eqref{eq:intsolaux1} makes the right hand side
\begin{align*}
	\Delta R + 6\Delta |T|^2 - \langle \del R, \del f\rangle+6\del_i(T_{pi}T_{pj})\del_jf - 2\langle \del |T|^2, \del f\rangle + 2|\del^2f|^2+\frac{2\lambda}{3} \Delta f
\end{align*}
and hence \eqref{eq:intsolaux1} becomes
\begin{align}
	\Delta R + 6\Delta |T|^2 - \langle \del R, \del f\rangle + 2|\del^2f|^2+\frac{2\lambda}{3} \Delta f&= 6T_{pj}T_{pm}\del_j\del_mf - 6\del_j\del_i(T_{pi}T_{pj}) -2|T|^2\Delta f. \label{eq:intsolaux2}
\end{align}

Since $\left|\del^2f-\frac{\Delta f}{n}g \right|^2=|\del^2f|^2-\frac{(\Delta f)^2}{n}$, \eqref{eq:intsolaux2} becomes,
\begin{align*}
\Delta R + 6\Delta |T|^2 + 2\left|\del^2f-\frac{\Delta f}{7}g \right|^2&= \langle \del R, \del f\rangle -\frac{2\lambda}{3}\Delta f -2\frac{(\Delta f)^2}{7}+6T_{pj}T_{pm}\del_j\del_mf - 6\del_j\del_i(T_{pi}T_{pj}) \\
& \quad  -2|T|^2\Delta f \\
&=\langle \del R, \del f\rangle +2\Delta f \left(-\frac{\lambda}{3}-\frac{\Delta f}{7}\right)+6T_{pj}T_{pm}\del_j\del_mf - 6\del_j\del_i(T_{pi}T_{pj}) \\
& \quad  -2|T|^2\Delta f \\
&= \langle \del R, \del f\rangle +\frac 27\Delta f \left(R+4|T|^2\right)+6T_{pj}T_{pm}\del_j\del_mf - 6\del_j\del_i(T_{pi}T_{pj}) \\
& \quad  -2|T|^2\Delta f \\
&=\langle \del R, \del f\rangle +\frac 27R \Delta f +6T_{pj}T_{pm}\del_j\del_mf - 6\del_j\del_i(T_{pi}T_{pj})  -\frac 67|T|^2\Delta f .
\end{align*}

Integrating by parts on compact $M$, we get
\begin{align*}
\left.\int _M\right. \left|\del^2f-\frac{\Delta f}{7}g \right|^2 \vol &= \int_M \left[\frac{5}{14} \langle \del R, \del f\rangle +6T_{pj}T_{pm}\del_j\del_mf - 6\del_j\del_i(T_{pi}T_{pj})  -\frac 67|T|^2\Delta f \right] \vol ,
\end{align*}
which is precisely \eqref{eq:intsol1}.
\end{proof}

We end this section by deriving the expression of the drift Laplacian acting on the scalar curvature of a gradient Ricci-harmonic soliton. The motivation for us is the importance of the drift Laplacian operator in various rigidity and classification results for gradient Ricci solitons. We have the following lemma.

\begin{lemma}\label{lem:driftR}
Let $(M, \g2, \lambda, Y)$ be a Ricci-harmonic soliton. Then the drift Laplacian of the scalar curvature $R$ is given by 
\begin{align}\label{eq:driftR}
\Delta_YR=6\Delta |T|^2 + 6 \del_j  \del_i (T_{pi}T_{pj}) + 2|\Ric|^2 - 6 R_{ij}T_{pi}T_{pj} + 2|T|^2R - \frac 23\lambda R.
\end{align}
\end{lemma}

\begin{proof}
If $\rho(t)$ and $Y(t)$ are as defined in \eqref{eq:selfsimilar1}, $\Theta(t)$ is the family of diffeomorphisms on $M$ generated by $Y(t)$ starting from the identity map and $\g2(t) = \rho(t)\Theta^*\g2$ is a Ricci-harmonic flow with initial condition $\g2$ then the underlying metric is $g(t)=\left(1+\frac 23\lambda t\right)\Theta^*(t)g$. Hence, the scalar curvature is 
\begin{align}\label{eq:driftR1}
R(t) = \left(1+\frac 23\lambda t \right)^{-1}\Theta^*R.
\end{align}
Differentiating \eqref{eq:driftR1} with respect to $t$ and setting $t=0$ while also using the evolution of the scalar curvature along the Ricci-harmonic flow \eqref{eq:RHFscalevol}, we get
\begin{align*}
\frac 23\lambda R+\cL_Y R = \ptt R(t)\Big|_{0}=\Delta R +  6\Delta |T|^2 + 6 \del_j  \del_i (T_{pi}T_{pj}) + 2|\Ric|^2 - 6 R_{ij}T_{pi}T_{pj} + 2|T|^2R
\end{align*}
and hence
\begin{align*}
\Delta R - \langle \del R, Y\rangle =  6\Delta |T|^2 + 6 \del_j  \del_i (T_{pi}T_{pj}) + 2|\Ric|^2 - 6 R_{ij}T_{pi}T_{pj} + 2|T|^2R - \frac 23\lambda R,
\end{align*}
which proves \eqref{eq:driftR}.
\end{proof}

\section{Ricci-harmonic flow for Spin(7)-structures}\label{sec-rhfspin7}
We  derive the Taylor series expansion of a Spin(7)-structure on an $8$-manifold $M^8$ and look at the "heat equation" for Spin(7)-structures which will give us a suitable coupling of the Ricci flow of Spin(7)-structures and the harmonic flow of Spin(7)-structures which was proposed in \cite{dle-isometric}.

\subsection{Taylor series expansion of Spin(7)-structures}\label{subsec:texpspin7}
Let $(M^8, \Phi)$ be a manifold with a Spin(7)-structure. For $x\in M$, we choose our local orthonormal frame $\{e_1,\ldots,e_8\}$ of $T_xM$ to be Spin(7)-\emph{adapted} which means that at the point $x$, the components $\Phi_{ijkl}$ agree with those of the standard flat model on $\bR^8$. We now state and prove the Taylor series expansion of a Spin(7)-structure which has not appeared in the literature before and is analogous to the result for $\G2$-structures Theorem~\ref{thm:taylor-g2}.

\begin{theorem}\label{thm:taylorspin7}
Let $(x^1,\ldots, x^8)$ be Spin(7)-adapted Riemannian normal coordinates centred at $x\in M$. The components $\Phi_{ijkl}$ of $\Phi$ have Taylor expansions about $0\in \bR^8$ corresponding to $x\in M^8$, given by
\begin{align}\label{eq:taylorspin7}
\Phi_{ijkl}(x^1,\ldots, x^8)&= \Phi_{ijkl}+\left(T_{q;im}\s7_{mjkl}+T_{q;jm}\s7_{imkl}+T_{q;kp}\s7_{ijml}+T_{q;lm}\s7_{ijkm}  \right)x^q+ Q_{pqijkl}x^px^q+O(||x||^3),
\end{align}
where 
\begin{align}\label{eq:quadtaylor}
	Q_{pqijkl}&=\frac 12 \left(\del_pT_{q;im}\Phi_{mjkl}+\del_pT_{q;jm}\Phi_{imkl}+\del_pT_{q;km}\Phi_{ijml}+\del_pT_{q;lm}\Phi_{ijkm}\right)  \nonumber \\
	& \qquad + \frac 12 T_{q;im}(T_{p;ms}\s7_{sjkl}+T_{p;js}\s7_{mskl}+T_{p;ks}\s7_{mjsl}+T_{p;ls}\s7_{mjks})\nonumber \\
	&\qquad + \frac 12 T_{q;jm}(T_{p;is}\s7_{smkl}+T_{p;ms}\s7_{iskl}+T_{p;ks}\s7_{imsl}+T_{p;ls}\s7_{imks})\nonumber \\
	& \qquad + \frac 12 T_{q;km}(T_{p;is}\s7_{sjml}+T_{p;js}\s7_{isml}+T_{p;ms}\s7_{ijsl}+T_{p;ls}\s7_{ijms}) \nonumber\\
	&\qquad + \frac 12 T_{q;lm}(T_{p;is}\s7_{sjkm}+T_{p;js}\s7_{iskm}+T_{p;ks}\s7_{ijsm}+T_{p;ms}\s7_{ijks})\nonumber \\
	& \quad +\frac{1}{6}(R_{piqm}\Phi_{mjkl}+R_{pjqm}\Phi_{imkl}+R_{pkqm}\Phi_{ijml}+R_{plqm}\Phi_{ijkm}).
\end{align}
Here all coefficient tensors on the right hand side are evaluated at $0$.
\end{theorem}

\begin{proof}
We follow the proof in \cite[Thm. 2.25]{dgk-flows}. The constant term in \eqref{eq:taylorspin7} is $\Phi_{ijkl}$ and that is precisely due to our choice of Spin(7)-adapted orthonormal frame at $x$. For the first order term in the expansion, we start  with the formula
\begin{align}\label{eq:tayloraux1}
\pt_q\Phi_{ijkl}= \del_q\Phi_{ijkl}+ \Gamma^m_{qi}\Phi_{mjkl}+\Gamma^m_{qj}\Phi_{imkl}+\Gamma^m_{qk}\Phi_{ijml}+\Gamma^m_{ql}\Phi_{ijkm}.
\end{align}
We recall that in Riemannian normal coordinates we have
\begin{align*}
g_{ij}=\delta_{ij},\ \ \ \Gamma^i_{jk}=0\ \ \ \ \text{at\ the\ point}\ x
\end{align*}
and hence, evaluating \eqref{eq:tayloraux1} at the point $x$ and using \eqref{Tdefneqn}, we get
\begin{align*}
\pt_q\Phi_{ijkl}= T_{q;im}\s7_{mjkl}+T_{q;jm}\s7_{imkl}+T_{q;km}\s7_{ijml}+T_{q;lm}\s7_{ijkm}    
\end{align*}
which is the linear term in \eqref{eq:taylorspin7}. For the second order term, we take the partial derivative of \eqref{eq:tayloraux1} to get
\begin{align}\label{eq:tayloraux2}
\pt_p\pt_q\Phi_{ijkl}&= \pt_p\del_q\Phi_{ijkl} + (\pt_p\Gamma^m_{qi})\Phi_{mjkl}+(\pt_p\Gamma^m_{qj})\Phi_{imkl}+(\pt_p\Gamma^m_{qk})\Phi_{ijml}+(\pt_p\Gamma^m_{ql})\Phi_{ijkm} + (\text{terms\ with}\ \Gamma'\text{s}).
\end{align}
The  terms with $\Gamma'$s in \eqref{eq:tayloraux2} will vanish when we evaluate at $x$. We also recall that in Riemannian normal coordinates centred at $x$, we have
\begin{align*}
\pt_i\Gamma^l_{jk}=\frac 12 \left(R^l_{ijk}+R^l_{ikj}\right)\ \ \text{at}\ x
\end{align*}
and hence the terms with partial derivatives of $\Gamma$ in \eqref{eq:tayloraux2} become
\begin{align*}
\sum_{i, j, k, l\  \text{cyclic}} \frac{1}{3}(R^m_{pqi}+R^m_{piq})\Phi_{mjkl}\ \ \ \text{at}\ x.	
\end{align*}
Similarly, we have
\begin{align*}
\pt_p\del_q\Phi_{ijkl}= \del_p\del_q\Phi_{ijkl} + (\text{terms\ with\ }\Gamma'\text{s})
\end{align*}
which at the point $x$ become
\begin{align*}
\pt_p\pt_q\Phi_{ijkl}=\del_p(T_{q;im}\s7_{mjkl}+T_{q;jm}\s7_{imkl}+T_{q;kp}\s7_{ijml}+T_{q;lm}\s7_{ijkm}  ).
\end{align*}
Overall, \eqref{eq:tayloraux2}, at the point $x$, becomes
\begin{align*}
(\pt_p\pt_q\Phi_{ijkl})\left.\right|_{0}&= \del_p(T_{q;im}\s7_{mjkl}+T_{q;jm}\s7_{imkl}+T_{q;km}\s7_{ijml}+T_{q;lm}\s7_{ijkm})+\sum_{i, j, k, l\  \text{cyclic}} \frac{1}{3}(R^m_{pqi}+R^m_{piq})\Phi_{mjkl}\\
&=\del_pT_{q;im}\Phi_{mjkl}+\del_pT_{q;jm}\Phi_{imkl}+\del_pT_{q;km}\Phi_{ijml}+\del_pT_{q;lm}\Phi_{ijkm} \\
& \qquad + T_{q;im}(T_{p;ms}\s7_{sjkl}+T_{p;js}\s7_{mskl}+T_{p;ks}\s7_{mjsl}+T_{p;ls}\s7_{mjks})\\
 &\qquad + T_{q;jm}(T_{p;is}\s7_{smkl}+T_{p;ms}\s7_{iskl}+T_{p;ks}\s7_{imsl}+T_{p;ls}\s7_{imks})\\
 & \qquad + T_{q;km}(T_{p;is}\s7_{sjml}+T_{p;js}\s7_{isml}+T_{p;ms}\s7_{ijsl}+T_{p;ls}\s7_{ijms})\\
   &\qquad + T_{q;lm}(T_{p;is}\s7_{sjkm}+T_{p;js}\s7_{iskm}+T_{p;ks}\s7_{ijsm}+T_{p;ms}\s7_{ijks})\\
   & \quad +\sum_{i, j, k, l\  \text{cyclic}} \frac{1}{3}(R^m_{pqi}+R^m_{piq})\Phi_{mjkl}.
\end{align*}
We multiply the  last  expression by $\frac 12x^px^q$ and sum over $p, q$ which causes the first curvature term to vanish because of the symmetry of the Riemann curvature tensor and we get
\begin{align*}
\frac 12 (\pt_p\pt_q\Phi_{ijkl})\left.\right|_{0} x^px^q = Q_{pqijkl}x^px^q,
\end{align*}
with the quadratic term $Q$ given by \eqref{eq:quadtaylor}. This completes the proof of the theorem.
\end{proof}

Having obtained the Taylor series expansion of $\Phi$, we now follow the same procedure in \textsection~\ref{sec:motivation} to obtain the Laplacian of the components of the Spin(7)-structure $\Phi$ at the point $x$. We contract on $p$ and $q$ in the quadratic term $Q$ in \eqref{eq:quadtaylor}. We have
\begin{align}
(\Delta \Phi)_{ijkl}&=\frac 12 (\Div T \diamond \Phi)_{ijkl} + \frac 12 \left(T_{p;is}(T_p\diamond \Phi)_{sjkl}+T_{p;js}(T_p\diamond \Phi)_{iskl}+T_{p;ks}(T_p\diamond \Phi)_{ijsl}+T_{p;ls}(T_p\diamond \Phi)_{ijks}\right) \nonumber \\
& \quad - \frac 16(\Ric \diamond \Phi)_{ijkl}.
\end{align}

We multiply the Ricci terms by $6$ (so that the evolution of the underlying metric is precisely the Ricci flow) to give the following definition.

\begin{definition}\label{defn:rhfspin7}
	Let $(M^8, \s7_0)$ be a compact manifold with a Spin(7)-structure $\s7_0$. The \emph{Ricci-harmonic flow} for the family of Spin(7)-structures $\s7(t)$ is the following initial value problem
	\begin{align} 
		\label{eq:rhfspin7} 
		\left\{\begin{array}{rl} 
			& \dfrac{\pt \s7}{\pt t} = \left(-\Ric+(T*T)+\frac 12 \Div T \right) \diamond \s7, \\
			& \s7(0) =\s7_{0},
			\tag{SRHF}
		\end{array}\right.
	\end{align}
where the term $T*T$ is a lower-order term and it is explicitly given by	
\begin{align*}
(T*T \diamond \s7)_{ijkl} &= \frac 12 \left(T_{p;is}(T_p\diamond \Phi)_{sjkl}+T_{p;js}(T_p\diamond \Phi)_{iskl}+T_{p;ks}(T_p\diamond \Phi)_{ijsl}+T_{p;ls}(T_p\diamond \Phi)_{ijks}\right).
\end{align*}	
\end{definition}

We see that the flow in \eqref{eq:rhfspin7} is different from the negative gradient flow of the $L^2$-norm of the torsion functional on the space of all Spin(7)-structures which was introduced by the author in \cite{dwivedi-spin7}.  The flow \eqref{eq:rhfspin7} however gives a natural coupling of the Ricci flow of Spin(7)-structures, i.e., the flow $\pt_t\s7(t)=-\Ric(t)\diamond_t \s7(t)$ and the harmonic flow of Spin(7)-structures which is the negative gradient flow of the $L^2$-norm of the torsion functional where the variation is being taken on the set of \emph{isometric} Spin(7)-structures and the corresponding flow is given by $\pt_t\s7(t) = \Div T\diamond_t\s7(t)$. 

The highest order (in $\Phi$) terms  in the flow \eqref{eq:rhfspin7} are $-\Ric$ and $\Div T$ and the principal symbols of these operators were computed by the author in \cite[Lemma 4.4]{dwivedi-spin7}. Using them and the discussion in \cite[\textsection 4.3]{dwivedi-spin7} shows that the Ricci-harmonic flow \eqref{eq:rhfspin7} is weakly parabolic and the only failure to parabolicity is due to the diffeomorsphism invariance of the operators $-\Ric$ and $\Div T$. As a result, we can use a similar modified DeTurck's trick from \cite[\textsection 4.4]{dwivedi-spin7} to prove the well-posedness of the flow \eqref{eq:rhfspin7}. We record this observation below.

\begin{theorem}\label{thm:srhfste}
Let $(M^8, \Phi_0)$ be a compact manifold with a Spin(7)-structure $\Phi_0$. Then there exists a unique solution $\Phi(t)$ to the Ricci-harmonic flow \eqref{eq:rhfspin7} for a short time $[0, \varepsilon]$ where $\varepsilon$ depends on $\Phi_0$.
\end{theorem}

We remark that all the results about the regularity of solutions, compactness theorem and long-time existence results obtained in \textsection~\ref{sec:evoleqns}-\ref{sec:ltecompact} go through for the Ricci-harmonic flow of Spin(7)-structures \eqref{eq:rhfspin7} with similar proofs.

\section{Questions for future}\label{sec:questions}
We compile a list of interesting future problems related to the Ricci-harmonic flow. 
\begin{question}
Can one use an Uhlenbeck-type trick and maybe some modified connection as in the Ricci flow case \cite{hamilton-4manifolds} or in the isometric flow of $\G2$-structures case \cite[\textsection 4 ]{dgk-isometric} to simplify the evolution equation \eqref{eq:Riemevol1} of the Riemann curvature tensor and the evolution equation \eqref{eq:torevol} of the torsion tensor along the flow? Such a simplification will shed light on possible applications of Hamilton's tensor maximum principle to prove preserved curvature and torsion conditions along the flow?
\end{question}

\begin{question}
Do we get interesting flows when we do dimensional reduction of the flow? For instance, what flows of $\mathrm{SU}(3)$-structures does the Ricci-harmonic flow induce when $M^7$ is of the form $M^6\times S^1$ with $M^6$ being endowed with a $\mathrm{SU}(3)$-structure. Similar questions can be asked for $\mathrm{SU}(2)$-structures on $M^4$ with $M^7=M^4\times T^3$. For instance, it would be interesting to study dimensional reduction of the flows considered in this paper following the approach of Picard--Suan \cite{picard-suan}.
\end{question}

\begin{question}
Are there other explicit solutions, for instance, shrinking solitons in the compact case and any type of soliton in the noncompact case, of the Ricci-harmonic flow?
\end{question}

\begin{question}
What is the behaviour of torsion-free $\G2$-structures as stationary points of the Ricci-harmonic flow? Are they stable stationary points or are there unstable directions in the variation? 
\end{question}

\appendix
\section{Taylor series expansion of the dual $4$-form and associated heat equation}\label{sec:appendix}
In this appendix, we compute the Taylor series expansion of the Hodge dual $\psi$ of a $\G2$-structure $\g2$ on a given oriented $7$-manifold $M$ and use the procedure described in the paper to write the "heat equation" for $4$-forms. Let $M^7$ be a smooth oriented manifold with a \emph{given orientation}. In this case, a non-degenerate $4$-form $\psi$ induces a metric $g_{\psi}$ and the prescribed orientation with $*_{\psi} \psi = \g2$. For $x\in M$, we choose our local orthonormal frame $\{e_1,\ldots,e_7\}$ of $T_xM$ to be $\G2$-\emph{adapted} which means that at the point $x$, the components $\psi_{ijkl}$ agree with those of the standard flat model on $\bR^7$. 

\begin{theorem} \label{thm:taylor-psi}
Let $(x^1, \ldots, x^7)$ be $\G2$-adapted Riemannian normal coordinates centred at $x \in M$. The components $\psi_{ijkl}$ of $\psi$ have Taylor expansions about $0$, which is the point in $\bR^7$ corresponding to $x \in M$, given by
	\begin{equation} \label{eq:taylor-psi}
		\psi_{ijkl} (x^1, \ldots, x^7) = \psi_{ijkl} + (-T_{qi}\g2_{jkl}+T_{qj}\g2_{ikl}-T_{qk}\g2_{ijl}+T_{ql}\g2_{ijk}) x^q + \prescript{\psi}{}{\! \mathcal Q_{pq \, ijkl}} x^p x^q + O(\|x\|^3),
	\end{equation}
	where
	\begin{align} \label{eq:taylor-psi-B}
\prescript{\psi}{}{\! \mathcal Q_{pq \, ijkl}} &=\frac 12(-\del_pT_{qi}\g2_{jkl}+\del_pT_{qj}\g2_{ikl}-\del_pT_{qk}\g2_{ijl}+\del_pT_{ql}\g2_{ijk}) \nonumber \\
& \quad +\frac 12(-T_{qi}T_{pm}\psi_{mjkl}+T_{qj}T_{pm}\psi_{mikl}-T_{qk}T_{pm}\psi_{mijl}+T_{ql}T_{pm}\psi_{mijk}) \nonumber \\
& \quad  +  \frac{1}{6}(R_{piqm}\psi_{mjkl}+R_{pjqm}\psi_{imkl}+R_{pkqm}\psi_{ijml}+R_{plqm}\psi_{ijkm}).
	\end{align}
Here all coefficient tensors on the right hand side are evaluated at $0$.
\end{theorem}
\begin{proof}
We follow the same ideas as in the proof of \cite[Thm. 2.25]{dgk-flows} or Theorem~\ref{thm:taylorspin7}. We start with
\begin{align*}
\pt_q\psi_{ijkl}=\del_q\psi_{ijkl}+\Gamma^{m}_{qi}\psi_{mjkl}+\Gamma^{m}_{qj}\psi_{imkl}+\Gamma^{m}_{qk}\psi_{ijml}+\Gamma^{m}_{ql}\psi_{ijkm},
\end{align*}
which on using the fact that the $\Gamma$s vanish at the point $x$ in normal coordinates and \eqref{eq:delpsi} imply
\begin{align}\label{eq:taypsi1}
\pt_q\psi_{ijkl}=-T_{qi}\g2_{jkl}+T_{qj}\g2_{ikl}-T_{qk}\g2_{ijl}+T_{ql}\g2_{ijk},
\end{align}
which gives the first order terms in \eqref{eq:taylor-psi}. In the same way as in the proof of the Taylor series expansion of the Spin(7)-structure, we have, at the point $x$,
\begin{align}
(\pt_p\pt_q \psi)_{ijkl}&=\del_p(-T_{qi}\g2_{jkl}+T_{qj}\g2_{ikl}-T_{qk}\g2_{ijl}+T_{ql}\g2_{ijk}) + \sum_{i, j, k, l\  \text{cyclic}} \frac{1}{3}(R^m_{pqi}+R^m_{piq})\psi_{mjkl}\nonumber \\
&=-\del_pT_{qi}\g2_{jkl}+\del_pT_{qj}\g2_{ikl}-\del_pT_{qk}\g2_{ijl}+\del_pT_{ql}\g2_{ijk} \nonumber \\
& \quad -T_{qi}T_{pm}\psi_{mjkl}+T_{qj}T_{pm}\psi_{mikl}-T_{qk}T_{pm}\psi_{mijl}+T_{ql}T_{pm}\psi_{mijk} + \sum_{i, j, k, l\  \text{cyclic}} \frac{1}{3}(R^m_{pqi}+R^m_{piq})\psi_{mjkl}.\label{eq:taypsi2}
\end{align}
We multiply \eqref{eq:taypsi2} by $\frac 12 x^{p}x^{q}$ and sum over $p, q$ to obtain the expression for $\mathcal{Q}$ in \eqref{eq:taylor-psi-B}.
\end{proof}

As in \textsection~\ref{sec:motivation}, we compute "$(\Delta \psi)_{ijkl}$" at the point $x$ by contracting on $p$ and $q$ in \eqref{eq:taylor-psi-B}. We get
\begin{align}
(\Delta \psi)_{ijkl}&=-\frac 16 (R_{im}\psi_{mjkl}+R_{jm}\psi_{imkl}+R_{km}\psi_{ijml}+R_{lm}\psi_{ijkm}) \nonumber \\
& \quad + \frac 12 (-T_{pi}T_{pm}\psi_{mjkl}-T_{pj}T_{pm}\psi_{imkl}-T_{pk}T_{pm}\psi_{ijml}-T_{pl}T_{pm}\psi_{ijkm}) \nonumber \\
& \quad +\frac 12 (-\Div T_i\g2_{jkl}+\Div T_j\g2_{ikl}-\Div T_k\g2_{ijl}+\Div T_l\g2_{ijk}) \nonumber \\
& = \left(-\frac 16 \Ric -\frac 12 T^tT\right)\diamond \psi -\frac 12 (\Div T \wedge \g2)_{ijkl}. \label{eq:lappsi}
\end{align}

It would be interesting to see if the flow of $4$-forms which one can define using \eqref{eq:lappsi} as the velocity has any nice properties as compared to other flows of $4$-forms. Using the expression \eqref{eq:scalarcurv} for the scalar curvature, it's easy to see that the stationary points of the flow defined using \eqref{eq:lappsi} will be exactly torsion-free $\G2$-structures.



	
	\printbibliography
	
	\noindent
	Fachbereich Mathematik, Universität Hamburg, Bundesstraße 55, 20146 Hamburg, Germany.\\
	\href{mailto:shubham.dwivedi@uni-hamburg.de}{shubham.dwivedi@uni-hamburg.de}

\end{document}